\pdfoutput=1
\documentclass[a4paper,]{amsart}

\usepackage{amssymb}

\usepackage{amsthm}

\usepackage{mathtools}

\numberwithin{equation}{section}

\usepackage[utf8]{inputenc}
\usepackage[T1]{fontenc}
\usepackage{libertine}
\usepackage[libertine,upint,smallerops]{newtxmath}

\usepackage{tikz-cd}

\usepackage{interval}

\usepackage{mathrsfs}

\usepackage[inline]{enumitem}
\setlist[enumerate]{label=\textup{(\alph*)}}

\DeclareMathOperator{\adj}{Ad}

\DeclareMathOperator{\aut}{Aut}

\DeclareMathOperator{\generallinear}{GL}
\DeclareMathOperator{\giso}{\mathcal{G}_{iso}}

\DeclareMathOperator{\id}{id}

\DeclareMathOperator{\indrep}{Ind}

\DeclareMathOperator{\irr}{Irr}
\DeclareMathOperator{\morph}{Mor}

\DeclareMathOperator{\pol}{Pol}
\DeclareMathOperator{\prglin}{PGL}

\DeclareMathOperator{\repcat}{Rep}

\DeclareMathOperator{\res}{Res}
\DeclareMathOperator{\selfmorph}{End}

\DeclareMathOperator{\supp}{supp}
\DeclareMathOperator{\tr}{Tr}

\DeclarePairedDelimiter\abs{\lvert}{\rvert}

\DeclarePairedDelimiter\norm{\lVert}{\rVert}
\DeclarePairedDelimiter\prosca{\lparen}{\rparen}

\DeclarePairedDelimiterX{\pair}[2]{\lparen}{\rparen}{#1\,,\,#2}
\DeclarePairedDelimiterX{\pairing}[2]{\langle}{\rangle}{#1\,,\,#2}

\DeclarePairedDelimiterX\set[1]\lbrace\rbrace{\def\given{\;:\;}#1}

\usepackage{color}

\usepackage[french,main=english]{babel}

\usepackage{hyperref}

\definecolor{mylinkcolor}{rgb}{0.1, 0.2, 0.6}

\hypersetup{colorlinks=true, linkcolor=mylinkcolor,
  filecolor=red, urlcolor=blue, citecolor=mylinkcolor,}

\usepackage[author-year]{amsrefs}

\theoremstyle{plain}

\newtheorem{theo}{Theorem}[section]

\newtheorem{lemm}[theo]{Lemma}

\newtheorem{prop}[theo]{Proposition}

\newtheorem{coro}[theo]{Corollary}

\newtheorem{nota}[theo]{Notations}

\theoremstyle{definition}

\newtheorem{defi}[theo]{Definition}

\newtheorem{rema}[theo]{Remark}

\newcommand{\leadmathskip}{\hspace{1.2em}}

\newcommand{\dmslskip}{\hspace{3em}}

\newcommand{\dmslmidskip}{\hspace{8em}}

\newcommand{\dmsllongskip}{\hspace{12em}}

\AtBeginDocument{%
   \def\MR#1{}
}

\begin{document}
\title[Representation theory of semidirect products of a CQG with a finite
group]{Representation theory of semidirect products of a compact quantum group
  with a finite group} \author{Hua Wang} \email{hua.wang@imj-prg.fr}
\address{Institut de mathématiques de Jussieu – Paris Rive Gauche}

\begin{abstract}
  We study unitary representations of semidirect products of a compact
  quantum group with a finite group. We give a classification of all
  irreducible unitary representations, a description of the conjugate
  representation of irreducible unitary representations in terms of
  this classification, and the fusion rules for the semidirect
  product.
\end{abstract}

\maketitle

\tableofcontents{}

\section*{Introduction}
\label{sec:1dc81828c3fedbc9}

It is often the case that one can retrieve significant information about
representations of a group \( G \) from representations of some subgroups of
\( G \). As a trivial example, the study of representations of a direct product
\( G \times H \) of groups of \( G \) and \( H \) can be easily reduced to the study
of representations of \( G \) and \( H \) separately. However, when one replaces
direct products with the more ubiquitous semidirect products, the situation
quickly becomes complicated. To get a taste of this complication, the classic
\cite[\S 8.2]{MR0450380} treats representations of a semidirect product
\( G \rtimes H \) in the special case where \( G, H \) are both finite and
\( G \) is abelian.

In the setting of locally compact groups and their unitary representations, via
the theories of systems of imprimitivity, induced representations, projective
representations (a.k.a.\ ray representations), etc., George Mackey developed a
heavy machinery of techniques, often referred as Mackey's analysis, Mackey's
machine or the little group method (which is also due to Wigner), to attack such
kind of problems. Subsequent works based on Mackey's analysis emerge rapidly,
making it one of the most powerful tools to study unitary representations of
locally compact groups. For an introduction of this development, we refer the
reader to~\cite{MR0098328,MR44536,MR0031489,MR936629,MR3012851} among the large
volumes of literature on this subject.

The author's own interest of this subject comes from the joint
work~\cite{2018arXiv181204078F} with Pierre Fima. In~\cite{2018arXiv181204078F},
we systematically studied the permanence of property (RD) and polynomial growth
of the dual of a bicrossed product of a matched pair consisting of a second
countable compact group and a countable discrete group. The natural subsequent
question of constructing examples of nontrivial bicrossed products with or
without (RD) leads one to study closely the representation theory of semidirect
products \( G \rtimes \Lambda \) of a compact group \( G \) with a finite group
\( \Lambda \). More precisely, as required by the study of length functions relevant
to these properties, we need a classification of all irreducible unitary
representations of \( G \rtimes \Lambda \), the conjugate (which, when we adopt
the point of view of topological quantum groups as in this paper, is also the
contragredient since classic groups are of Kac type) of irreducible
representations in terms of this classification, and most importantly, the
fusion rules of \( G \rtimes \Lambda \), i.e.\ how the tensor product of two
irreducible representations decomposes into a direct sum of irreducible
representations. While the first two questions can be settled using Mackey's
machine as mentioned above, the fusion rules, however, to the best of the
author's limited knowledge, are never calculated in the literature.

This paper treats these questions in the more general setting of semidirect
products of the form \( \mathbb{G} \rtimes \Lambda \), where \( \mathbb{G} \) is
a compact quantum group and \( \Lambda \) a finite group. However, instead of
using systems of imprimitivity, we introduce the notion of representation
parameters (see Definition~\ref{defi:a5316bfdaeb22cfa}), which appears naturally
when we try to analyze the rigid \( C^{\ast} \)-tensor category
\( \repcat(\mathbb{G} \rtimes \Lambda) \). Roughly speaking, a representation
parameter is a triple \( (u, V, v) \), where \( u \) is an irreducible
representation of \( \mathbb{G} \) on some finite dimensional Hilbert space
\( \mathscr{H} \), \( V \) is a unitary projective representation of a certain
subgroup \( \Lambda_{0} \) of \( \Lambda \) on the same space \( \mathscr{H} \),
and \( v \) is a unitary projective representation of the same \( \Lambda_{0} \)
on some other finite dimensional Hilbert space, such that \( V \) is covariant
with \( u \) in a certain sense, and \( V \) and \( v \) have opposing
cocycles. Here, the subgroup \( \Lambda_{0} \) arises as an isotropy subgroup of
a natural action \( \Lambda \curvearrowright \irr(\mathbb{G}) \), and the
projective representation \( V \) is then determined by Schur's lemma on
irreducible representations.

As the precise formulation of our main results are long and complicated, we give
here only a crude summary of these results in terms of representation parameters
(see Definition~\ref{defi:a5316bfdaeb22cfa}) mentioned above.

\begin{enumerate}[label=\textup{(\Alph*)}]
\item \label{item:4389d27a59ad7bdf} Up to equivalence, irreducible unitary
  representations of \( \mathbb{G} \rtimes \Lambda \) are classified by
  (equivalence classes of) representation parameters (see
  Theorem~\ref{theo:91ac3dad0bd1219d} for the precise formulation);
\item \label{item:bac6c72841631b57} The classification in
  \ref{item:4389d27a59ad7bdf} is compatible with the conjugate operation--- the
  conjugate\footnote{The conjugate should \emph{not} be confused with the
    contragredient, with the contragredient not necessarily unitary if the
    quantum group is not unimodular (of Kac-type).}of an irreducible
  representation of \( \mathbb{G} \rtimes \Lambda \) parameterized by some
  representation parameter \( (u, V, v) \) is itself parameterized by the
  conjugate of \( (u, V, v) \) (see Theorem~\ref{theo:e6d75eb838fadd50} for the
  precise formulation);
\item \label{item:68d5bf644d090876} The fusion rules of
  \( \mathbb{G} \rtimes \Lambda \) is calculated by summing a series of
  incidence numbers, where all of these numbers can be calculated using unitary
  projective representations of some suitable subgroup of \( \Lambda \) through
  an explicit reduction procedure (see Theorem~\ref{theo:ff9839aa14503b7c} for
  the precise formulation), where the reduction procedure itself is determined
  by the representation theory of \( \mathbb{G} \) and the action of
  \( \Lambda \) acting on \( \mathbb{G} \), with respect to which we form the
  semidirect product.
\end{enumerate}

While \ref{item:4389d27a59ad7bdf} and \ref{item:bac6c72841631b57} may well be
regarded as the quantum analogue of the corresponding results of Mackey's
analysis in the classical case of groups, our result \ref{item:68d5bf644d090876}
is new, even in the case where \( \mathbb{G} \) is another finite group.  We
should mention that our main idea of this paper starts with reformulating
\( \repcat(\mathbb{G}) \) as a semisimple rigid
\( C^{\ast} \)\nobreakdash-tensor category for an arbitrary compact quantum
group \( \mathbb{G} \), which is the modern point of view; however, Mackey's
ingenious ideas, such as studying the dynamics of a naturally appeared group
action on the representations of a normal subgroup, and using projective
representations of the isotropy subgroups of this action, still play an
essential part in the development of this theory.

We now describe the organization of this paper. The numerous sections of this
paper are roughly divided into the following four parts.  In the first part
(\S~\ref{sec:8407664c09390193} and \S~\ref{sec:cac0f6b0354689f8}), we lay out
the basic properties and constructions of the objects to be studied in this
paper---semidirect products of a compact quantum group by a finite group and
their unitary representations. In \S~\ref{sec:cac0f6b0354689f8}, the problem of
describing unitary representations of these semidirect products is reduced to
the study of the so-called covariant pair of representations for each of the
factors. The second part
(\S\S~\ref{sec:90f8ecd4790ea7ac}\nobreakdash-~\ref{sec:51325cc5d5808588}) gives
a self-contained treatment of induced representation which will be used later in
this paper. We are aware that there are already much more general theory for
induced representations in the quantum setting, e.g.~\cite{MR1934609} based on
the classic work \cite{MR0353003}. Moreover S.\ Vaes has generalizes a large
part of Mackey's theory of imprimitivity to locally compact quantum groups in
\cite{MR2182592}. Besides the obvious reason for fixing the notations, the
treatment of the induced representation here is specially tailored to the
various calculations in the later half of this paper. The third part
(\S\S~\ref{sec:deea2fa481e1836a}\nobreakdash-\ref{sec:51cb2ca9f8d7976c}) is the
technical core of this paper. The treatment here is largely inspired by
Woronowicz's Krein-Tannaka reconstruction~\cite{MR943923} of a compact quantum
group from its representation category. Here instead of directly attacking the
representation category \( \repcat{\mathbb{G} \rtimes \Lambda} \) of the
semidirect product, we introduce and study a family of rigid
\( C^{\ast} \)\nobreakdash-tensor categories (called the category of covariant
systems of representations and denoted by \( \mathcal{CSR}_{\Lambda_{0}} \) with
respect to some suitable subgroup of \( \Lambda \)), each of which has a simpler
structure. Combining the information we have on these simpler
\( C^{\ast} \)\nobreakdash-tensor categories allows us not only to classify the
irreducible unitary representations of \( \mathbb{G} \rtimes \Lambda \), but
also to calculate the fusion rules of \( \mathbb{G} \rtimes \Lambda \). The
details of this classification and calculation are given in the fourth part
(\S\S~\ref{sec:f7e4fdcede696c78}\nobreakdash-\ref{sec:20a173060e7c2f82}).

Before we proceed further, we feel that we should say a little more about
\S~\ref{sec:8407664c09390193} for the experts. We emphasize our construction of
semidirect products as the axiomatically more elaborate \emph{algebraic} compact
quantum groups, the theory of which is developed by van Daele
\cite{MR1658585,MR1378538,MR1220906}, instead of the more modern and standard
formulation, due to Woronowicz \cite{MR1616348,MR943923,MR901157}, using
\( C^{\ast} \)\nobreakdash-algebras. Of course, these two approaches are
essentially equivalent---one passes from Woronowicz's approach to van Daele's
via the Peter-Weyl theory for compact quantum groups, and from van Daele's
approach to Woronowicz's via the famous GNS construction with respect to the
Haar integral. The reasons we prefer van Daele's algebraic theory here are
two-fold: on the one hand, one has the advantage of having direct access to the
Haar state and the antipode, as well as the polynomial algebra, which are
powerful tools for our purposes of studying the representations of these objects
(or corepresentations if one insists on viewing these essentially analytic
objects as Hopf algebras); on the other hand, when one tries to restrict
representations to certain (quantum) subgroups of these semidirect products, as
we will do later, one will need to use the counit, which is always everywhere
defined in the more elaborate algebraic approach of van Daele, but is merely
densely defined if the compact quantum group in the sense of Woronowicz is not
universal. We also point out here that the term \emph{semidirect product} in the
quantum setting has an unfortunate ambiguity. Nowadays many use this term to
refer to the crossed product, as first defined and studied by S.\
Wang~\cite{MR1347410}. In the case of classical groups, it is long known that
this crossed product construction yields the \emph{convolution algebra} of the
semidirect of groups. So if we believe classic compact groups are exactly
compact quantum groups whose algebra is commutative, then this is not the
correct formulation for semidirect products, even though these are closely
related via the convolution operation (which is a manifestation of the quantum
version of Pontryagin's duality as developed in \cite{MR1832993}, preceded by
many important works along the lines of the Kac's program, a history of which is
described in the introduction of the above article). To be more precise, this
crossed product of S.\ Wang is in fact a special case of the bicrossed product
as described in \cite{MR1970242} where one of the actions for the matched pair
is trivial; what we call semidirect in this paper is a special case of the
double crossed product as described in \cite{MR2115071} where again one of the
actions for the matched pair is trivial. We don't pursue the full generality of
the bicrossed product construction and double crossed product construction here,
but merely point out that they are all based on the notion of matched pair of
(quantum) groups (\cite{MR1045735,MR1092128},\cite{MR611561}). We also mention
in passing works such as \cite{MR1235438}, \cite{MR1098985}, \cite{MR320056},
\cite{MR1745504}, and \cite{MR1970242} in the direction of bicrossed products,
and works such as \cite{MR1235438}, \cite{MR1098985} and \cite{MR2115071} in the
direction of double crossed products. We hope these backgrounds provide some
justification of our choice of terminology for semidirect products by its
consistency with the classical group case. We also note that representation
theory for compact bicrossed products (which includes the crossed product as a
special case) of a matched pair of classical groups are thoroughly investigated
in the author's joint work with P.\ Fima \cite{2018arXiv181204078F}, and as one
can see by comparing the results there and the results of this paper, the
representation theory for semidirect products are significantly more delicate
than crossed products, even for classical finite groups.

We conclude this introduction by making some conventions. All representations
and projective representations in this paper are finite dimensional. All of
them are unitary, except the contragredient of a unitary representation, which
may not be unitary when the compact quantum group is not of Kac-type. We also
assume all (projective) representations are over a finite dimensional Hilbert
space instead of a mere complex vector space. Terminologies and notations
concerning compact quantum groups and \( C^{\ast} \)-tensor categories are
largely in consistent with those in~\cite{MR3204665}. We also use freely the
Peter-Weyl theory for projective representations of finite groups as presented
in~\cite{MR3299063}. We also freely use the Heyenmann-Sweedler notation in
performing calculations on comultiplications. The unitary group of unitary
transformations from a Hilbert space \( \mathscr{H} \) to itself is denoted by
\( \mathcal{U}(\mathscr{H}) \). From \S~\ref{sec:0bd5d546e466b2bc} on,
\( \mathbb{T} \) denotes the circle group, i.e.\ the abelian compact group
\( \set*{z \in \mathbb{C} \given \abs*{z} = 1} \) viewed as a subgroup of
\( \mathbb{C}^{\times} \). Since we often view a representation of compact
quantum groups as an operator, we denote the tensor product of representations
using \( \times \) instead of \( \otimes \), as the latter is reserved to denote
tensor products of spaces, algebras, linear operators, etc. Finally, throughout
this paper, we fix a compact quantum group \( \mathbb{G} = (A, \Delta) \), a
finite group \( \Lambda \), and an antihomomorphism of groups
\( \alpha^{\ast} \colon \Lambda \rightarrow \aut\bigl( C(\mathbb{G}), \Delta
\bigr) \), where \( \aut\bigl( C(\mathbb{G}), \Delta \bigr) \) is the
subgroup\footnote{Note that the notation \( \aut(\mathbb{G}) \) has a certain
  ambiguity which we try to avoid: one the one hand, if we let \( \mathbb{G} \)
  to be a classical compact group, then elements of \( \aut(\mathbb{G}) \) are
  group automorphisms, and the group law of \( \aut(\mathbb{G}) \) is given by
  composition of \emph{set-theoretic} mappings; on the other hand, if we view
  \( \mathbb{G} \) as a Hopf-\( C^{\ast} \)-algebra, say
  \( \bigl(C(\mathbb{G}), \Delta\bigr) \), then \( \aut(\mathbb{G}) \) can also
  be mean the automorphism group of this
  Hopf-\( C^{\ast} \)\nobreakdash-algebra, whose group law is given by
  composition of \emph{Hopf algebraic}-morphisms. This is the reason we prefer
  the more cumbersome notation \( \aut\bigl( C(\mathbb{G}), \Delta \bigr) \)
  instead of the ambiguous but more succinct \( \aut(\mathbb{G}) \).} of
\( \aut\bigl( C(\mathbb{G}) \bigr) \) consisting of automorphisms of the
\( C^{\ast} \)\nobreakdash-algebra \( C(\mathbb{G}) \) that intertwines the
comultiplication \( \Delta \).

\textbf{Acknowledgment.} This work is funded by the ANR Project
(No. ANR-19-CE40-0002), to which the author expresses his sincere gratitude. The
author especially thanks Pierre Fima for his encouragement and various helpful
discussions duringthe production of this work.

\section{Semidirect product of a compact quantum group with a
  finite group}
\label{sec:8407664c09390193}

Let \( \mathbb{G} = ( A, \Delta ) \) be a compact quantum group, \( \Lambda \) a
finite group. An action of \( \Lambda \) on \( \mathbb{G} \) via quantum
automorphisms is an antihomomorphism
\( \alpha^{\ast} \colon \Lambda \rightarrow \aut\bigl( C(\mathbb{G}), \Delta
\bigr) \).  One can then form the semi-direct
\( \mathbb{G} \rtimes_{\alpha^{\ast}} \Lambda \), or simply
\( \mathbb{G} \rtimes \Lambda \) if the action \( \alpha^{\ast} \) is clear from
the context, which is again a compact quantum group. The underlying
{\( C^{\ast} \)}\nobreakdash-algebra \( \mathscr{A} \) of
\( \mathbb{G} \rtimes \Lambda \) is \( A \otimes C( \Lambda ) \), and the
comultiplication \( \widetilde{\Delta} \) on \( \mathscr{A} \) is determined by
\begin{equation}
  \label{eq:0f23581b02e7ab34}
  \widetilde{\Delta} (a \otimes \delta_{r}) = \sum_{s \in \Lambda} %
  {\left[( \id_{A} \otimes \alpha^{\ast}_{s} ) \Delta(a)\right]}_{13} %
  {( \delta_{s} \otimes \delta_{s^{-1}r})}_{24} %
  \in A \otimes C( \Lambda ) \otimes A \otimes C( \Lambda )
\end{equation}
for any \( a \in A \) and \( r \in \Lambda \). As we've mentioned at the end of
the introduction, from now on, \( \mathbb{G} \), \( \Lambda \) and the action
\( \alpha^{\ast} \) are fixed until the end of the paper.

It is clear that \( \widetilde{\Delta} \) is a unital
\( \ast \)\nobreakdash-morphism. We now check that in the six-fold tensor
product
\( A \otimes C(\Lambda) \otimes A \otimes C(\Lambda) \otimes A \otimes
C(\Lambda) \), we have

\begin{equation}
  \forall a \in A, r \in \Lambda, \qquad
  \label{eq:aa3dd430575f2bc3}
  (\id \otimes \id \otimes \widetilde{\Delta})
  [\widetilde{\Delta}(a \otimes \delta_{r})]
  = (\widetilde{\Delta} \otimes \id \otimes \id) %
  [\widetilde{\Delta}(a \otimes \delta_{r})],
\end{equation}
i.e.\ our new comultiplication \( \widetilde{\Delta} \) is
coassociative. Indeed, put
\( \Delta^{(2)}:= (\id \otimes \Delta) \Delta = (\Delta \otimes \id) \Delta \),
since \( \alpha^{\ast}_{s} \in \aut\bigl( C(\mathbb{G}), \Delta \bigr) \) for
all \( s \in \Lambda \), we have
\( (\alpha_{s}^{\ast} \otimes \alpha_{s}^{\ast}) \circ \Delta = \Delta \circ
\alpha_{s}^{\ast} \),
\begin{equation}
  \label{eq:0eeb26f70674d5c0}
  \begin{split}
    & \leadmathskip (\id \otimes \id \otimes \widetilde{\Delta}) %
    [\widetilde{\Delta}(a \otimes \delta_{r})] \\
    &= (\id \otimes \id \otimes \widetilde{\Delta}) %
    \left( %
      \sum_{s \in \Lambda} {[(\id \otimes \alpha_{s}^{\ast}) \Delta(a)]}_{13} %
      {(\delta_{s} \otimes \delta_{s^{-1}r})}_{24} %
    \right) \\
    &= \sum_{s,t \in \Lambda} %
    {\Bigl\{ %
      \bigl[\bigl( %
      \id \otimes \alpha_{s}^{\ast} \otimes (\alpha_{t}^{\ast} \circ
      \alpha_{s}^{\ast})\bigr) (\id \otimes \Delta)\Delta\bigr] (a) %
      \Bigr\}}_{135}
    {(\delta_{s} \otimes \delta_{t} \otimes \delta_{t^{-1}s^{-1}r})}_{246} \\
    &= \sum_{s,t \in \Lambda} %
    {\Bigl\{ %
      \bigl[(\id \otimes \alpha_{s}^{\ast} \otimes \alpha_{st}^{\ast})
      \Delta^{(2)}\bigr] (a) %
      \Bigr\}}_{135} {(\delta_{s} \otimes \delta_{t} \otimes
      \delta_{t^{-1}s^{-1}r})}_{246}.
  \end{split}
\end{equation}
On the other hand,
\begin{equation}
  \label{eq:6ae689469c41ffe8}
  \begin{split}
    & \leadmathskip (\widetilde{\Delta} \otimes \id \otimes \id) %
    [\widetilde{\Delta}(a \otimes \delta_{r})] \\
    &= (\widetilde{\Delta} \otimes \id \otimes \id) %
    \left( %
      \sum_{s \in \Lambda} {[(\id \otimes \alpha_{s}^{\ast}) \Delta(a)]}_{13} %
      {(\delta_{s} \otimes \delta_{s^{-1}r})}_{24} %
    \right) \\
    &= \sum_{s, t \in \Lambda}\Bigl\{ %
    \bigl[ %
    (\id \otimes \alpha_{t}^{\ast} \otimes \alpha_{s}^{\ast}) %
    (\Delta \otimes \id) \Delta %
    \bigr] (a) \Bigr\}_{135} %
    {(\delta_{t} \otimes \delta_{t^{-1}s} \otimes \delta_{s^{-1}r})}_{246} \\
    & \leadmathskip (s' = t, t'=t^{-1}s \iff s = s't', t = s') \\
    &= \sum_{s', t' \in \Lambda}\Bigl\{ %
    \bigl[ %
    (\id \otimes \alpha_{s'}^{\ast} \otimes \alpha_{s't'}^{\ast}) %
    \Delta^{(2)} \bigr] (a) \Bigr\}_{135} %
    {(\delta_{s'} \otimes \delta_{t'} \otimes \delta_{t'^{-1}s'^{-1}r})}_{246}.
  \end{split}
\end{equation}
Now \eqref{eq:aa3dd430575f2bc3} follows from \eqref{eq:0eeb26f70674d5c0} and
\eqref{eq:6ae689469c41ffe8}.

Since \( \pol(\mathbb{G}) \otimes C( \Lambda ) \) is dense in
\( A \otimes C( \Lambda ) \), in order to prove that
\( \mathbb{G} \rtimes \Lambda \) is indeed a compact quantum group, it suffices
to show that \( (\pol(\mathbb{G}) \otimes C( \Lambda ), \widetilde{\Delta}) \)
is an algebraic compact quantum group, i.e.\ a Hopf
\( \ast \)\nobreakdash-algebra with an invariant state (called the Haar state or
Haar integral).

First of all, since
\( \alpha_{s}^{\ast} \in \aut\bigl( C(\mathbb{G}), \Delta \bigr) \) for all
\( s \in \Lambda \), we have
\( \alpha_{s}^{\ast}\bigl(\pol(\mathbb{G})\bigr) = \pol(\mathbb{G}) \), and
\( \widetilde{\Delta} \) indeed restricts to a well-defined comultiplication on
\( \pol(\mathbb{G}) \otimes C(\Lambda) \).

Let \( \epsilon, S \) be the counit and the antipode respectively for the Hopf
\( \ast \)\nobreakdash-algebra \( \pol(\mathbb{G}) \). Denoting the neutral
element of the group \( \Lambda \) by \( e \), we define
\begin{equation}
  \label{eq:e0375ae7a578342a}
  \begin{split}
    \widetilde{\epsilon} \colon \pol(\mathbb{G}) \otimes C( \Lambda ) %
    & \rightarrow \mathbb{C} \\
    \sum_{r \in \Lambda} x_{r} \otimes \delta_{r} & \mapsto \epsilon(x_{e}),
  \end{split}
\end{equation}
and
\begin{equation}
  \label{eq:e8daacc3f7b72fa2}
  \begin{split}
    \widetilde{S} \colon \pol(\mathbb{G}) \otimes C( \Lambda ) & \rightarrow
    \pol(\mathbb{G}) \otimes C( \Lambda ) \\
    \sum_{r \in \Lambda} x_{r} \otimes \delta_{r} & \mapsto \sum_{r \in \Lambda}
    \alpha_{r}^{\ast}(S(x_{r})) \otimes \delta_{r^{-1}} = %
    \sum_{r \in \Lambda} S( \alpha_{r^{-1}}^{\ast}(x_{r^{-1}}) ) \otimes
    \delta_{r}.
  \end{split}
\end{equation}
Since \( \epsilon \) is a \( \ast \)-morphism of algebras, so is
\( \widetilde{\epsilon} \). Moreover, for any \( x \in \pol(\mathbb{G}) \) and
\( r \in \Lambda \), we have
\begin{align*}
  & \leadmathskip (\widetilde{\epsilon} \otimes \id) %
    \widetilde{\Delta} (x \otimes \delta_{r})
    =  (\widetilde{\epsilon} \otimes \id) %
    \sum_{s \in \Lambda} \sum x_{(1)} \otimes \delta_{s} %
    \otimes \alpha_{s}^{\ast}(x_{(2)})  \otimes \delta_{s^{-1}r} \\
  &= \sum \epsilon(x_{(1)}) \alpha_{e}^{\ast}(x_{(2)}) \otimes \delta_{r} %
    = \sum \epsilon(x_{(1)}) x_{(2)} %
    \otimes \delta_{r} = x \otimes \delta_{r} \\
  &= \sum x_{(1)} \epsilon(x_{(2)}) \otimes \delta_{r} %
    = \sum x_{(1)} \epsilon( \alpha_{r}^{\ast}(x_{(2)})) \otimes \delta_{r} \\
  &= (\id \otimes \widetilde{\epsilon}) %
    \sum_{s \in \Lambda} \sum x_{(1)} \otimes \delta_{s} %
    \otimes \alpha_{s}^{\ast}(x_{(2)})  \otimes \delta_{s^{-1}r} %
    = (\id \otimes \widetilde{\epsilon}) %
    \widetilde{\Delta} (x \otimes \delta_{r}).
\end{align*}
Hence \( \widetilde{\epsilon} \) is the counit for \( \widetilde{\Delta} \). Let
\( m \colon \pol(G) \otimes \pol(G) \rightarrow \pol(G) \) be the multiplication
map, and \( \widetilde{m} \) the multiplication map on
\( \pol(\mathbb{G}) \otimes C( \Lambda ) \), then
\begin{align*}
  \widetilde{m}(\widetilde{S} \otimes \id) %
  \widetilde{\Delta}(x \otimes \delta_{r}) %
  &= \widetilde{m}(\widetilde{S} \otimes \id) %
    \sum_{s \in \Lambda} \sum x_{(1)} \otimes \delta_{s} %
    \otimes \alpha_{s}^{\ast}(x_{(2)})  \otimes \delta_{s^{-1}r} \\
  &= \widetilde{m} \sum_{s \in \Lambda} \sum %
    \alpha_{s}^{\ast}(S(x_{(1)})) \otimes \delta_{s^{-1}} \otimes
    \alpha_{s}^{\ast}(x_{(2)}) \otimes \delta_{s^{-1}r} \\
  &= \sum_{s \in \Lambda} \left[m(S \otimes \id) %
    ( \alpha_{s}^{\ast} \otimes \alpha_{s}^{\ast}) \Delta(x)
    \right] \otimes \delta_{s^{-1}} \cdot \delta_{s^{-1}r} \\
  &= \delta_{e, r} \sum_{s \in \Lambda} \left[m(S \otimes \id) %
    \Delta( \alpha_{s}^{\ast}(x) )\right] \otimes
    \delta_{s^{-1}} \\
  &= \delta_{e, r} \sum_{s \in \Lambda} \epsilon( \alpha_{s}^{\ast}(x) ) %
    1_{A} \otimes \delta_{s^{-1}} \\
  &= \delta_{e, r} \epsilon(x) 1_{A} %
    \otimes \sum_{s \in \Lambda} \delta_{s^{-1}} \\
  &= \delta_{e,r} \epsilon(x) 1_{A} \otimes 1_{C( \Lambda )} %
    = \widetilde{\epsilon}(x \otimes \delta_{r}) 1_{A} \otimes 1_{C( \Lambda )}.
\end{align*}
Similarly, since for any \( s \in \Lambda \),
\begin{displaymath}
  \alpha_{s^{-1}r}^{\ast} S \alpha_{s}^{\ast} %
  = \alpha_{s^{-1}r}^{\ast} \alpha_{s}^{\ast} S =
  \alpha_{r}^{\ast} S,
\end{displaymath}
we have
\begin{align*}
  \widetilde{m}(\id \otimes \widetilde{S}) %
  \widetilde{\Delta} (x \otimes \delta_{r}) %
  &= \widetilde{m} (\id \otimes \widetilde{S}) %
    \sum_{s \in \Lambda} \sum x_{(1)} \otimes \delta_{s} %
    \otimes \alpha_{s}^{\ast}(x_{(2)})  \otimes \delta_{s^{-1}r} \\
  &= \widetilde{m} \sum_{s \in \Lambda} \sum x_{(1)} \otimes \delta_{s} %
    \otimes ( \alpha_{s^{-1}r}^{\ast} S \alpha_{s}^{\ast}) %
    (x_{(2)}) \otimes \delta_{r^{-1}s} \\
  &= \widetilde{m} \sum_{s \in \Lambda} \sum x_{(1)} \otimes \delta_{s} %
    \otimes ( \alpha_{r}^{\ast} S )(x_{2}) \otimes \delta_{r^{-1} s} \\
  &= \delta_{e, r} \sum_{s \in \Lambda} \sum %
    x_{(1)} [S(x_{(2)})] \otimes \delta_{s} %
    = \delta_{e, r} \sum_{s \in \Lambda} \epsilon(x) 1_{A} \otimes \delta_{s} \\
  &= \delta_{e, r} \epsilon(x) 1_{A} \otimes 1_{C( \Lambda )} %
    = \widetilde{\epsilon}(x \otimes \delta_{r}) 1_{A} \otimes 1_{C( \Lambda )}.
\end{align*}
Therefore, \( \widetilde{S} \) is the antipode for
\( (\pol(\mathbb{G}) \otimes C( \Lambda ), \widetilde{\Delta}) \).

It remains to construct the Haar state on the Hopf
\( \ast \)\nobreakdash-algebra \( \pol(\mathbb{G}) \otimes C(\Lambda)
\). Suppose \( h \colon A \rightarrow \mathbb{C} \) is the Haar state on
\( \mathbb{G} \), define
\begin{equation}
  \label{eq:1c67d64dafdd1948}
  \begin{split}
    \widetilde{h} \colon \pol(\mathbb{G}) \otimes C( \Lambda ) %
    & \to \mathbb{C} \\
    \sum_{r} x_{r} \otimes \delta_{r} %
    & \mapsto \abs*{\Lambda}^{-1} \sum_{r \in \Lambda} h(x_{r}).
  \end{split}
\end{equation}
It is obvious that \( \widetilde{h} \) is a state. For any
\( x \in \pol(\mathbb{G}) \), \( r \in \Lambda \),
\begin{align*}
  ( \widetilde{h} \otimes \id) \widetilde{\Delta}(x \otimes \delta_{r}) %
  &= \abs*{\Lambda}^{-1}\sum_{s \in \Lambda} \sum %
    h(x_{(1)}) \alpha_{s}^{\ast}(x_{(2)}) \otimes
    \delta_{s^{-1}r} \\
  &= \abs*{\Lambda}^{-1} \sum_{s \in \Lambda} \alpha_{s}^{\ast}\bigl(\sum
    h(x_{(1)})x_{(2)}\bigr) \otimes \delta_{s^{-1}r} \\
  &= \abs*{\Lambda}^{-1} \sum_{s \in \Lambda} \alpha_{s}^{\ast}(h(x) 1_{A}) %
    \otimes \delta_{s^{-1}r} \\
  &= \abs*{\Lambda}^{-1} h(x) \sum_{s \in \Lambda} %
    1_{A} \otimes \delta_{s^{-1} r} \\
  &= \widetilde{h}(x \otimes \delta_{r}) 1_{A} \otimes 1_{C( \Lambda )}.
\end{align*}
The uniqueness of the Haar state implies that
\( h \circ \alpha_{s}^{\ast} = h \) for any \( s \in \Lambda \), hence
\begin{align*}
  (\id \otimes \widetilde{h}) \widetilde{\Delta}(x \otimes \delta_{r}) %
  &= \abs*{\Lambda}^{-1} \sum_{s \in \Lambda} \sum x_{(1)} %
    h\bigl( \alpha_{s}^{\ast}(x_{(2)})
    \bigr) \otimes \delta_{s} \\
  &= \abs*{\Lambda}^{-1} \sum_{s \in \Lambda} \sum %
    x_{(1)} h(x_{(2)}) \otimes \delta_{s} \\
  &= \abs*{\Lambda}^{-1} h(x) 1_{A} \otimes \sum_{s \in \Lambda} \delta_{s} \\
  &= \widetilde{h}(x \otimes \delta_{r}) 1_{A} \otimes 1_{C( \Lambda )}.
\end{align*}
Therefore, \( \widetilde{h} \) is indeed the Haar state on
\( (\pol(\mathbb{G}) \otimes C( \Lambda ), \widetilde{\Delta}) \). So far, we've
established that
\( (\pol(\mathbb{G}) \otimes C( \Lambda ), \widetilde{\Delta}) \) is an
algebraic compact quantum group (cf.~\cite[paper \( 3 \)]{MR2397671}).

Now the density of \( \pol(\mathbb{G}) \otimes C( \Lambda ) \) in
\( A \otimes C( \Lambda ) \) implies that
\( (A \otimes C( \Lambda ), \widetilde{\Delta}) \) is indeed a compact quantum
group, with
\begin{equation}
  \label{eq:bc55df6d41c99f75}
  \pol(\mathbb{G} \rtimes \Lambda) = \pol(\mathbb{G}) \otimes C( \Lambda ),
\end{equation}
and Haar state (which we still denote by \( \widetilde{h} \))
\begin{equation}
  \label{eq:60aee55ee4c2168e}
  \begin{split}
    \widetilde{h} : A \otimes C( \Lambda ) & \rightarrow \mathbb{C} \\
    \sum x_{r} \otimes \delta_{r} %
    & \mapsto \abs*{\Lambda}^{-1} \sum_{r \in \Lambda} h(x_{r}).
  \end{split}
\end{equation}
Furthermore, we've seen that the counit \( \widetilde{\epsilon} \) and the
antipode \( \widetilde{S} \) of the Hopf \( \ast \)\nobreakdash-algebra
\( \pol(\mathbb{G} \rtimes \Lambda) \) are given by \eqref{eq:e0375ae7a578342a}
and~\eqref{eq:e8daacc3f7b72fa2} respectively (cf.~\cite[\S 5.4.2]{MR2397671}).

\begin{defi}
  \label{defi:cdbab2a2a6c441a2}
  Using the above notations, it is well-known that the analytic compact quantum
  group \( (\mathscr{A}, \widetilde{\Delta}) \) and the algebraic compact
  quantum group \( (A \otimes C(\Lambda), \widetilde{\Delta}) \) are equivalent
  descriptions of the same object, which we call \textbf{the semidirect product}
  of \( \mathbb{G} \) and \( \Lambda \) with respect to the action
  \( \alpha^{\ast} \), and is denoted by
  \( \mathbb{G} \rtimes_{\alpha^{\ast}} \Lambda \), or simply
  \( \mathbb{G} \rtimes \Lambda \) if the underlying action \( \alpha^{\ast} \)
  is clear from context.
\end{defi}

\begin{rema}
  \label{rema:9fe237a324b85a98}
  There is a faster way of establishing \( \mathbb{G} \rtimes \Lambda \) as a
  compact quantum group, which we refer to as the analytic approach. Namely, one
  might use~\eqref{eq:0f23581b02e7ab34} directly to define a comultiplication on
  the {\( C^{\ast} \)}\nobreakdash-algebra \( A \otimes C(\Lambda) \) and show
  that this comultiplication satisfy the density condition in the definition of
  a compact quantum group in the sense of Woronowicz (cf.~\cite{MR1616348}). We
  prefer the more algebraic approach presented above as it provides more insight
  for our purpose of studying representations of
  \( \mathbb{G} \rtimes \Lambda \). As an illustration, from our treatment, one
  knows immediately that
  \( \pol(\mathbb{G} \rtimes \Lambda) = \pol(\mathbb{G}) \rtimes \Lambda \), a
  fact that is not clear from the faster analytic approach.
\end{rema}

\begin{rema}
  \label{rema:1db1819693329372}
  When \( \mathbb{G} \) comes from a genuine compact group \( G \), it is easy
  to check via Gelfand theory that the antihomomorphism
  \( \alpha^{\ast} \colon \Lambda \rightarrow \aut\bigl( C(\mathbb{G}), \Delta
  \bigr) \) comes from the pull-back of a group morphism
  \( \alpha \colon \Lambda \rightarrow \aut\bigl( C(\mathbb{G}), \Delta \bigr)
  \), and \( \mathbb{G} \rtimes \Lambda \) is exactly the compact group
  \( G \rtimes_{\alpha} \Lambda \) viewed as a compact quantum group, where the
  group law on \( G \times \Lambda \) is defined by
  \begin{equation}
    \label{eq:6521b852c75464b9}
    \forall g, h \in G, \, r, s \in \Lambda, \qquad
    (g, r)(h, s) = \bigl(g \alpha_{r}(h), rs\bigr).q
  \end{equation}
\end{rema}

In treating the dual objects of some rigid \( C^{\ast} \)-tensor to be presented
later, the following result will be useful.
\begin{prop}
  \label{prop:22073f2b39059e9e}
  The compact quantum group \( \mathbb{G} \rtimes \Lambda \) is of Kac type if
  and only if \( \mathbb{G} \) is of Kac type.
\end{prop}
\begin{proof}
  Of the many equivalent characterization for a compact quantum group to be of
  Kac type\footnote{see e.g.~\cite[\S 1.7]{MR3204665}}, we use the fact that
  such a quantum group is of Kac type if and only if the antipode of its
  polynomial algebra preserves adjoints. The proposition now becomes trivial in
  view of \eqref{eq:e8daacc3f7b72fa2}.
\end{proof}

\section{A first look at unitary representations of
  \texorpdfstring{\( \mathbb{G} \rtimes \Lambda \)}{G x| Lambda}}
\label{sec:cac0f6b0354689f8}

A unitary representation \( U \) of a classic compact semidirect product
\( G \rtimes \Lambda \) is determined by the restrictions \( U_{G} \) and
\( U_{\Lambda} \) on the subgroups \( G \times 1_{\Lambda} \simeq G \) and
\( 1_{G} \times \Lambda \simeq \Lambda \) respectively. It is easy to see that
(cf.\ \eqref{eq:6521b852c75464b9})
\begin{equation}
  \label{eq:45f4b7519a2de266}
  \begin{split}
    \forall g \in G, r \in \Lambda,%
    &\leadmathskip U_{G}( \alpha_{r}(g) ) U_{\Lambda}(r)
    = U( \alpha_{r}(g), r ) \\
    &= U((1, r)(g, 1)) = U_{\Lambda}(r) U_{G}(g).
  \end{split}
\end{equation}
Conversely, suppose \( U_{G} \), \( U_{\Lambda} \) are unitary representations
on the same Hilbert space of \( G \) and \( \Lambda \) respectively,
if~\eqref{eq:45f4b7519a2de266} is satisfied, then
\( U(g, r) := U_{G}(g) U_{\Lambda}(r) \) defines a unitary representation of
\( G \rtimes \Lambda \). When \( G \) is replaced by a general compact quantum
group \( \mathbb{G} \), even though the ``elements'' of \( \mathbb{G} \) are no
longer available, one can still establish a reasonable quantum analogue. We
begin with a simple lemma.

\begin{lemm}
  \label{lemm:d93132078264edaa}
  Let \( \epsilon \) be the counit for \( \pol(\mathbb{G}) \),
  \( \epsilon_{\Lambda} \) the counit for \( C( \Lambda ) \), then
  \( \epsilon \otimes \id_{C( \Lambda )} \) is a Hopf
  \( \ast \)\nobreakdash-algebra morphism from
  \( \pol(\mathbb{G}) \otimes C( \Lambda ) \) onto \( C( \Lambda ) \), and
  \( \id_{\pol(\mathbb{G})} \otimes \epsilon_{\Lambda} \) is a Hopf
  \( \ast \)\nobreakdash-algebra morphism from
  \( \pol(\mathbb{G}) \otimes C( \Lambda ) \) onto \( \pol(\mathbb{G}) \).
\end{lemm}
\begin{proof}
  Since the antipodes are \( \ast \)-morphisms of involutive algebras, it
  suffices to check that both morphisms preserve comultiplication.

  Take any \( a \in \pol(\mathbb{G}) \), \( r \in \Lambda \), we have
  \begin{align*}
    & \leadmathskip [( \epsilon \otimes \id) \otimes ( \epsilon \otimes \id )] %
      \Delta_{\mathbb{G} \rtimes \Lambda}(a \otimes
      \delta_{r}) \\
    &= \sum_{s \in \Lambda} \sum \epsilon(a_{(1)}) %
      \epsilon( \alpha_{s}^{\ast}(a_{(2)})) \delta_{s} \otimes
      \delta_{s^{-1}r} \\
    &= \sum_{s \in \Lambda} \sum \epsilon(a_{(1)}) \epsilon(a_{(2)})
      \delta_{s} \otimes \delta_{s^{-1}r} \\
    &= \sum_{s \in \Lambda} \epsilon(a) \delta_{s} \otimes \delta_{s^{-1} t} %
      = \Delta_{\Lambda}( \epsilon \otimes \id )(a \otimes \delta_{r}),
  \end{align*}
  where \( \Delta_{\Lambda} \) is the comultiplication for \( \Lambda \) viewed
  as a compact quantum group. Thus \( \epsilon \otimes \id \) preserves
  comultiplication. On the other hand, note that
  \( \epsilon_{\Lambda}( \delta_{r} ) = \delta_{r, 1_{\Lambda}} \), we have
  \begin{align*}
    & \leadmathskip %
      [(\id \otimes \epsilon_{\Lambda})
      \otimes (\id \otimes \epsilon_{\Lambda})] %
      \Delta_{\mathbb{G} \otimes \Lambda}(a \otimes \delta_{r}) \\
    &= \sum_{s \in \Lambda} \delta_{s, 1_{\Lambda}} %
      \delta_{s^{-1}r, 1_{\Lambda}} \sum a_{(1)} \otimes
      \alpha_{s}^{\ast}(a_{(2)}) \\
    &= \delta_{r, 1_{\Lambda}} \sum a_{(1)} \otimes a_{(2)} \\
    &= \delta_{r, 1_{\Lambda}} \Delta(a)  %
      = \Delta [(\id \otimes \epsilon_{\Lambda})(a \otimes \delta_{r})].
  \end{align*}
  Thus \( \id \otimes \epsilon_{\Lambda} \) preserves comultiplication too.
\end{proof}

Let
\( U \in \mathcal{B}(\mathscr{H}) \otimes \pol(\mathbb{G}) \otimes C( \Lambda )
\) be a finite dimensional unitary representation of
\( \mathbb{G} \rtimes \Lambda \). Define the unitaries
\begin{displaymath}
  \res_{\mathbb{G}}(U) \colon = (\id_{\mathcal{B}(\mathscr{H})} \otimes
  \id_{\pol(\mathbb{G})} \otimes \epsilon_{\Lambda}) (U) %
  \in \mathcal{B}(\mathscr{H}) \otimes
  \pol(\mathbb{G}),
\end{displaymath}
and
\begin{displaymath}
  \res_{\Lambda}(U) \colon %
  = (\id_{\mathcal{B}(\mathscr{H})} \otimes \epsilon_{\mathbb{G}} \otimes
  \id_{C( \Lambda )})(U) \in \mathcal{B}(\mathscr{H}) \otimes C( \Lambda ).
\end{displaymath}
Then by Lemma~\ref{lemm:d93132078264edaa}, we see that
\( \res_{\mathbb{G}}(U) \) is a finite dimensional unitary representation of
\( \mathbb{G} \) and \( \res_{\Lambda}(U) \) a finite dimensional unitary
representation of \( \Lambda \). We call \( \res_{\mathbb{G}}(U) \) (resp.\
\( \res_{\Lambda}(U) \)) the restriction of \( U \) to \( \mathbb{G} \) (resp.\
\( \Lambda \)). For reasons to be explained presently, we also write
\( U_{\mathbb{G}} \) for \( \res_{\mathbb{G}}(U) \) and \( U_{\Lambda} \) for
\( \res_{\Lambda}(U) \).

\begin{prop}
  \label{prop:3693cca5392c7dd3}
  Using the above notations, we have
  \begin{equation}
    \label{eq:1611871ef2500d9a}
    \forall  r_{0} \in \Lambda, \quad
    (U_{\Lambda}(r_{0}) \otimes 1_{A}) U_{\mathbb{G}} %
    =
    [(\id_{\mathcal{B}(\mathscr{H})} \otimes
    \alpha_{r_{0}}^{\ast})(U_{\mathbb{G}})] (U_{\Lambda}(r_{0}) \otimes 1_{A}) %
  \end{equation}
  in \( \mathcal{B}(\mathscr{H}) \otimes \pol(\mathbb{G}) \). Moreover,
  \begin{equation}
    \label{eq:050021bf02f8bfac}
    U = {(U_{\mathbb{G}})}_{12} {(U_{\Lambda})}_{13} \in
    \mathcal{B}(\mathscr{H}) \otimes \pol(\mathbb{G}) \otimes C( \Lambda ).
  \end{equation}

  Conversely, suppose \( U_{\mathbb{G}} \) and \( U_{\Lambda} \) are finite
  dimensional unitary representations of \( \mathbb{G} \) and \( \Lambda \)
  respectively on the same Hilbert space \( \mathscr{H} \), if
  \( U_{\mathbb{G}} \) and \( U_{\Lambda} \)
  satisfy~\eqref{eq:1611871ef2500d9a}, then~\eqref{eq:050021bf02f8bfac} defines
  a finite dimensional unitary representation \( U \) of
  \( \mathbb{G} \rtimes \Lambda \) on \( \mathscr{H} \). Moreover,
  \begin{subequations}
    \begin{equation}
      \label{eq:538a1df8eb219db1}
      U_{\mathbb{G}} = (\id_{\mathcal{B}(\mathscr{H})} %
      \otimes \id_{\pol(\mathbb{G})} \otimes \epsilon_{\Lambda})(U) %
      \in \mathcal{B}(\mathscr{H}) \otimes \pol(\mathbb{G}),
    \end{equation}
    \begin{equation}
      \label{eq:ccf0a0e79115abc1}
      U_{\Lambda} %
      = (\id_{\mathcal{B}(\mathscr{H})} \otimes \epsilon_{\mathbb{G}} \otimes
      \id_{C( \Lambda )}) (U) \in \mathcal{B}(\mathscr{H}) \otimes C( \Lambda ).
    \end{equation}
  \end{subequations}

\end{prop}
\begin{proof}
  Let \( d = \dim \mathscr{H} \), and fix a Hilbert basis
  \( (e_{1}, \ldots, e_{d}) \) for \( \mathscr{H} \). Let
  \( (e_{ij}, i, j = 1, \ldots, d) \) be the corresponding matrix units (i.e.\
  \( e_{ij} \in \mathcal{B}(\mathscr{H}) \) is characterized by
  \( e_{ij}(e_{k}) = \delta_{j,k} e_{i} \)). Then there is a unique
  \( U_{ij} \in \pol(\mathbb{G}) \otimes C( \Lambda ) \) for each pair of
  \( i, j \), such that
  \begin{displaymath}
    U = \sum_{i,j} e_{ij} \otimes U_{ij},
  \end{displaymath}
  with each \( U_{ij} \) decomposed further as
  \( U_{ij} = \sum_{r \in \Lambda} U_{ij, r} \otimes \delta_{r} \), where each
  \( U_{ij, r} \in \pol(\mathbb{G}) \).  Since \( U \) is a finite dimensional
  unitary representation of \( \mathbb{G} \rtimes \Lambda \), for any
  \( i,j \in \set*{1,\ldots, d} \), we have
  \begin{equation}
    \label{eq:cc7d223f6265c3fe}
    \Delta_{\mathbb{G} \rtimes \Lambda}(U_{ij}) %
    = \sum_{k=1}^{d} U_{ik} \otimes U_{kj},
  \end{equation}
  where in
  \( \pol(\mathbb{G}) \otimes C( \Lambda ) \otimes \pol(\mathbb{G}) \otimes C(
  \Lambda )\), we have
  \begin{equation}
    \label{eq:d39d6aeeff6fc8ca}
    \begin{split}
      \Delta_{\mathbb{G} \rtimes \Lambda} (U_{ij}) &%
      = \sum_{\substack{r, s, t \in \Lambda, \\ r = st}} %
      {\left[(\id_{A} \otimes \alpha_{s}^{\ast}) \Delta(U_{ij,
            r})\right]}_{13}(1_{A} \otimes
      \delta_{s} \otimes 1_{A} \otimes \delta_{t}) \\
      &= \sum_{s,t \in \Lambda} {\left[(\id_{A} \otimes \alpha_{s}^{\ast})
          \Delta(U_{ij, st})\right]}_{13}(1_{A} \otimes \delta_{s} \otimes 1_{A}
      \otimes \delta_{t})
    \end{split}
  \end{equation}
  and
  \begin{equation}
    \label{eq:f66121ea3fa81543}
    \sum_{k=1}^{d} U_{ik} \otimes U_{kj}
    = \sum_{k=1}^{d} \sum_{s,t \in \Lambda} U_{ik, s} \otimes
    \delta_{s} \otimes U_{kj, t} \otimes \delta_{t}.
  \end{equation}
  Comparing~\eqref{eq:cc7d223f6265c3fe},~\eqref{eq:d39d6aeeff6fc8ca}
  and~\eqref{eq:f66121ea3fa81543}, we get
  \begin{equation}
    \label{eq:8f374b723a1ebd77}
    (\id_{A} \otimes \alpha_{s}^{\ast}) \Delta(U_{ij, st})
    =  \sum_{k=1}^{d} U_{ik,
      s} \otimes U_{kj, t} \in A \otimes A
  \end{equation}
  or equivalently (by applying \( (\id_{A} \otimes \alpha_{s^{-1}}^{\ast}) \) on
  both sides)
  \begin{equation}
    \label{eq:34d7d2460dff9955}
    \Delta(U_{ij, st})
    =  \sum_{k=1}^{d} U_{ik, s} \otimes \alpha_{s^{-1}}^{\ast}
    (U_{kj, t})
  \end{equation}
  for every \( s,t \in \Lambda \).  Since
  \( (\id \otimes \epsilon) \Delta = \id = ( \epsilon \otimes \id ) \Delta \),
  we have
  \begin{equation}
    \label{eq:ac66a15872bdcc9d}
    U_{ij, st} =  \sum_{k=1}^{d} \epsilon(U_{ik,s})
    \alpha_{s^{-1}}^{\ast}(U_{kj,t}) %
    = \sum_{k=1}^{d} \epsilon(U_{kj,s}) U_{ik,t}
  \end{equation}
  for any \( i,j \in \set*{1, \ldots, d} \), \( s,t \in \Lambda \).

  We have \( \epsilon_{\Lambda}( \delta_{r} ) = \delta_{r, 1_{\Lambda}} \), thus
  by definition
  \begin{equation}
    \label{eq:48b8d35e145b4bed}
    U_{\mathbb{G}} = \sum_{i,j = 1}^{d} e_{ij} \otimes U_{ij, 1_{\Lambda}} \in
    \mathcal{B}(\mathscr{H}) \otimes \pol(\mathbb{G}).
  \end{equation}
  Similarly,
  \begin{equation}
    \label{eq:1962e45312618a38}
    U_{\Lambda} = \sum_{r \in \Lambda} \sum_{i,j=1}^{d} %
    \epsilon(U_{ij, r}) e_{ij} \otimes \delta_{r} \in
    \mathcal{B}(\mathscr{H}) \otimes C( \Lambda ).
  \end{equation}
  Thus
  \begin{equation}
    \label{eq:187e572ceb733907}
    U_{\Lambda}(r_{0}) = \sum_{i,j=1}^{d} \epsilon(U_{ij, r_{0}}) e_{ij}
    \in \mathcal{B}(\mathscr{H}).
  \end{equation}
  Hence,
  \begin{equation}
    \label{eq:58ab5f576a30cbb4}
    \begin{split}
      (U_{\Lambda}(r_{0}) \otimes 1_{A}) U_{\mathbb{G}} %
      &= \sum_{i,j,k,l=1}^{d} \delta_{j,k} \epsilon(U_{ij, r_{0}}) %
      e_{il} \otimes U_{kl, 1_{\Lambda}} \\
      &= \sum_{i, l=1}^{d} e_{il} \otimes %
      \sum_{k=1}^{d} \epsilon(U_{ik, r_{0}}) U_{kl, 1_{\Lambda}} \\
      &= \sum_{i, l=1}^{d} e_{il} \otimes %
      \alpha_{r_{0}}^{\ast} \left( \sum_{k=1}^{d} \epsilon(U_{ik, r_{0}}) %
        \alpha_{r_{0}^{-1}}^{\ast}(U_{kl, 1_{\Lambda}}) \right) \\
      &= \sum_{i, l=1}^{d} e_{il} \otimes \alpha_{r_{0}}^{\ast}(U_{il, r_{0}})
    \end{split}
  \end{equation}
  where the last equality follows from~\eqref{eq:ac66a15872bdcc9d}; and
  \begin{equation}
    \label{eq:9323ac3bb551c22c}
    \begin{split}
      [(\id \otimes \alpha_{r_{0}}^{\ast})U_{\mathbb{G}}] %
      (U_{\Lambda}(r_{0}) \otimes 1_{A}) %
      &= \sum_{i,j,k,l=1}^{d} \delta_{j,k} \epsilon(U_{kl, r_{0}}) %
      e_{il} \otimes \alpha_{r_{0}}^{\ast}(U_{ik, 1_{\Lambda}}) \\
      &= \sum_{i,l = 1} e_{il} \otimes \sum_{k=1}^{d} \epsilon(U_{kl, r_{0}}) %
      \alpha_{r_{0}}^{\ast}(U_{ik, 1_{\Lambda}}) \\
      &= \sum_{i,l = 1} e_{il} \otimes \alpha_{r_{0}}^{\ast} %
      \left( \sum_{k=1}^{d} \epsilon(U_{ik, r_{0}})
        U_{kj, 1_{\Lambda}} \right) \\
      &= \sum_{i,l = 1} e_{il} \otimes \alpha_{r_{0}}^{\ast}(U_{il, r_{0}})
    \end{split}
  \end{equation}
  where~\eqref{eq:ac66a15872bdcc9d} is used again in the last equality.

  Combining~\eqref{eq:58ab5f576a30cbb4} and~\eqref{eq:9323ac3bb551c22c} finishes
  the proof of~\eqref{eq:1611871ef2500d9a}.

  By~\eqref{eq:48b8d35e145b4bed},~\eqref{eq:1962e45312618a38}
  and~\eqref{eq:ac66a15872bdcc9d}, one has
  \begin{equation}
    \label{eq:dd48b4afaec2999b}
    \begin{split}
      {(U_{\mathbb{G}})}_{12} {(U_{\Lambda})}_{13} %
      &= \sum_{i,j,k,l=1}^{d} \sum_{r \in \Lambda} \delta_{j,k} %
      \epsilon(U_{kl, r}) e_{il}
      \otimes U_{ij, 1_{\Lambda}} \otimes \delta_{r} \\
      &= \sum_{i,l = 1}^{d} e_{il} \otimes \sum_{r \in \Lambda} %
      \left( \sum_{k=1}^{d} \epsilon(U_{kl, r})U_{ik, 1_{\Lambda}}
      \right) \otimes \delta_{r} \\
      &= \sum_{i,l = 1}^{d} e_{il} \otimes \sum_{r \in \Lambda} U_{il, r}
      \otimes \delta_{r} %
      = U
    \end{split}
  \end{equation}
  in
  \( \mathcal{B}(\mathscr{H}) \otimes \pol(\mathbb{G}) \otimes C( \Lambda )
  \). This proves~\eqref{eq:050021bf02f8bfac}.

  Conversely, suppose \( U_{\mathbb{G}} \) and \( U_{\Lambda} \) are unitary
  representations on some finite dimensional Hilbert space \( \mathscr{H} \).
  We still use \( (e_{1}, \ldots, e_{d}) \) to denote a Hilbert basis for
  \( \mathscr{H} \), where \( d = \dim \mathscr{H} \), and
  \( (e_{ij}, i,j=1,\ldots,d) \) the corresponding matrix unit of
  \( \mathcal{B}(\mathscr{H}) \). Then for each pair \( i, j \), one has a
  unique \( u_{ij} \in \pol(\mathbb{G}) \) and a unique
  \( f_{ij} \in C( \Lambda ) \), such that
  \( U_{\mathbb{G}} = \sum_{i,j} e_{ij} \otimes u_{ij} \),
  \( U_{\Lambda} = \sum_{ij} e_{ij} \otimes f_{ij} \). By suitably choosing the
  basis \( (e_{1}, \ldots, e_{d}) \), we may and do assume
  \( \epsilon(u_{ij}) = \delta_{i,j} \). Since these are representations, we
  have
  \begin{subequations}
    \begin{equation}
      \label{eq:b68c12bf557b2401}
      \Delta(u_{ij}) = \sum_{k=1}^{d} u_{ik} \otimes u_{kj},
    \end{equation}
    \begin{equation}
      \label{eq:6ffec2247b657e85}
      \Delta_{\Lambda}{(f_{ij})} = \sum_{k=1}^{d} f_{ik} \otimes f_{kj}.
    \end{equation}
  \end{subequations}
  By definition,
  \begin{equation}
    \label{eq:4ed86e037b84532d}
    U = \sum_{i,j,k,l=1}^{d} \delta_{jk} e_{il} \otimes u_{ij} \otimes f_{kl} %
    = \sum_{i,j = 1}^{d} e_{ij} \otimes U_{ij}
  \end{equation}
  with
  \begin{equation}
    \label{eq:5500fee400ee124f}
    U_{ij} = \sum_{k=1}^{d} u_{ik} \otimes f_{kj}
    = \sum_{r \in \Lambda} \sum_{k=1}^{d}
    f_{kj}(r) u_{ik} \otimes \delta_{r}.
  \end{equation}
  Since \( U_{\mathbb{G}} \) and \( U_{\Lambda} \) are unitary, so is \( U \).
  Using \( \epsilon(u_{ij}) = \delta_{i,j} \), one has
  \begin{equation}
    \label{eq:6c864f88f58eadf6}
    (\id_{\mathcal{B}(\mathscr{H})}
    \otimes \epsilon \otimes \id_{C( \Lambda )})(U)
    = \sum_{i,j=1}^{d} e_{ij} \otimes \sum_{k=1}^{d} \delta_{i,k} f_{kj}
    = \sum_{i,j}^{d} e_{ij} \otimes f_{ij} = U_{\Lambda}.
  \end{equation}
  This proves~\eqref{eq:ccf0a0e79115abc1}. The proof
  of~\eqref{eq:538a1df8eb219db1} is more involved and must resort to
  condition~\eqref{eq:1611871ef2500d9a}, which using the above notations,
  translates to
  \begin{equation}
    \label{eq:7d6af848b8bde7f3}
    \forall r \in \Lambda, \quad
    \sum_{i,j} e_{ij} \otimes \sum_{k}f_{ik}(r)u_{kj} %
    = \sum_{i,j} e_{ij} \otimes \sum_{k} f_{kj}(r) \alpha_{r}^{\ast}(u_{ik}),
  \end{equation}
  or equivalently,
  \begin{equation}
    \label{eq:16eca2e60cbf97b9}
    \forall r \in \Lambda, i,j \in \set*{1, \ldots, d}, \quad
    \sum_{k=1}^{d} f_{ik}(r) u_{kj}
    = \sum_{k=1}^{d} f_{kj}(r) \alpha_{r}^{\ast}(u_{ik}).
  \end{equation}
  Since \( U_{\Lambda}(1_{\Lambda}) = \id_{\mathscr{H}} \), one has
  \( f_{ij}( 1_{\Lambda} ) = \delta_{i,j} \). Taking \( r = 1_{\Lambda} \)
  in~\eqref{eq:16eca2e60cbf97b9} yields
  \begin{equation}
    \label{eq:abf017c9bd3df263}
    \begin{split}
      & \leadmathskip (\id_{\mathcal{B}(\mathscr{H})} \otimes
      \id_{\pol(\mathbb{G})} \otimes \epsilon_{\Lambda})(U) %
      = \sum_{i,j=1}^{d} e_{ij} \otimes \sum_{k=1}^{d} %
      f_{kj}(1_{\Lambda}) u_{ik} \\
      &= \sum_{i,j=1}^{d} e_{ij} \otimes \sum_{k=1}^{d} \delta_{k,j} u_{ik} %
      = \sum_{i,j=1}^{d} e_{ij} \otimes u_{ij} = U_{\mathbb{G}},
    \end{split}
  \end{equation}
  which proves~\eqref{eq:538a1df8eb219db1}. To finishes the proof of the
  proposition, it remains to check that the unitary \( U \) is indeed a
  representation of \( \mathbb{G} \rtimes \Lambda \).

  Using~\eqref{eq:b68c12bf557b2401},~\eqref{eq:5500fee400ee124f}
  and~\eqref{eq:16eca2e60cbf97b9}, one has
  \begin{equation}
    \label{eq:7f124108a934a3f9}
    \begin{split}
      \Delta_{\mathbb{G} \rtimes \Lambda}(U_{ij}) %
      &= \sum_{r \in \Lambda} \sum_{s \in \Lambda} %
      {\left[(\id_{\pol(\mathbb{G})} \otimes \alpha_{s}^{\ast}) %
          \Delta\left(\sum_{k=1}^{d} f_{kj}(r)u_{ik}\right)\right]}_{13} %
      {\left( \delta_{s} \otimes \delta_{s^{-1}1}\right)}_{24} \\
      &= \sum_{r, s \in \Lambda} \sum_{k,l=1}^{d} f_{kj}(r)u_{il} \otimes
      \delta_{s} %
      \otimes \alpha_{s}^{\ast}(u_{lk}) \otimes \delta_{s^{-1}r} \\
      &= \sum_{r, s \in \Lambda} \sum_{l=1}^{d} u_{il} \otimes \delta_{s} %
      \otimes \left[\sum_{k=1}^{d} f_{kj}(r) \alpha_{s}^{\ast}(u_{lk})\right] %
      \otimes \delta_{s^{-1}r} \\
      &= \sum_{s,t \in \Lambda} \sum_{k,l =1 }^{d} u_{il} \otimes \delta_{s} %
      \otimes \left[f_{kj}(st) \alpha_{s}^{\ast}(u_{lk}) \right] %
      \otimes \delta_{t} \\
      &= \sum_{s,t \in \Lambda} \sum_{k,l =1 }^{d} u_{il} \otimes \delta_{s} %
      \otimes \left[\sum_{h=1}^{d} f_{hj}(t)f_{kh}(s)
        \alpha_{s}^{\ast}(u_{lk})\right]
      \otimes \delta_{t} \\
      &= \sum_{s,t \in \Lambda} \sum_{h,l =1 }^{d} u_{il} \otimes \delta_{s} %
      \otimes \left[f_{hj}(t) \sum_{k=1}^{d} f_{kh}(s)
        \alpha_{s}^{\ast}(u_{lk})\right]
      \otimes \delta_{t} \\
      &= \sum_{s,t \in \Lambda} \sum_{h,l =1 }^{d} u_{il} \otimes \delta_{s} %
      \otimes \left[f_{hj}(t) \sum_{k=1}^{d} f_{lk}(s) u_{kh} \right]
      \otimes \delta_{t} \\
      &= \sum_{h,k,l=1}^{d} u_{il} \otimes f_{lk} %
      \otimes u_{kh} \otimes f_{hj} \\
      &= \sum_{h=1}^{d} \left( \sum_{l=1}^{d} u_{il} \otimes f_{lk}\right) %
      \otimes \left( \sum_{k=1}^{d} u_{kh} \otimes f_{hj} \right) %
      = \sum_{h=1}^{d}U_{ih} \otimes U_{hj}.
    \end{split}
  \end{equation}
  Thus \( U \) is indeed a (unitary) representation.
\end{proof}

\begin{defi}
  \label{defi:8e2ca3d66e4a106c}
  Let \( U_{\mathbb{G}} \in \mathcal{B}(\mathscr{H}) \otimes \pol(\mathbb{G}) \)
  be a finite dimensional unitary representation of \( \mathbb{G} \),
  \( U_{\Lambda} \in \mathcal{B}(\mathscr{H}) \otimes C( \Lambda ) \) a finite
  dimensional unitary representation of \( \Lambda \) on the same space
  \( \mathscr{H} \), we say \( U_{\mathbb{G}} \) and \( U_{\Lambda} \) are
  \textbf{covariant} if they satisfy condition~\eqref{eq:1611871ef2500d9a}.
\end{defi}

We track here a simple criterion for two representations to be covariant using
matrix units and matrix coefficients.
\begin{prop}
  \label{prop:38cc8753970bd4b5}
  Let
  \( U_{\mathbb{G}} \in \mathcal{B}(\mathscr{H}) \otimes \pol(\mathbb{G}) \),
  \( U_{\Lambda} \in \mathcal{B}(\mathscr{H}) \otimes C( \Lambda ) \) be
  finite-dimensional unitary representations of \( \mathbb{G} \) and
  \( \Lambda \) respectively. Let \( (e_{1}, \ldots, e_{d}) \) be a Hilbert
  basis of \( \mathscr{H} \), \( e_{ij} \in \mathcal{B}(\mathscr{H}) \) the
  operator with \( e_{ij}(e_{k}) = \delta_{j,k} e_{i} \), and
  \( U_{\mathbb{G}} = \sum_{i,j} e_{ij} \otimes u_{ij} \),
  \( U_{\Lambda} = \sum_{i,j} e_{ij} \otimes f_{ij} \), then
  \( U_{\mathbb{G}} \) and \( U_{\Lambda} \) are covariant if and only if
  \begin{equation}
    \label{eq:c97618b75ad566e3}
    \forall r \in \Lambda, i,j \in \set*{1, \ldots, d}, \quad
    \sum_{k=1}^{d} f_{ik}(r) u_{kj} = \sum_{k=1}^{d} f_{kj}(r)
    \alpha_{r}^{\ast}(u_{ik}).
  \end{equation}
\end{prop}
\begin{proof}
  This is just a restatement of condition~\eqref{eq:1611871ef2500d9a}.
\end{proof}

By Proposition~\ref{prop:3693cca5392c7dd3}, unitary representations of
\( \mathbb{G} \rtimes \Lambda \), at least the finite dimensional ones,
correspond bijectively to pairs of covariant unitary representations of
\( \mathbb{G} \) and \( \Lambda \).

\section{Principal subgroups of \texorpdfstring{\( \mathbb{G} \rtimes \Lambda \)
  }{G x| Lambda}}
\label{sec:90f8ecd4790ea7ac}

\begin{defi}
  \label{defi:2f7e12385791f036}
  Let \( \mathbb{H} = (B, \Delta_{B}) \), \( \mathbb{K} = (C, \Delta_{C}) \) be
  compact quantum groups, we say \( \mathbb{K} \) is isomorphic to a closed
  quantum subgroup of \( \mathbb{H} \), or simply \( \mathbb{K} \) is a closed
  subgroup of \( \mathbb{H} \), if there exists a \emph{surjective} mapping
  \( \varphi : \pol(\mathbb{H}) \to \pol(\mathbb{K}) \) such that \( \varphi \)
  is a morphism of Hopf \( \ast \)\nobreakdash-algebras.
\end{defi}

When \( \mathcal{\mathbb{H}} \) is universal, then
Definition~\ref{defi:2f7e12385791f036} can be reformulated as the existence of a
surjective unital \( C^{\ast} \)\nobreakdash-algebra morphism
\( \varphi : B \to C \) such that
\( (\varphi \otimes \varphi) \Delta_{B} = \Delta_{C} \varphi \).

In the context of compact quantum groups, we will use the terms ``quantum closed
subgroup'' and ``closed subgroup'', sometimes even ``subgroup'', interchangeably
without further explanation.

\begin{rema}
  \label{rema:a1a33b079ed90648}
  If \( \mathbb{H} \) and \( \mathbb{K} \) are commutative, i.e.\ they come from
  genuine compact groups, then \( \mathbb{K} \) being isomorphic to a closed
  subgroup, says exactly that there exists a continuous injective map
  \( \varphi_{\ast} \) from \( \operatorname{Spec}(C) \), the underlying space
  of the compact group \( \mathbb{K} \), into \( \operatorname{Spec}(B) \), the
  underlying space of the compact group \( \mathbb{H} \), such that
  \( \varphi_{\ast} \) preserves multiplication. Thus the above definition for
  closed (quantum) subgroups is consistent with the classical case of compact
  groups.
\end{rema}

Recall that \( \mathbb{G} = (A, \Delta) \),
\( C(\mathbb{G} \rtimes \Lambda) = A \otimes C( \Lambda ) \), and
\( \pol(\mathbb{G} \rtimes \Lambda) = C(\mathbb{G}) \rtimes C(\Lambda) \).
\begin{prop}
  \label{prop:e672494fc51990cf}
  Let \( \Lambda_{0} \) be a subgroup of \( \Lambda \), then the mapping
  \begin{align*}
    \varphi \colon A \otimes C( \Lambda )
    & \rightarrow A \otimes C( \Lambda_{0}) \\
    \sum_{r \in \Lambda} a_{r} \otimes \delta_{r}
    &\mapsto \sum_{r \in \Lambda_{0}} a_{r} \otimes \delta_{r}
  \end{align*}
  is a unital surjective morphism\footnote{Note that \( \delta_{r} \) has
    different meanings when viewed as functions in \( C( \Lambda ) \) and in
    \( C( \Lambda_{0} ) \)} of \( C^{\ast} \)\nobreakdash-algebras that also
  intertwines the comultiplications on \( \mathbb{G} \rtimes \Lambda_{0} \) and
  \( \mathbb{G} \rtimes \Lambda \). In particular,
  \( \mathbb{G} \rtimes \Lambda_{0} \) is a closed subgroup of
  \( \mathbb{G} \rtimes \Lambda \).
\end{prop}
\begin{proof}
  Obviously \( \varphi \) is a unital surjective morphism of
  {\( C^{\ast} \)}\nobreakdash-algebras. We need to show that \( \varphi \)
  intertwines the comultiplication \( \widetilde{\Delta} \) on
  \( \mathbb{G} \rtimes \Lambda \) and the comultiplication
  \( \widetilde{\Delta}_{0} \) on \( \mathbb{G} \rtimes \Lambda_{0} \). For
  this, by density, it suffices to prove that the restriction
  \begin{equation}
    \label{eq:358657f5787e0893}
    \begin{split}
      \varphi \colon \pol(\mathbb{G}) \otimes C( \Lambda ) & \rightarrow
      \pol(\mathbb{G}) \otimes C(
      \Lambda_{0} ) \\
      \sum_{r \in \Lambda} a_{r} \otimes \delta_{r} & \mapsto \sum_{r \in
        \Lambda_{0}} a_{r} \otimes \delta_{r}
    \end{split}
  \end{equation}
  intertwines the comultiplications. Indeed, given an arbitrary
  \( a_{r} \in \pol(G) \) for any \( r \in \Lambda \), note that for any
  \( a \in \pol(\mathbb{G}) \) and \( \lambda \in \Lambda \),
  \( \varphi(a \otimes \delta_{\lambda}) = 0 \) whenever
  \( \lambda \notin \Lambda_{0} \), we have
  \begin{equation}
    \label{eq:14d5653b2f9f560e}
    \begin{split}
      & \leadmathskip ( \varphi \otimes \varphi ) \widetilde{\Delta} \left(
        \sum_{r \in \Lambda} a_{r} \otimes \delta_{r}
      \right) \\
      &= ( \varphi \otimes \varphi ) \sum_{r \in \Lambda} \sum_{s \in \Lambda}
      \sum {(a_{r})}_{(1)} \otimes \delta_{s} \otimes
      \alpha_{s}^{\ast}\bigl({(a_{r})}_{(2)}\bigr) \otimes \delta_{s^{-1} r} \\
      &= \sum_{r \in \Lambda_{0}} \sum_{s \in \Lambda_{0}} \sum {(a_{r})}_{(1)}
      \otimes \delta_{s} \otimes
      \alpha_{s}^{\ast}\bigl({(a_{r})}_{(2)}\bigr) \otimes \delta_{s^{-1} r} \\
      & \text{(Since \( s, s^{-1}r \in \Lambda_{0} \) implies
        \( r = s(s^{-1}r) \in \Lambda_{0}\))} \\
      &= \widetilde{\Delta}_{0} \left( \sum_{r \in \Lambda_{0}} a_{r} \otimes
        \delta_{r} \right) = \widetilde{\Delta}_{0} \varphi \left( \sum_{r \in
          \Lambda} a_{r} \otimes \delta_{r} \right).
    \end{split}
  \end{equation}
  This shows that \( \varphi \) indeed intertwines comultiplications and
  finishes the proof.
\end{proof}

\begin{defi}
  \label{defi:f8208962169ffe31}
  A closed subgroup of \( \mathbb{G} \rtimes \Lambda \) of the form
  \( \mathbb{G} \rtimes \Lambda_{0} \), where \( \Lambda_{0} \) is a subgroup
  \( \Lambda \), is called a \textbf{principal subgroup} of
  \( \mathbb{G} \rtimes \Lambda \).
\end{defi}

\begin{rema}
  \label{rema:a255d16de9b29654}
  If we let \( p_{0} = \sum_{r \in \Lambda_{0}} \delta_{r} \in C( \Lambda ) \),
  then \( p_{0} \) is a projection in \( C( \Lambda ) \), thus
  \( 1 \otimes p_{0} \) is a central projection in \( A \otimes C( \Lambda ) \).
  The morphism \( \varphi \) is in fact given by the ``compression'' map
  \( (1 \otimes p_{0})( \cdot )(1 \otimes p_{0}) \). Essentially, these data
  says that the principal subgroup \( \mathbb{G} \rtimes \Lambda_{0} \) is in
  fact an \textbf{open} subgroup of \( \mathbb{G} \rtimes \Lambda \). As we
  don't really need the general theory of open subgroups of topological quantum
  groups in this paper, we won't recall the relevant notions here and refer
  the interested reader to the articles~\cite{MR2980506, MR3552528} for a
  treatment in the more general setting of locally compact quantum groups.
\end{rema}

\begin{coro}
  \label{coro:75db6f2d370caba0}
  Using the notations in Proposition~\ref{prop:e672494fc51990cf}, if
  \( U \in \mathcal{B}(\mathscr{H}) \otimes A \otimes C( \Lambda ) \) is a
  (unitary) representation of \( \mathbb{G} \rtimes \Lambda \), then
  \( (\id \otimes \varphi)(U) \) is a (unitary) representation of
  \( \mathbb{G} \rtimes \Lambda_{0} \).
\end{coro}
\begin{proof}
  This follows directly from the fact that the restriction of the mapping
  \( \varphi \) as specified in~\eqref{eq:358657f5787e0893} is a morphism of
  Hopf \( \ast \)-algebras.
\end{proof}

\begin{defi}
  \label{defi:3cf91dacf8289790}
  Using the above notations, the representation \( (\id \otimes \varphi)(U) \)
  is called the restriction of \( U \) to \( \mathbb{G} \rtimes \Lambda_{0} \),
  and is denoted by \( U \vert_{\mathbb{G} \rtimes \Lambda_{0}} \).
\end{defi}

\begin{rema}
  \label{rema:bd179f3ea46227d4}
  Again, when \( \mathbb{G} \) is an classical compact group \( G \), we recover
  the classical notion of restriction of a representation of
  \( G \rtimes \Lambda \) to the subgroup \( G \rtimes \Lambda_{0} \).
\end{rema}

There is a natural ``conjugate'' relation between principal subgroups of the
form \( \mathbb{G} \rtimes \Lambda_{0} \) where \( \Lambda_{0} \) is a subgroup
of \( \Lambda \), which will be used to simplify some calculations in our later
treatment of representations. This relation is described in the following
proposition.

\begin{prop}
  \label{prop:e4e616da0a0995e6}
  Let \( \Lambda_{0} \) be a subgroup of \( \Lambda \), \( r \in \Lambda \),
  \( \adj_{r} \colon \Lambda_{0} \rightarrow r \Lambda_{0} r^{-1} \) the
  isomorphism \( s \mapsto r s r^{-1} \). Then
  \( \alpha_{r}^{\ast} \otimes \adj_{r}^{\ast} \) is an isomorphism of compact
  quantum groups from \( \mathbb{G} \rtimes \Lambda_{0} \) to
  \( \mathbb{G} \rtimes r \Lambda_{0} r^{-1} \).
\end{prop}
\begin{proof}
  By density, it suffices to prove that the unital \( \ast \)-isomorphism
  \begin{displaymath}
    \alpha_{r}^{\ast} \otimes \adj_{r}^{\ast} \colon \pol(\mathbb{G}) %
    \otimes C( r \Lambda_{0} r^{-1} ) %
    \rightarrow \pol(\mathbb{G}) \otimes C(\Lambda_{0})
  \end{displaymath}
  of involutive algebras preserves comultiplication. To fix the notations, let
  \( \Delta_{0} \) (resp.\ \( \Delta_{r} \)) be the comultiplication on
  \( \pol(\mathbb{G}) \otimes C(\Lambda_{0}) \) (resp.\
  \( \pol(\mathbb{G}) \otimes C(r \Lambda_{0} r^{-1}) \)).  For any
  \( x \in \pol(\mathbb{G}) \), \( \lambda \in \Lambda_{0} \), we have
  \begin{equation}
    \label{eq:2a2413697f1f2ace}
    \begin{split}
      & \leadmathskip %
      (\alpha_{r}^{\ast} \otimes \adj_{r}^{\ast} %
      \otimes \alpha_{r}^{\ast} \otimes \adj_{r}^{\ast}) %
      \Delta_{r}(x \otimes \delta_{r \lambda r^{-1}}) \\
      &= \sum_{\mu \in \Lambda_{0}} {\bigl[ %
        \bigl( %
        \alpha_{r}^{\ast} \otimes (\alpha_{r}^{\ast}\alpha_{r\mu r^{-1}}^{\ast})
        \bigr) \Delta(x) %
        \bigr]}_{13} %
      {((\adj_{r}^{\ast}\delta_{r\mu r^{-1}}) \otimes %
        (\adj_{r}^{\ast}\delta_{r \mu^{-1}\lambda r^{-1}}))}_{24} \\
      &= \sum_{\mu \in \Lambda_{0}} {\bigl[ %
        \bigl( %
        \alpha_{r}^{\ast} \otimes \alpha_{r\mu}^{\ast} \bigr) \Delta(x) %
        \bigr]}_{13} %
      {((\adj_{r}^{\ast}\delta_{r\mu r^{-1}}) \otimes %
        (\adj_{r}^{\ast}\delta_{r \mu^{-1}\lambda r^{-1}}))}_{24} \\
      &= \sum_{\mu \in \Lambda_{0}} {\bigl[ %
        \bigl( %
        (\id \otimes \alpha_{\mu}^{\ast}) %
        [(\alpha_{r}^{\ast} \otimes \alpha_{r}^{\ast}) \Delta(x)] %
        \bigr) %
        \bigr]}_{13} %
      {(\delta_{\mu} \otimes \delta_{\mu^{-1}\lambda})}_{24} \\
      &= \sum_{\mu \in \Lambda_{0}} {\bigl[ %
        \bigl( %
        (\id \otimes \alpha_{\mu}^{\ast}) %
        \Delta(\alpha_{r}^{\ast}(x)) %
        \bigr) %
        \bigr]}_{13} %
      {(\delta_{\mu} \otimes \delta_{\mu^{-1}\lambda})}_{24} \\
      &= \Delta_{0} \bigl(\alpha_{r}^{\ast}(x) \otimes \delta_{\lambda}\bigr) =
      [\Delta_{0}(\alpha_{r}^{\ast} \otimes \adj_{r}^{\ast})] %
      (x \otimes \delta_{r\lambda r^{-1}}).
    \end{split}
  \end{equation}
  Thus \( \alpha_{r}^{\ast} \otimes \adj_{r}^{\ast} \) indeed preserves
  comultiplication.
\end{proof}

\section{Induced representations of principal subgroups}
\label{sec:3e39eff392b139a7}

We begin by describing an outline of our approach to induced representations of
principal subgroups of \( \mathbb{G} \rtimes \Lambda \). Let \( \Lambda_{0} \)
be a subgroup of \( \Lambda \),
\( U \in \mathcal{B}(\mathscr{H}) \otimes \pol(\mathbb{G}) \otimes C(
\Lambda_{0} ) \) a finite dimensional unitary representation of
\( \mathbb{G} \rtimes \Lambda_{0} \). We want to construct the induced
representation
\( \indrep^{\mathbb{G} \rtimes \Lambda}_{\mathbb{G} \rtimes \Lambda_{0}}(U) \)
of the larger quantum group \( \mathbb{G} \rtimes \Lambda \).  The idea of the
construction goes as follows: by the results in \S~\ref{sec:cac0f6b0354689f8},
we know \( U \) is determined by its restrictions
\( U_{\mathbb{G}} = \res_{\mathbb{G}}(U) \) and
\( U_{\Lambda_{0}} = \res_{\Lambda_{0}}(U) \). While one may not be able to
directly extend the representation \( U_{\Lambda_{0}} \) of \( \Lambda_{0} \) to
a representation of \( \Lambda \) on the same space \( \mathscr{H} \), we do
have the right-regular representation \( \widetilde{W}_{\Lambda} \) of
\( \Lambda \) on \( \ell^{2}(\Lambda) \otimes \mathscr{H} \) using the group
structure of \( \Lambda \). On the other hand, the direct sum
\( \widetilde{W}_{\mathbb{G}} \) of various copies of \( U_{\mathbb{G}} \)
placed suitably in \( \ell^{2}(\Lambda) \otimes \mathscr{H} \) will give a
representation of \( \mathbb{G} \) on
\( \ell^{2}(\Lambda) \otimes \mathscr{H} \). It is then easy to check that
\( \widetilde{W}_{\mathbb{G}} \) and \( \widetilde{W}_{\Lambda} \) are
covariant, thus determine a representation \( \widetilde{W} \) of
\( \mathbb{G} \rtimes \Lambda \) on \( \ell^{2}(\Lambda) \otimes \mathscr{H}
\). To retrieve the information of \( U_{\Lambda_{0}} \), which is implicitly
encoded in the \( \mathscr{H} \) factor of
\( \ell^{2}(\Lambda) \otimes \mathscr{H} \), we consider the subspace
\( \mathscr{K} \) of \( \ell^{2}(\Lambda) \otimes \mathscr{H} \) consisting of
vectors which behave in a covariant way with the representation
\( U_{\Lambda_{0}} \) on \( \mathscr{H} \). More precisely, \( \mathscr{K} \) is
defined by
\begin{equation}
  \label{eq:5d413640e021805b}
  \mathscr{K} = \set*{\sum_{r \in \Lambda} \delta_{r} \otimes \xi_{r} %
    \given \forall r_{0} \in \Lambda_{0}, \forall r \in \Lambda,
    \, \xi_{r_{0}r} = U_{\Lambda_{0}}(r_{0}) \xi_{r}}.
\end{equation}
One checks that \( \mathscr{K} \) is an invariant subspace for both
\( \widetilde{W}_{\Lambda} \) and \( \widetilde{W}_{\mathbb{G}} \), hence
\( \mathscr{K} \) is a subrepresentation \( W \) of \( \widetilde{W} \), and we
define \( W \) to be the induced representation \( \indrep(U) \). We now proceed
to carry out this idea precisely.

\begin{defi}
  \label{defi:f314227fa95c43c3}
  Let \( U \), \( \mathscr{H} \), \( \Lambda_{0} \) retain their meanings as
  above, and let \( (e_{r,s}; \, r, s \in \Lambda) \) be the matrix unit of
  \( \mathcal{B}(\ell^{2}(\Lambda)) \) associated with the standard Hilbert
  basis \( (\delta_{r}; \, r \in \Lambda) \) for \( \ell^{2}(\Lambda) \), i.e.\
  \( e_{r, s} \delta_{t} = \delta_{s, t} \delta_{r} \) for all
  \( r, s, t \in \Lambda \). The right regular representation
  \( \widetilde{W}_{\Lambda} \) of \( \Lambda \) on
  \( \ell^{2} \otimes \mathscr{H} \) is an operator in
  \( \mathcal{B}(\ell^{2}(\Lambda)) \otimes \mathcal{B}(\mathscr{H}) \otimes
  C(\Lambda) \) defined by
  \begin{equation}
    \label{eq:93f09a241bf0115a}
    \widetilde{W}_{\Lambda} = \sum_{r, s \in \Lambda} e_{rs^{-1}, r} \otimes
    \id_{\mathscr{H}} \otimes \delta_{s}.
  \end{equation}
\end{defi}

It is easy to see that if we regard \( \ell^{2}(\Lambda) \otimes \mathscr{H} \)
as \( \ell^{2}(\Lambda, \mathscr{H}) \), then for any \( s \in \Lambda \),
\( \widetilde{W}_{\Lambda}(s) \) is the operator in
\( \mathcal{B}(\ell^{2}(\Lambda, \mathscr{H})) \) sending each
\( F \colon \Lambda \rightarrow \mathscr{H} \) to \( F \circ R_{s} \), where
\( R_{s} \colon \Lambda \rightarrow \Lambda \) is the right multiplication by
\( s \). Hence \( \widetilde{W}_{\Lambda} \) is indeed a unitary representation
of \( \Lambda \) on \( \ell^{2}(\Lambda) \otimes \mathscr{H} \). By definition,
for any \( s \in \Lambda \), the unitary operator
\( \widetilde{W}_{\Lambda}(s) \in \mathcal{U}(\ell^{2}(\Lambda) \otimes
\mathscr{H}) \) is characterized by
\begin{equation}
  \label{eq:15bcfeb39da18428}
  \begin{split}
    \widetilde{W}_{\Lambda}(s) \colon \ell^{2}(\Lambda) \otimes \mathscr{H} %
    & \rightarrow \ell^{2}(\Lambda) \otimes \mathscr{H} \\
    \delta_{r} \otimes \xi & \mapsto \delta_{rs^{-1}} \otimes \xi,
  \end{split}
\end{equation}
or equivalently
\begin{equation}
  \label{eq:210cee9951f8d9f5}
  \begin{split}
    \widetilde{W}_{\Lambda}(s) \colon \ell^{2}(\Lambda) \otimes \mathscr{H} %
    & \rightarrow \ell^{2}(\Lambda) \otimes \mathscr{H} \\
    \sum_{r \in \Lambda} \delta_{r} \otimes \xi_{r} & \mapsto \sum_{r \in
      \Lambda} \delta_{rs^{-1}} \otimes \xi_{r} %
    = \sum_{r \in \Lambda} \delta_{r} \otimes \xi_{rs}.
  \end{split}
\end{equation}

\begin{prop}
  \label{prop:18c6664e737a486a}
  Using the above notations, the unitary operator
  \begin{equation}
    \label{eq:b6a527744a7eb0a7}
    \widetilde{W}_{\mathbb{G}} := %
    \sum_{s \in \Lambda} e_{s,s} %
    \otimes [(\id \otimes \alpha^{\ast}_{s})(U_{\mathbb{G}})] %
    \in \mathcal{B}(\ell^{2}(\Lambda)) \otimes %
    \mathcal{B}(\mathscr{H}) \otimes \pol(\mathbb{G})
  \end{equation}
  is a unitary representation of \( \mathbb{G} \) on
  \( \ell^{2}(\Lambda) \otimes \mathscr{H} \). Furthermore, for every
  \( s \in \Lambda \), \( \delta_{s} \otimes \mathscr{H} \) is invariant under
  \( \widetilde{W}_{\mathbb{G}} \), and the subrepresentation
  \( \delta_{s} \otimes \mathscr{H} \) of \( \widetilde{W}_{\mathbb{G}} \) is
  unitarily equivalent to the unitary representation
  \( (\id \otimes \alpha^{\ast}_{s})(U_{\mathbb{G}}) \) of \( \mathbb{G} \). In
  particular,
  \( \widetilde{W}_{\mathbb{G}} \simeq \bigoplus_{s \in \Lambda} (\id \otimes
  \alpha_{s}^{\ast})(U_{\mathbb{G}}) \).
\end{prop}
\begin{proof}
  For each \( s \in \Lambda \), since
  \( \alpha^{\ast}_{s} \in \aut\bigl( C(\mathbb{G}), \Delta \bigr) \), the
  unitary operator \( (\id \otimes \alpha^{\ast}_{s})(U_{\mathbb{G}}) \) is
  indeed a representation of \( \mathbb{G} \) on \( \mathscr{H} \). It is easy
  to see that
  \begin{displaymath}
    e_{s,s}(\ell^{2}(\Lambda)) \otimes \mathscr{H}
    = \mathbb{C} \delta_{s} \otimes \mathscr{H} =
    \delta_{s} \otimes \mathscr{H},
  \end{displaymath}
  hence \( e_{s,s} \otimes \id_{\mathscr{H}} \) is the orthogonal projection in
  \( \mathcal{B}(\ell^{2}(\Lambda) \otimes \mathscr{H}) \) onto the subspace
  \( \delta_{s} \otimes \mathscr{H} \) of
  \( \ell^{2}(\Lambda) \otimes \mathscr{H} \) (Here and below, we abuse the
  notation \( \delta_{s} \otimes \mathscr{H} \) to denote the subspace
  \( \set{\delta_{s} \otimes \xi \given \xi \in \mathscr{H}} \) of
  \( \ell^{2}(\Lambda) \otimes \mathscr{H} \)). We also have the intertwining
  relation
  \begin{equation}
    \label{eq:4774a5c9ef074e08}
    (e_{s, s} \otimes \id_{\mathscr{H}} \otimes 1_{A})
    \widetilde{W}_{\mathbb{G}} %
    = e_{s, s} \otimes [(\id \otimes \alpha^{\ast}_{s})(U_{\mathbb{G}})]
    = \widetilde{W}_{\mathbb{G}}
    (e_{s, s} \otimes \id_{\mathscr{H}} \otimes 1_{A}).
  \end{equation}
  Now the theorem follows from \eqref{eq:4774a5c9ef074e08}, the direct sum
  decomposition
  \begin{equation}
    \label{eq:a2cad66555bfb966}
    \ell^{2}(\Lambda) \otimes \mathscr{H} %
    = \bigoplus_{s \in \Lambda} e_{s, s}(\ell^{2}(\Lambda)) \otimes
    \mathscr{H} = \bigoplus_{s \in \Lambda} \delta_{s} \otimes \mathscr{H},
  \end{equation}
  and the obvious fact that the unitary operator
  \( \delta_{s} \otimes \mathscr{H} \to \mathscr{H} \),
  \( \delta_{s} \otimes \xi \mapsto \xi \) intertwines the representation
  \( (\id \otimes \delta_{s})(U_{\mathbb{G}}) \) and the subrepresentation of
  \( \widetilde{W}_{\mathbb{G}} \) determined by the subspace
  \( \delta_{s} \otimes \mathscr{H} \) of
  \( \ell^{2}(\Lambda) \otimes \mathscr{H} \).
\end{proof}

\begin{prop}
  \label{prop:11a935d2017b9167}
  The representations \( \widetilde{W}_{\mathbb{G}} \) and
  \( \widetilde{W}_{\Lambda} \) are covariant.
\end{prop}
\begin{proof}
  For any \( s \in \Lambda \), by definition,
  \begin{equation}
    \label{eq:62dcdb5142f576cb}
    \widetilde{W}_{\Lambda}(s) = \sum_{r \in \Lambda} e_{rs^{-1}, r} \otimes
    \id_{\mathscr{H}} %
    \in \mathcal{B}(\ell^{2}(\Lambda)) \otimes \mathcal{B}(\mathscr{H})
    = \mathcal{B}(\ell^{2}(\Lambda) \otimes \mathscr{H}).
  \end{equation}
  Thus
  \begin{equation}
    \label{eq:ecb4a7429be89210}
    \begin{split}
      & \leadmathskip \left(\widetilde{W}_{\Lambda}(s) \otimes 1\right)
      \widetilde{W}_{\mathbb{G}} \\
      &= \left( \sum_{r \in \Lambda} e_{rs^{-1}, r} %
        \otimes \id_{\mathscr{H}} \otimes 1_{A} \right) %
      \sum_{t \in \Lambda} e_{t, t} \otimes [(\id \otimes
      \alpha^{\ast}_{t})(U_{\mathbb{G}})]
      \\
      &= \sum_{r, t \in \Lambda} \delta_{r, t} e_{rs^{-1}, t} \otimes [(\id
      \otimes \alpha^{\ast}_{t})(U_{\mathbb{G}})] %
      = \sum_{r \in \Lambda} e_{rs^{-1}, r} \otimes [(\id \otimes
      \alpha^{\ast}_{r})(U_{\mathbb{G}})] \\
      &= (\id \otimes \id \otimes \alpha^{\ast}_{s}) \sum_{r \in \Lambda}
      e_{rs^{-1}, r} \otimes [(\id \otimes
      \alpha^{\ast}_{rs^{-1}})(U_{\mathbb{G}})] \\
      &= (\id \otimes \id \otimes \alpha^{\ast}_{s}) \left[\left(\sum_{t \in
            \Lambda} e_{t, t} \otimes [(\id_{\mathscr{H}} \otimes
          \alpha^{\ast}_{t})(U_{\mathbb{G}})]\right) %
        \left(%
          \sum_{r \in \Lambda} e_{rs^{-1}, r} \otimes \id_{\mathscr{H}} \otimes
          1_{A} \right)
      \right] \\
      &= [(\id \otimes \id \otimes
      \alpha^{\ast}_{s})(\widetilde{W}_{\mathbb{G}})]
      \bigl(\widetilde{W}_{\Lambda}(s) \otimes 1\bigr).
    \end{split}
  \end{equation}
  This proves that \( \widetilde{W}_{\mathbb{G}} \) and
  \( \widetilde{W}_{\Lambda} \) are indeed covariant.
\end{proof}

\begin{coro}
  \label{coro:a2b4b06e5632471c}
  The unitary operator
  \begin{equation}
    \label{eq:90563ed2f5a65ac2}
    \begin{split}
      \widetilde{W} : &=
      {(\widetilde{W}_{\mathbb{G}})}_{123}{(\widetilde{W}_{\Lambda})}_{124} =
      \sum_{r,s,t \in \Lambda} e_{t,t} e_{rs^{-1}, r} \otimes [(\id \otimes
      \alpha^{\ast}_{t})(U_{\mathbb{G}})] \otimes \delta_{s} \\
      &= \sum_{r, s \in \Lambda} e_{rs^{-1}, r} \otimes [(\id \otimes
      \alpha^{\ast}_{rs^{-1}})(U_{\mathbb{G}})] \otimes \delta_{s} \\
      &\in \mathcal{B}(\ell^{2}(\Lambda)) \otimes \mathcal{B}(\mathscr{H})
      \otimes \pol(\mathbb{G}) \otimes C(\Lambda)
    \end{split}
  \end{equation}
  is a representation of \( \mathbb{G} \rtimes \Lambda \) on
  \( \ell^{2}(\Lambda) \otimes \mathscr{H} \).
\end{coro}
\begin{proof}
  This follows from Proposition~\ref{prop:3693cca5392c7dd3} and
  Proposition~\ref{prop:11a935d2017b9167}.
\end{proof}

We now proceed to prove the invariance of the subspace \( \mathscr{K} \) defined
in \eqref{eq:5d413640e021805b} under \( \widetilde{W}_{\mathbb{G}} \) and
\( \widetilde{W}_{\Lambda} \).

\begin{lemm}
  \label{lemm:8636ddbbcbafe76a}
  Using the above notations, the following hold:
  \begin{enumerate}
  \item \label{item:2ce0298769aca408} the orthogonal projection
    \( \pi \in \mathcal{B}( \ell^{2}( \Lambda ) \otimes \mathscr{H}) \) with
    range \( \mathscr{K} \) is given by\footnote{Recall that we've identified
      \( \mathcal{B}( \ell^{2}( \Lambda ) \otimes \mathscr{H} ) \) with
      \( \mathcal{B}( \ell^{2}( \Lambda )) \otimes \mathcal{B}( \mathscr{ H})
      \)} the following formula:
    \begin{equation}
      \label{eq:1cfcb4b787ac3478}
      \begin{split}
        \pi \colon \ell^{2}( \Lambda ) \otimes \mathscr{H} & \rightarrow
        \ell^{2}( \Lambda ) \otimes \mathscr{H}
        \\
        \delta_{r} \otimes \xi & \mapsto \abs*{\Lambda_{0}}^{-1} \sum_{r_{0} \in
          \Lambda_{0}} \delta_{r_{0} r} \otimes U_{\Lambda_{0}}(r_{0}) \xi.
      \end{split}
    \end{equation}
    In other words,
    \begin{equation}
      \label{eq:3c57a5fc33525cfd}
      \pi = \abs*{\Lambda_{0}}^{-1}
      \sum_{r_{0} \in \Lambda_{0}} \sum_{s \in \Lambda}
      e_{r_{0}s, s} \otimes U_{\Lambda_{0}}(r_{0});
    \end{equation}
  \item \label{item:4a8b15b53e251b8b} \( \mathscr{K} \) is invariant under both
    \( \widetilde{W}_{\mathbb{G}} \) and \( \widetilde{W}_{\Lambda} \), i.e.\
    \begin{subequations}
      \begin{equation}
        \label{eq:0d4d16dfcc1731dc}
        (\pi \otimes 1) \widetilde{W}_{\mathbb{G}} %
        = \widetilde{W}_{\mathbb{G}} (\pi \otimes 1) %
        = (\pi \otimes 1) \widetilde{W}_{\mathbb{G}} (\pi \otimes 1),
      \end{equation}
      \begin{equation}
        \label{eq:31ef088a6c0e308b}
        (\pi \otimes 1) \widetilde{W}_{\Lambda}
        = \widetilde{W}_{\Lambda}(\pi \otimes 1) = (\pi \otimes 1)
        \widetilde{W}_{\Lambda}(\pi \otimes 1).
      \end{equation}
    \end{subequations}
    In particular, we have
    \begin{equation}
      \label{eq:6b8a2b0066da754a}
      (\pi \otimes 1 \otimes 1) \widetilde{W}
      = \widetilde{W}(\pi \otimes 1 \otimes 1) %
      = (\pi \otimes 1 \otimes 1) \widetilde{W} (\pi \otimes 1 \otimes 1).
    \end{equation}
  \end{enumerate}
\end{lemm}
\begin{proof}
  It is easy to see that \( \pi( \ell^{2}( \Lambda ) \otimes \mathscr{H} ) \) is
  precisely \( \mathscr{K} \) and \( \pi_{\mathscr{K}} = \id_{\mathscr{K}} \).
  To finish the proof of \ref{item:2ce0298769aca408}, it suffices to check that
  \( \pi \) is self-adjoint (or even stronger, positive). Since
  \begin{equation}
    \label{eq:cd748833a9aa7751}
    \begin{split}
      \prosca*{ \pi ( \delta_{r} \otimes \xi_{r} ), \delta_{r} \otimes
        \xi_{r}} %
      &= \abs*{\Lambda_{0}}^{-1} \sum_{r_{0} \in \Lambda_{0}} \prosca*{
        \delta_{r_{0}r} \otimes
        U_{\Lambda_{0}}(r_{0}) \xi, \delta_{r} \otimes \xi} \\
      &= \abs*{\Lambda_{0}}^{-1} \norm*{\xi}^{2} \geq 0,
    \end{split}
  \end{equation}
  \( \pi \) is indeed positive.

  We now prove \ref{item:4a8b15b53e251b8b}.  The invariance of \( \mathscr{K} \)
  under \( \widetilde{W}_{\Lambda} \) (equation \eqref{eq:31ef088a6c0e308b})
  follows from \eqref{eq:5d413640e021805b} and~\eqref{eq:210cee9951f8d9f5}. We
  now prove the invariance of \( \mathscr{K} \) under
  \( \widetilde{W}_{\mathbb{G}} \) (equation \eqref{eq:0d4d16dfcc1731dc}). By
  the definitions of \( \pi \) and \( \widetilde{W}_{\mathbb{G}} \), we have
  \begin{equation}
    \label{eq:4fa54d277c8a8c3a}
    \begin{split}
      & \leadmathskip
      \abs*{\Lambda_{0}} (\pi \otimes 1) \widetilde{W}_{\mathbb{G}} \\
      &= \sum_{r_{0} \in \Lambda_{0}} \sum_{r, s \in \Lambda} \bigl(e_{r_{0}s,
        s} \otimes U_{\Lambda_{0}}(r_{0}) \otimes 1\bigr) \bigl(e_{r,r} \otimes
      [(\id \otimes
      \alpha^{\ast}_{r})(U_{\mathbb{G}})] \bigr) \\
      &= \sum_{r \in \Lambda} (\id \otimes \id \otimes \alpha^{\ast}_{r}) \left(
        \sum_{r_{0} \in \Lambda_{0}} \sum_{s \in \Lambda} e_{r_{0}s, s} e_{r,r}
        \otimes \bigl[\bigl(U_{\Lambda_{0}}(r_{0}) \otimes 1\bigr)
        U_{\mathbb{G}}\bigr]\right) \\
      &= \sum_{r \in \Lambda} (\id \otimes \id \otimes \alpha^{\ast}_{r}) \left(
        \sum_{r_{0} \in \Lambda_{0}} e_{r_{0}r, r} \otimes
        \bigl[\bigl(U_{\Lambda_{0}}(r_{0}) \otimes 1\bigr)
        U_{\mathbb{G}}\bigr]\right);
    \end{split}
  \end{equation}
  and
  \begin{equation}
    \label{eq:ba962bf6f2ea04b7}
    \begin{split}
      & \leadmathskip
      \abs*{\Lambda_{0}} \widetilde{W}_{\mathbb{G}} (\pi \otimes 1) \\
      &= \sum_{r_{0} \in \Lambda_{0}} \sum_{r, s \in \Lambda} \Bigl(e_{r, r}
      \otimes \bigl[(\id \otimes \alpha^{\ast}_{r})(U_{\mathbb{G}})\bigr]\Bigr)
      \bigl(e_{r_{0}s, s} \otimes
      U_{\Lambda_{0}}(r_{0}) \otimes 1\bigr) \\
      &= \sum_{r_{0} \in \Lambda_{0}} \sum_{r, s \in \Lambda} (\id \otimes \id
      \otimes \alpha^{\ast}_{r}) \left( e_{r, r} e_{r_{0}s, s} \otimes
        \bigl[U_{\mathbb{G}}(U_{\Lambda_{0}}(r_{0})
        \otimes 1)\bigr]\right) \\
      &= \sum_{r_{0} \in \Lambda_{0}} \sum_{s \in \Lambda} (\id \otimes \id
      \otimes \alpha^{\ast}_{r_{0}s}) \left( e_{r_{0}s, s} \otimes
        \bigl[U_{\mathbb{G}}(U_{\Lambda_{0}}(r_{0}) \otimes
        1)\bigr]\right) \\
      &= \sum_{s \in \Lambda} (\id \otimes \id \otimes \alpha^{\ast}_{s})
      \left[(\id \otimes \id \otimes \alpha^{\ast}_{r_{0}}) \left(\sum_{r_{0}
            \in \Lambda_{0}} e_{r_{0}s, s} \otimes
          \bigl[U_{\mathbb{G}}(U_{\Lambda_{0}}(r_{0}) \otimes 1)\bigr]
        \right)\right] \\
      &= \sum_{s \in \Lambda} (\id \otimes \id \otimes \alpha^{\ast}_{s}) \left(
        \sum_{r_{0} \in \Lambda_{0}} e_{r_{0}s, s} \otimes \Bigl(\bigl[(\id
        \otimes
        \alpha^{\ast}_{r_{0}})(U_{\mathbb{G}})\bigr](U_{\Lambda_{0}}(r_{0})
        \otimes
        1)\Bigr)\right) \\
      &= \sum_{s \in \Lambda} (\id \otimes \id \otimes \alpha^{\ast}_{s}) \left(
        \sum_{r_{0} \in \Lambda_{0}} e_{r_{0}s, s} \otimes
        \bigl[\bigl(U_{\Lambda_{0}}(r_{0}) \otimes 1\bigr)
        (U_{\mathbb{G}})\bigr]\right),
    \end{split}
  \end{equation}
  where the last equality used the covariance of \( U_{\mathbb{G}} \) and
  \( U_{\Lambda} \). Combining \eqref{eq:4fa54d277c8a8c3a} and
  \eqref{eq:ba962bf6f2ea04b7} proves
  \begin{equation}
    \label{eq:5c584d28e4c59b50}
    (\pi \otimes 1) \widetilde{W}_{\mathbb{G}} = \widetilde{W}_{\mathbb{G}}(\pi
    \otimes 1),
  \end{equation}
  from which~\eqref{eq:0d4d16dfcc1731dc} follows by noting that \( \pi \) is a
  projection. Now \eqref{eq:6b8a2b0066da754a} follows from
  \eqref{eq:0d4d16dfcc1731dc}, \eqref{eq:31ef088a6c0e308b} and
  \eqref{eq:90563ed2f5a65ac2}. This proves \ref{item:4a8b15b53e251b8b}.
\end{proof}

\begin{prop}
  \label{prop:5f3365d2f8dc1e89}
  Using the above notations, let
  \( c_{\pi} \colon \mathcal{B}(\ell^{2}(\Lambda) \otimes \mathscr{H}) \to
  \mathcal{B}(\mathscr{K}) \) be the compression by the projection \( \pi \)
  (i.e.\ the graph of \( c_{\pi}(A) \) is the intersection of the graph of
  \( \pi A \pi \) with \( \mathscr{K} \times \mathscr{K} \)), then the following
  holds:
  \begin{enumerate}
  \item \label{item:f0f0a1532f3a3e35} the unitary operator
    \begin{displaymath}
      W = \left(c_{\pi} \otimes \id_{\pol(\mathbb{G})} \otimes
        \id_{C(\Lambda)}\right)\left(\widetilde{W}_{\mathbb{G}}\right) \in
      \mathcal{B}(\mathscr{K}) \otimes \pol(\mathbb{G}) \otimes C(\Lambda)
    \end{displaymath}
    is a unitary representation of \( \mathbb{G} \rtimes \Lambda \) on
    \( \mathscr{K} \);
  \item \label{item:45dc99a4ffdb624e} The subrepresentation \( \mathscr{K} \) of
    \( \widetilde{W}_{\mathbb{G}} \) (resp.\ \( \widetilde{W}_{\Lambda} \)) is
    given by
    \( W_{\mathbb{G}} = (c_{\pi} \otimes
    \id)\left(\widetilde{W}_{\mathbb{G}}\right) \) (resp.\
    \( W_{\Lambda} = (c_{\pi} \otimes \id)\left(\widetilde{W}_{\Lambda}\right)
    \)), and
    \begin{equation}
      \label{eq:5958791d3f21dd65}
      W_{\mathbb{G}} = \res_{\mathbb{G}}(W), \quad
      W_{\Lambda} = \res_{\Lambda}(W).
    \end{equation}
  \end{enumerate}
\end{prop}
\begin{proof}
  This follows from Proposition~\ref{prop:3693cca5392c7dd3},
  Corollary~\ref{coro:a2b4b06e5632471c}, Lemma~\ref{lemm:8636ddbbcbafe76a} and
  the definition of subrepresentations.
\end{proof}

\begin{defi}
  \label{defi:700ef7b2d5ffafd8}
  Using the above notations, we call \( W \) the induced representation of
  \( U \), and denote it by
  \( \indrep^{\mathbb{G} \rtimes \Lambda}_{\mathbb{G} \rtimes \Lambda_{0}} %
  (U) \), or simply \( \indrep(U) \) when the underlying compact quantum groups
  \( \mathbb{G} \rtimes \Lambda_{0} \) and \( \mathbb{G} \rtimes \Lambda \) are
  clear from context.
\end{defi}

\section{Some character formulae}
\label{sec:d471d5544009eb55}

Let \( \Lambda_{0} \) be a subgroup of \( \Lambda \), \( U \) a finite
dimensional unitary representation of \( \mathbb{G} \rtimes \Lambda_{0} \),
\( \indrep^{\mathbb{G} \rtimes \Lambda}_{\mathbb{G} \rtimes \Lambda_{0}}(U) \)
the induced representation of the global compact quantum group
\( \mathbb{G} \rtimes \Lambda \). In this section, we aim to calculate the
character of the induced representation
\( \indrep^{\mathbb{G} \rtimes \Lambda}_{\mathbb{G} \rtimes \Lambda_{0}}(U)
\). The approach adopted here emphasizes the underlying group action of
\( \Lambda \) on the characters of the conjugacy class of the principal subgroup
\( \mathbb{G} \rtimes \Lambda_{0} \) as described in
Proposition~\ref{prop:e4e616da0a0995e6}.

For any subgroup \( \Lambda_{1} \) and any \( f_{0} \in C( \Lambda_{1} ) \), we
use \( E_{\Lambda_{1}}(f_{0}) \) to denote the function in \( C( \Lambda ) \)
with \( [E_{\Lambda_{1}}(f_{0})](r) = 0 \) if \( r \notin \Lambda_{1} \) and
\( [E_{\Lambda_{1}}(f_{0})](r) = f_{0}(r) \) if \( r \in \Lambda_{1} \). Then
\( E_{\Lambda_{1}} \colon C( \Lambda_{1} ) \rightarrow C( \Lambda ) \) is a
morphism of {\( C^{\ast} \)}\nobreakdash-algebras, which is not unital unless
\( \Lambda_{1} = \Lambda \), in which case
\( E_{\Lambda_{1}} = \id_{C( \Lambda )} \). By
Proposition~\ref{prop:e4e616da0a0995e6}, we have an action
\begin{equation}
  \label{eq:c378b50e92dd51b2}
  \begin{split}
    \Lambda & \curvearrowright \set*{\mathbb{G} \rtimes r \Lambda_{0} r^{-1}
      \given r \in \Lambda} \\
    s & \mapsto \left\{ \mathbb{G} \rtimes r \Lambda_{0} r^{-1} \mapsto
      \mathbb{G} \rtimes sr \Lambda_{0} {(sr)}^{-1} \right\}
  \end{split}
\end{equation}
of \( \Lambda \) on the set of subgroups of \( \mathbb{G} \rtimes \Lambda \)
conjugate to \( \mathbb{G} \rtimes \Lambda_{0} \) via elements in \( \Lambda \)
(the term conjugate is justified by considering the case when \( \mathbb{G} \)
is a genuine compact group).

Our main result in this section is the following proposition.

\begin{prop}
  \label{prop:56963e422c66f26b}
  Let \( \Lambda_{0} \) be a subgroup of \( \Lambda \),
  \( U \in \mathcal{B}(\mathscr{H}) \otimes \pol(\mathbb{G}) \otimes C(
  \Lambda_{0} ) \) a finite dimensional unitary representation of
  \( \mathbb{G} \rtimes \Lambda_{0} \), \( W \) the induced representation
  \( \indrep^{\mathbb{G} \rtimes \Lambda}_{\mathbb{G} \rtimes \Lambda_{0}}( U )
  \). Suppose \( \chi \) is the character of the unitary representation \( U \)
  of \( \mathbb{G} \rtimes \Lambda_{0} \), and for each \( r \), define
  \begin{equation}
    \label{eq:8716c8ee3e3d259b}
    r \cdot U :=
    ( \id_{\mathscr{H}} \otimes  \alpha_{r^{-1}}^{\ast}
    \otimes \adj_{r^{-1}}^{\ast}) (U)
    \in \mathcal{B}( \mathscr{H} )
    \otimes \pol(\mathbb{G}) \otimes C( r \Lambda_{0}r^{-1} ).
  \end{equation}
  Then \( r \cdot U \) is a unitary representation of
  \( \mathbb{G} \rtimes r \Lambda_{0} r^{-1} \) with \( 1 \cdot U = U \), and
  \( (rs) \cdot U = r \cdot (s \cdot U)\) for all \( r, s \in \Lambda \). Denote
  the character of \( r \cdot U \) by \( \chi_{r} \) (so
  \( \chi_{1_{\Lambda}} = \chi \)), then
  \begin{equation}
    \label{eq:0391706cde602b93}
    \chi_{W} = \abs*{\Lambda_{0}}^{-1}
    \sum_{r \in \Lambda} (\id_{A} \otimes E_{r \Lambda_{0} r^{-1}}) \chi_{r},
  \end{equation}
  where \( \chi_{W} \) is the character of \( W \).
\end{prop}
\begin{proof}
  That \( r \cdot U \) is a finite dimensional unitary representation of
  \( \mathbb{G} \rtimes r \Lambda_{0}r^{-1} \) follows from the fact
  (Proposition~\ref{prop:e4e616da0a0995e6}) that
  \begin{displaymath}
    \alpha_{r}^{\ast} \otimes \adj_{r}^{\ast} \colon A \otimes r^{-1}
    \Lambda_{0} r \rightarrow A \otimes \Lambda_{0}
  \end{displaymath}
  is an isomorphism of compact quantum groups for any \( r \in \Lambda \). The
  identities \( 1_{\Lambda} \cdot U = U \) and
  \( r \cdot (s \cdot U) = (rs) \cdot U\) follows directly from definitions.  We
  proceed to prove the character formula~\eqref{eq:0391706cde602b93}.

  For any \( r \in \Lambda \), let \( {(r \cdot U)}_{\mathbb{G}} \) be the
  restriction of \( r \cdot U \) to \( \mathbb{G} \), and
  \( {(r \cdot U)}_{r\Lambda_{0}r^{-1}} \) the restriction of \( r \cdot U \) to
  \( r \Lambda_{0} r^{-1} \). We denote the character of
  \( {(r \cdot U)}_{\mathbb{G}} \) (resp.\
  \( {(r \cdot U)}_{r \Lambda_{0} r^{-1}} \)) by \( \chi_{r, \mathbb{G}} \)
  (resp.\ \( \chi_{r, r \Lambda_{0}r^{-1}} \)). One easily checks that
  \( \chi_{r, \mathbb{G}} = \alpha_{r^{-1}}^{\ast}(
  \chi_{1_{\Lambda},\mathbb{G}} ) \) and
  \( \chi_{r, r \Lambda_{0} r^{-1}} = \adj_{r^{-1}}^{\ast}( \chi_{1_{\Lambda},
    \Lambda_{0}}) \).  Fix a Hilbert basis \( (e_{1}, \ldots, e_{d}) \) for
  \( \mathscr{H} \), and let \( (e_{ij}, i,j=1,\ldots,d) \) be the corresponding
  matrix unit for \( \mathcal{B}(\mathscr{H}) \). Using this matrix unit, we can
  write
  \begin{subequations}
    \begin{equation}
      \label{eq:9d5af321ab99c204}
      U_{\mathbb{G}} =
      \sum_{i,j=1}^{d} e_{i,j} \otimes u_{ij},\; u_{ij} \in \pol(\mathbb{G});
    \end{equation}
    \begin{equation}
      U_{\Lambda_{0}} =
      \sum_{r_{0} \in \Lambda_{0}} U_{\Lambda_{0}}(r_{0})
      \otimes \delta_{r_{0}}.
    \end{equation}
  \end{subequations}
  Let \( e_{r,s} \), \( \pi \), \( \mathscr{K} \),
  \( \widetilde{W}_{\mathbb{G}} \), \( \widetilde{W}_{\Lambda} \),
  \( W_{\mathbb{G}} \) and \( W_{\Lambda} \) have the same meaning as in
  \S~\ref{sec:3e39eff392b139a7}, then the construction in
  \S~\ref{sec:3e39eff392b139a7} tells us that
  \begin{equation}
    \label{eq:de3a56f01a978237}
    \begin{split}
      \chi_{W} &= (\tr_{\ell^{2}( \Lambda )} \otimes \tr_{\mathscr{H}} \otimes
      \id_{A} \otimes \id_{C( \Lambda )}) \left[ \pi_{12} \cdot
        {(\widetilde{W}_{\mathbb{G}})}_{123} \cdot \pi_{12} \cdot
        {(\widetilde{W}_{\Lambda})}_{124} \cdot \pi_{12} \right].
    \end{split}
  \end{equation}
  In the following calculations, we often omit the subscripts of the trace
  functions \( \tr \) on \( \ell^{2}( \Lambda ) \) or on \( \mathscr{H} \), and
  also the subscripts for the multiplicative neutral element \( 1 \) of various
  algebras, whenever it is a trivial task to decipher to which trace and
  multiplicative neutral element we are referring. The same goes with \( \id \)
  without subscripts.

  Note that for any \( r, s \in \Lambda \),
  \( \adj_{r}^{\ast}( \delta_{s} ) = \delta_{r^{-1}sr} \).  With these
  preparations, we now have
  \begin{equation}
    \label{eq:1e1851295aeba376}
    \begin{split}
      \chi_{r} &= ( \alpha_{r^{-1}}^{\ast}
      \otimes \adj_{r^{-1}}^{\ast} )( \chi ) \\
      &= ( \alpha_{r^{-1}}^{\ast} \otimes \adj_{r^{-1}}^{\ast} ) \left(
        \sum_{i,j=1}^{d} \sum_{r_{0} \in \Lambda_{0}} \tr \Bigl(e_{i,j}
        U_{\Lambda_{0}}(r_{0})\Bigr) u_{ij}  \otimes \delta_{r_{0}}\right) \\
      &= \sum_{i,j = 1}^{d} \sum_{r_{0} \in \Lambda_{0}} \tr\Bigl( e_{i,j}
      U_{\Lambda_{0}}(r_{0}) \Bigr) \alpha_{r^{-1}}^{\ast}(u_{ij}) \otimes
      \delta_{rr_{0}r^{-1}}.
    \end{split}
  \end{equation}

  By \eqref{eq:210cee9951f8d9f5}, \eqref{eq:90563ed2f5a65ac2} and
  \eqref{eq:3c57a5fc33525cfd}, we deduce from \eqref{eq:de3a56f01a978237} that
  \begin{equation}
    \label{eq:867fccfe24eca142}
    \begin{split}
      \abs*{\Lambda_{0}}^{3} \chi_{W} &= \sum_{a_{0},b_{0},c_{0} \in
        \Lambda_{0}} \sum_{a,b,c \in \Lambda} \sum_{r,s,t \in \Lambda}
      \sum_{i,j=1}^{d} \tr(e_{a_{0}a, a}e_{r,r}
      e_{b_{0}b,b} e_{st^{-1},s} e_{c_{0}c,c}) \\
      & \hspace{8ex} \tr\Bigl(U_{\Lambda_{0}}(a_{0}) e_{i,j}
      U_{\Lambda_{0}}(b_{0})U_{\Lambda_{0}}(c_{0})\Bigr)
      \alpha_{r}^{\ast}(u_{ij}) \otimes \delta_{t}.
    \end{split}
  \end{equation}
  On the right side of the above sum, the first trace doesn't vanish if and only
  if it is \( 1 \), which happens exactly when
  \begin{equation}
    \label{eq:1b6272277497e48a}
    \begin{split}
      & \phantom{\iff} a = r = b_{0}b, \; b = st^{-1}, \;
      s=c_{0}c,\; a_{0}a = c \\
      & \iff b = b_{0}^{-1} a, \; c = a_{0}a, \; r = a,\; s = c_{0}a_{0}a,\; t =
      b^{-1}s = a^{-1}b_{0}c_{0}a_{0}a.
    \end{split}
  \end{equation}
  Using this condition in~\eqref{eq:867fccfe24eca142}, we get
  \begin{equation}
    \label{eq:cb61b437d9c99a05}
    \begin{split}
      & \leadmathskip \abs*{\Lambda_{0}}^{3} \chi_{W} \\
      &= \sum_{a_{0},b_{0},c_{0} \in \Lambda_{0}} \sum_{a \in \Lambda}
      \sum_{i,j=1}^{d} \tr\Bigl(U_{\Lambda_{0}}(a_{0}) e_{i,j}
      U_{\Lambda_{0}}(b_{0})U_{\Lambda_{0}}(c_{0})\Bigr)
      \alpha_{a}^{\ast}(u_{ij}) \otimes
      \delta_{a^{-1}b_{0}c_{0}a_{0}a} \\
      &= \sum_{a_{0},b_{0},c_{0} \in \Lambda_{0}} \sum_{a \in \Lambda}
      \sum_{i,j=1}^{d} \tr\Bigl( e_{i,j}
      U_{\Lambda_{0}}(b_{0})U_{\Lambda_{0}}(c_{0}) U_{\Lambda_{0}}(a_{0}) \Bigr)
      \alpha_{a}^{\ast}(u_{ij}) \otimes
      \delta_{a^{-1}b_{0}c_{0}a_{0}a} \\
      &= \sum_{a_{0},b_{0},c_{0} \in \Lambda_{0}} \sum_{a \in \Lambda}
      \sum_{i,j=1}^{d} \tr\Bigl( e_{i,j} U_{\Lambda_{0}}(b_{0}c_{0}a_{0}) \Bigr)
      \alpha_{a}^{\ast}(u_{ij}) \otimes
      \delta_{a^{-1}b_{0}c_{0}a_{0}a} \\
      &= \abs*{\Lambda_{0}}^{2} \sum_{a \in \Lambda} \sum_{r_{0} \in
        \Lambda_{0}} \sum_{i,j=1}^{d} \tr\Bigl( e_{i,j} U_{\Lambda_{0}}(r_{0})
      \Bigr) \alpha_{a}^{\ast}(u_{ij}) \otimes
      \delta_{a^{-1} r_{0} a} \\
      &= \abs*{\Lambda_{0}}^{2} \sum_{r \in \Lambda} (\id \otimes E_{r^{-1}
        \Lambda_{0} r}) ( \chi_{r}),
    \end{split}
  \end{equation}
  where the last line uses~\eqref{eq:1e1851295aeba376} and the change of
  variable \( r = a^{-1} \). Dividing \( \abs*{\Lambda_{0}}^{3} \) on both sides
  of~\eqref{eq:cb61b437d9c99a05} proves~\eqref{eq:0391706cde602b93}.
\end{proof}

\begin{coro}
  \label{coro:7b7b679f656174f9}
  Using the notations in Proposition~\ref{prop:56963e422c66f26b}, \( U \) and
  \( r \cdot U \) induce equivalent unitary representations of
  \( \mathbb{G} \rtimes \Lambda \) for all \( r \cdot U \).
\end{coro}
\begin{proof}
  By Proposition~\ref{prop:56963e422c66f26b}, we see that \( \indrep(U) \) and
  \( \indrep(r \cdot U) \) have the same character.
\end{proof}

It is worth pointing out that there are in fact many repetitions in the terms of
the right side of formula~\eqref{eq:0391706cde602b93}, as is shown by the
following lemma.

\begin{lemm}
  \label{lemm:96218436efb1ec44}
  Using the notations of Proposition~\ref{prop:56963e422c66f26b}, the following
  holds:
  \begin{enumerate}
  \item \label{item:798c8845ec549eea} for any \( r \in \Lambda \), we have
    \begin{equation}
      \label{eq:ab9447ac2beea186}
      (\id \otimes E_{r \Lambda_{0} r^{-1}}) \chi_{r} = %
      \bigl( \alpha_{r^{-1}}^{\ast} \otimes \adj_{r^{-1}}^{\ast} \bigr) %
      \left[(\id \otimes E_{\Lambda_{0}}) ( \chi ) \right];
    \end{equation}
    in \( \pol(\mathbb{G}) \otimes C( \Lambda ) \);
  \item \label{item:38dcc9ada3065523} for any \( r, s \in \Lambda \), if
    \( r^{-1}s \in \Lambda_{0} \), i.e.\ \( r \Lambda_{0} = s \Lambda_{0} \) and
    \( r \Lambda_{0} r^{-1} = s \Lambda_{0} s^{-1} \), then
    \begin{equation}
      \label{eq:48a692f4e0c50430}
      (\id \otimes E_{r \Lambda_{0} r^{-1}}) \chi_{r}
      = (\id \otimes E_{s \Lambda_{0} s^{-1}}) \chi_{s}
    \end{equation}
    in \( \pol(\mathbb{G}) \otimes C( \Lambda ) \). In particular,
    \begin{equation}
      \label{eq:58712915d021c218}
      \chi_{r} = \chi_{s},
    \end{equation}
    or equivalently, \( r \cdot U \) and \( s \cdot U \) are unitarily
    equivalent unitary representations of the same compact quantum group
    \( \mathbb{G} \rtimes r \Lambda_{0}r^{-1} \).
  \end{enumerate}
\end{lemm}
\begin{proof}
  Using the same notations as in the proof of
  Proposition~\ref{prop:56963e422c66f26b}, it is clear that
  \begin{subequations}
    \begin{equation}
      \label{eq:5c2a358aceaa9749}
      {(r \cdot U)}_{\mathbb{G}}
      = \sum_{i,j = 1}^{d} e_{i,j} \otimes \alpha_{r^{-1}}^{\ast}(u_{ij}),
    \end{equation}
    \begin{equation}
      \label{eq:bb7d35412572cc19}
      {(r \cdot U)}_{r \Lambda_{0} r^{-1}}
      = \sum_{r_{0} \in \Lambda_{0}} U_{\Lambda_{0}}(r_{0})
      \otimes \delta_{rr_{0}r^{-1}}.
    \end{equation}
  \end{subequations}
  Calculating in \( \pol(\mathbb{G}) \otimes C( \Lambda ) \), we have
  \begin{equation}
    \label{eq:73b5c68e7ad31238}
    \begin{split}
      (\id \otimes E_{r\Lambda_{0}r^{-1}}) \chi_{r} %
      &= \sum_{i,j=1}^{d} \sum_{r_{0} \in \Lambda_{0}}
      \tr\bigl(e_{i,j}U_{\Lambda_{0}}(r_{0})\bigr)
      \otimes \alpha_{r^{-1}}^{\ast}(u_{ij}) \otimes \delta_{rr_{0}r^{-1}} \\
      &= \bigl( \alpha_{r^{-1}}^{\ast} \otimes \adj_{r^{-1}}^{\ast} \bigr) %
      \sum_{i,j=1}^{d} \sum_{r_{0} \in \Lambda_{0}} \tr\bigl(e_{i,j}
      U_{\Lambda_{0}}(r_{0})\bigr) \otimes u_{ij} \otimes \delta_{r_{0}} \\
      &= \bigl( \alpha_{r^{-1}}^{\ast} \otimes \adj_{r^{-1}} \bigr) %
      \left[(\id \otimes E_{\Lambda_{0}}) \chi\right].
    \end{split}
  \end{equation}
  This proves \ref{item:798c8845ec549eea}.

  By \ref{item:798c8845ec549eea}, to establish \ref{item:38dcc9ada3065523}, it
  suffices to show that
  \begin{equation}
    \label{eq:52f9e26b7886c145}
    \forall s_{0} \in \Lambda_{0}, \quad
    (\id \otimes E_{\Lambda_{0}}) \chi = %
    \left( \alpha_{s_{0}}^{\ast} \otimes \adj_{s_{0}}^{\ast} \right) %
    \left[(\id \otimes E_{\Lambda_{0}}) \chi\right].
  \end{equation}
  Calculating the right side gives
  \begin{equation}
    \label{eq:902eb3db52bc8ac7}
    \begin{split}
      & \leadmathskip \bigl( \alpha_{s_{0}}^{\ast} \otimes \adj_{s_{0}}^{\ast}
      \bigr) %
      \left[(\id \otimes E_{\Lambda_{0}}) \chi\right] \\
      &= \sum_{i,j=1}^{d} \sum_{r_{0} \in \Lambda_{0}} \tr\bigl(
      e_{i,j}U_{\Lambda_{0}}(r_{0})\bigr) \otimes \alpha_{s_{0}}^{\ast}(u_{ij})
      \otimes
      \delta_{s_{0}^{-1}r_{0}s_{0}} \\
      &= \sum_{i,j=1}^{d} \sum_{r_{0} \in \Lambda_{0}} \tr\bigl(
      e_{i,j}U_{\Lambda_{0}}(s_{0}r_{0}s_{0}^{-1}) \bigr) \otimes
      \alpha_{s_{0}}^{\ast}(u_{ij}) \otimes \delta_{r_{0}} \\
      &= \sum_{i,j=1}^{d} \sum_{r_{0} \in \Lambda_{0}} \tr\Bigl(U(r_{0})
      U_{\Lambda_{0}}(s_{0}^{-1}) e_{i,j} U(s_{0}) \Bigr) \otimes
      \alpha_{s_{0}}^{\ast}(u_{ij}) \otimes \delta_{r_{0}}.
    \end{split}
  \end{equation}
  Since \( U_{\Lambda_{0}} \) and \( U_{\mathbb{G}} \) are covariant, we have
  \begin{equation}
    \label{eq:9a914bb57b68de88}
    \sum_{i,j}^{d} U(s_{0}) e_{i,j} \otimes u_{ij} = \sum_{i,j=1}^{d}
    e_{i,j}U(s_{0}) \otimes \alpha^{\ast}_{s_{0}}(u_{ij}).
  \end{equation}
  Combining~\eqref{eq:902eb3db52bc8ac7} and~\eqref{eq:9a914bb57b68de88}, we have
  \begin{equation}
    \label{eq:a253f7226c88cdef}
    \begin{split}
      & \leadmathskip \bigl( \alpha_{s_{0}}^{\ast} \otimes \adj_{s_{0}}^{\ast}
      \bigr) %
      \left[(\id \otimes E_{\Lambda_{0}}) \chi\right] \\
      &= \sum_{i,j=1}^{d} \sum_{r_{0} \in \Lambda_{0}} \tr\Bigl(U(r_{0})
      U_{\Lambda_{0}}(s_{0}^{-1}) e_{i,j} U(s_{0}) \Bigr) \otimes
      \alpha_{s_{0}}^{\ast}(u_{ij}) \otimes \delta_{r_{0}} \\
      &= \sum_{i,j=1}^{d} \sum_{r_{0} \in \Lambda_{0}} \tr\Bigl(U(r_{0})
      U_{\Lambda_{0}}(s_{0}^{-1}) U(s_{0}) e_{i,j} \Bigr) \otimes
      u_{ij} \otimes \delta_{r_{0}} \\
      &= \sum_{i,j=1}^{d} \sum_{r_{0} \in \Lambda_{0}} \tr\Bigl(U(r_{0}) e_{i,j}
      \Bigr) \otimes u_{ij} \otimes \delta_{r_{0}} \\
      &= \sum_{i,j=1}^{d} \sum_{r_{0} \in \Lambda_{0}} \tr\Bigl(e_{i,j} U(r_{0})
      \Bigr) \otimes u_{ij} \otimes \delta_{r_{0}} \\
      &= (\id \otimes E_{\Lambda_{0}}) \chi.
    \end{split}
  \end{equation}
  This establishes~\eqref{eq:52f9e26b7886c145} and proves
  \ref{item:38dcc9ada3065523}.
\end{proof}

\begin{rema}
  \label{rema:26d6215da29584a1}
  By Lemma~\ref{lemm:96218436efb1ec44} \ref{item:38dcc9ada3065523} and
  Proposition~\ref{prop:56963e422c66f26b}, one can in fact choose any complete
  set \( L \subseteq \Lambda \) of representatives of the left coset space
  \( \Lambda / \Lambda_{0} \), and the character
  formula~\eqref{eq:0391706cde602b93} can then be written more concisely as
  \begin{equation}
    \label{eq:88bd62f11469fd21}
    \chi_{W} =
    \sum_{r \in L} (\id_{A} \otimes E_{r \Lambda_{0} r^{-1}}) \chi_{r}.
  \end{equation}
  In the classical case where \( \mathbb{G} \) is a genuine compact group, one
  can easily check that the usual character formula for the representation
  induced by a representation of an open subgroup takes the
  form~\eqref{eq:88bd62f11469fd21}. The reason we
  prefer~\eqref{eq:0391706cde602b93} is that it does not involve a seemingly
  arbitrary choice of a complete set of representatives \( L \) for
  \( \Lambda / \Lambda_{0} \), and thus, in the author's opinion, is more
  aesthetically pleasing. One might also use this choice of left coset
  representatives to fabric the induced representation. However, in our more
  symmetric approach (cf.\ \S~\ref{sec:3e39eff392b139a7}), everything seems more
  natural, and the underlying group action of \( \Lambda \) on the various
  characters \( \chi_{r}, r \in \Lambda \) becomes more transparent
  in~\eqref{eq:0391706cde602b93}, and we hope this hidden symmetry will keep the
  reader from losing himself/herself in the details of the tedious calculations
  to be presented later.
\end{rema}

\section{Dimension of the intertwiner space of induced representations}
\label{sec:51325cc5d5808588}

Let \( \Theta, \Xi \) be subgroups of \( \Lambda \),
\( U \in \mathcal{B}(\mathscr{H}) \otimes \pol(\mathbb{G}) \otimes C(\Theta) \)
a finite dimensional unitary representation of \( \mathbb{G} \rtimes \Theta \),
\( W \in \mathcal{B}(\mathscr{H}) \otimes \pol(\mathbb{G}) \otimes C(\Xi) \) a
finite dimensional unitary representation of \( \mathbb{G} \rtimes \Xi \). For
the sake of brevity, we denote the induced representation
\( \indrep_{\mathbb{G} \rtimes \Theta}^{\mathbb{G} \rtimes \Lambda}(U) \) simply
by \( \indrep(U) \), and \( \indrep(W) \) has the similar obvious
meaning. Equipped with the character formula established in
\S~\ref{sec:d471d5544009eb55}, one naturally wonders how can we calculate
\( \dim \morph_{\mathbb{G} \rtimes \Lambda}\bigl(\indrep(U), \indrep(W)\bigr) \)
in terms of some simpler data. This section focuses on this calculation, and the
result here will play an important role in proving the irreducibility of some
induced representations (as it turns out, these are all irreducible
representations of \( \mathbb{G} \rtimes \Lambda \) up to equivalence) as well
as our later calculation of the fusion rules.

For any representation \( \rho \), we use \( \chi( \rho ) \) to denote the
character of the representation. We denote the Haar state on \( \mathbb{G} \) by
\( h \), and the Haar state on \( \mathbb{G} \rtimes \Lambda_{0} \) by
\( h^{\Lambda_{0}} \) whenever \( \Lambda_{0} \) is a subgroup of \( \Lambda \).

By the general representation theory of compact quantum groups, we have
\begin{equation}
  \label{eq:a77b87f9f0cd1a15}
  \dim \morph_{\mathbb{G} \rtimes \Lambda}\bigl(\indrep(U), \indrep(W)\bigr) %
  = h^{\Lambda}\bigl({[\chi(\indrep(U))]}^{\ast}{[\chi(\indrep(W))]}\bigr).
\end{equation}
By Proposition~\ref{prop:56963e422c66f26b}, for each \( r \in \Lambda \), we
have a representation \( r \cdot U \) (resp.\ \( r \cdot W \)) of
\( \mathbb{G} \rtimes r \Theta r^{-1} \) (resp.\
\( \mathbb{G} \rtimes r\Xi r^{-1} \)), and combined
with~\eqref{eq:a77b87f9f0cd1a15}, we have
\begin{equation}
  \label{eq:5d0471df9a2c42d1}
  \begin{split}
    & \leadmathskip \dim \morph_{\mathbb{G}}(\indrep(U), \indrep(W)) \\
    &= \frac{1}{\abs*{\Theta} \cdot \abs*{\Xi}} \sum_{r, s \in \Lambda} %
    h^{\Lambda} \left( {[(\id \otimes E_{r \Theta r^{-1}}) \chi(r \cdot
        U)]}^{\ast} %
      {[(\id \otimes E_{s \Xi s^{-1}}) \chi(s \cdot W)]}\right).
  \end{split}
\end{equation}
\begin{nota}
  \label{nota:ce8ce7c22fef0d7f}
  To simplify our notations, let
  \( \Lambda(r,s) := r \Theta r^{-1} \cap s \Xi s^{-1} \) for any
  \( r, s \in \Lambda \).
\end{nota}

\begin{lemm}
  \label{lemm:c118e0a0bf93e9c8}
  Using the above notations, for any \( r, s \in \Lambda \), we have
  \begin{equation}
    \label{eq:959b778f19bf0de1}
    \begin{split}
      & \leadmathskip h^{\Lambda}\left( {[(\id \otimes E_{r \Theta r^{-1}})
          \chi(r \cdot U)]}^{\ast} %
        {[(\id \otimes E_{s \Xi s^{-1}}) \chi(s \cdot W)]}\right) \\
      &= \frac{1}{[\Lambda \colon \Lambda(r, s)]} %
      \dim \morph_{\mathbb{G} \rtimes \Lambda(r,s)} %
      \left((r \cdot U) \vert_{\mathbb{G} \rtimes \Lambda(r,s)}, %
        (s \cdot W) \vert_{\mathbb{G} \rtimes \Lambda(r,s)}\right).
    \end{split}
  \end{equation}
\end{lemm}
\begin{proof}
  For any subgroup \( \Lambda_{0} \) of \( \Lambda \), whenever
  \( f \in \pol(\mathbb{G}) \), \( r_{0} \in \Lambda_{0} \),
  by~\eqref{eq:1c67d64dafdd1948} in \S~\ref{sec:8407664c09390193}, we have
  \begin{equation}
    \label{eq:4428392246ea1824}
    h^{\Lambda}(f \otimes \delta_{r_{0}}) =  \frac{1}{\abs*{\Lambda}} h(f) %
    = \frac{1}{[\Lambda \colon \Lambda_{0}]}
    h^{\Lambda_{0}}(f \otimes \delta_{r_{0}}).
  \end{equation}
  Hence,
  \begin{equation}
    \label{eq:96c92a27cb9c8369}
    h^{\Lambda} \circ (\id \otimes E_{\Lambda_{0}})
    = \frac{1}{[\Lambda \colon \Lambda_{0}]} h^{\Lambda_{0}}.
  \end{equation}
  By definition and a straightforward calculation, we have
  \begin{subequations}
    \begin{equation}
      \label{eq:277754d3c7abc899}
      (\id \otimes E_{r \Theta r^{-1}}) \chi(r \cdot U) %
      = \sum_{t \in r \Theta r^{-1}} (\tr \otimes \id)
      \left({(r \cdot U)}_{\mathbb{G}}\bigl({(r \cdot
          U)}_{r \Theta r^{-1}}(t) \otimes 1\bigr)\right) \otimes \delta_{t},
    \end{equation}
    \begin{equation}
      \label{eq:b7f13b6a452d48ed}
      (\id \otimes E_{s \Xi s^{-1}}) \chi(s \cdot W) %
      = \sum_{t \in s \Xi s^{-1}} (\tr \otimes \id)
      \left( {(s \cdot W)}_{\mathbb{G}}
        \bigl({(s \cdot W)}_{s \Xi s^{-1}}(t)
        \otimes 1\bigr)\right) \otimes \delta_{t}.
    \end{equation}
  \end{subequations}
  It follows from~\eqref{eq:277754d3c7abc899} and~\eqref{eq:b7f13b6a452d48ed}
  that
  \begin{equation}
    \label{eq:413c676ff42e9238}
    \begin{split}
      & \leadmathskip {[(\id \otimes E_{r \Theta r^{-1}}) \chi(r \cdot
        U)]}^{\ast} %
      {[(\id \otimes E_{s \Xi s^{-1}}) \chi(s \cdot W)]} \\
      &= \sum_{t \in \Lambda(r,s)} \Big\{ {\left[ (\tr \otimes \id) \left({(r
              \cdot U)}_{\mathbb{G}}\bigl({(r \cdot U)}_{r \Theta r^{-1}}(t)
            \otimes
            1\bigr)\right) \right]}^{\ast} \\
      & \dmslskip \left[(\tr \otimes \id) \left({(s \cdot
            W)}_{\mathbb{G}}\bigl({(s \cdot W)}_{s \Xi s^{-1}}(t) \otimes
          1\bigr)\right) \right]\Big\} \otimes \delta_{t} \\
      &= (\id \otimes E_{\Lambda(r, s)}) \left( %
        {\left[ \chi\left( (r \cdot U) \vert_{\mathbb{G} \rtimes \Lambda(r, s)}
            \right)\right]}^{\ast} \left[ \chi \left( (s \cdot W)
            \vert_{\mathbb{G} \rtimes \Lambda(r,s)} \right) \right] %
      \right).
    \end{split}
  \end{equation}
  Taking \( \Lambda_{0} = \Lambda(r,s) \) in~\eqref{eq:96c92a27cb9c8369} and
  combining with~\eqref{eq:413c676ff42e9238} proves~\eqref{eq:959b778f19bf0de1}.
\end{proof}

\begin{prop}
  \label{prop:a8bf0fe189bec421}
  Using the above notations, we have
  \begin{equation}
    \label{eq:628ef117a2f86cef}
    \begin{split}
      & \leadmathskip \dim \morph_{\mathbb{G} \rtimes \Lambda} \bigl(\indrep(U),
      \indrep(W)\bigr) \\
      &= \frac{1}{\abs*{\Theta} \cdot \abs*{\Xi}} \sum_{r, s \in \Lambda} %
      \frac{1}{[\Lambda \colon \Lambda(r, s)]} \\
      & \dmslskip \dim \morph_{\mathbb{G} \rtimes \Lambda(r, s)} \left( {(r
          \cdot U)} \vert_{\mathbb{G} \rtimes \Lambda(r, s)}, {(s \cdot W)}
        \vert_{\mathbb{G} \rtimes \Lambda(r, s)} \right).
    \end{split}
  \end{equation}
\end{prop}
\begin{proof}
  This follows directly from the formula~\eqref{eq:5d0471df9a2c42d1} and
  Lemma~\ref{lemm:c118e0a0bf93e9c8}.
\end{proof}

\begin{coro}
  \label{coro:a250dbc9bc328b12}
  Let \( \Lambda_{0} \) be a subgroup of \( \Lambda \), \( U \) a unitary
  representation of \( \mathbb{G} \rtimes \Lambda_{0} \), then the following are
  equivalent:
  \begin{enumerate}
  \item \label{item:365832347489d5b6} the unitary representation
    \( \indrep(U) \) of \( \mathbb{G} \rtimes \Lambda \) is irreducible;
  \item \label{item:b73a27a9d6dde808} for any \( r, s \in \Lambda \), posing
    \( \Lambda(r,s) = r\Lambda_{0}r^{-1} \cap s\Lambda_{0}s^{-1} \), we have
    \begin{equation}
      \label{eq:5e8e3176c982537a}
      \dim \morph_{\mathbb{G} \rtimes \Lambda(r, s)}
      \Bigl((r \cdot U) \vert_{\mathbb{G} \rtimes
        \Lambda(r,s)},
      (s \cdot U) \vert_{\mathbb{G} \rtimes \Lambda(r,s)}\Bigr)
      = \delta_{r \Lambda_{0}, s \Lambda_{0}};
    \end{equation}
  \item \label{item:ae7d985ccc925ef9} \( U \) is irreducible, and
    \begin{equation}
      \label{eq:64a005ad1b44f601}
      \begin{split}
        & \forall r, s \in \Lambda, \quad r^{-1} s \notin \Lambda_{0} \\
        &\implies \dim \morph_{\mathbb{G} \rtimes \Lambda(r, s)} \Bigl((r \cdot
        U) \vert_{\mathbb{G} \rtimes \Lambda(r,s)}, (s \cdot U)
        \vert_{\mathbb{G} \rtimes \Lambda(r,s)}\Bigr) = 0.
      \end{split}
    \end{equation}
  \end{enumerate}
  In particular, if any of the above conditions holds, then \( U \) itself is
  irreducible.
\end{coro}
\begin{proof}
  If \( r^{-1}s \in \Lambda_{0} \), then
  \( r \Lambda_{0} r^{-1} = s \Lambda_{0} s^{-1} \), so
  \( \Lambda(r, s) = r \Lambda_{0} r^{-1} = s \Lambda_{0} s^{-1} \).  By
  Proposition~\ref{prop:e4e616da0a0995e6}, we see that
  \begin{equation}
    \label{eq:10e362da86ac4546}
    \dim \morph_{\mathbb{G} \rtimes r \Lambda_{0} r^{-1}}
    \bigl(r \cdot U, r \cdot
    U\bigr) %
    = \dim \morph_{\mathbb{G} \rtimes \Lambda_{0}}(U, U).
  \end{equation}
  By Proposition~\ref{prop:a8bf0fe189bec421}, Lemma~\ref{lemm:96218436efb1ec44},
  and the above, we have
  \begin{displaymath}
    \begin{split}
      &\leadmathskip \dim \morph_{\mathbb{G}}\bigl(\indrep(U),
      \indrep(U)\bigr) \\
      &= \frac{1}{\abs*{\Lambda_{0}}^{2}} \sum_{\substack{r, s \in \Lambda, \\
          r^{-1}s \in \Lambda_{0}}} \frac{1}{[\Lambda \colon
        r\Lambda_{0}r^{-1}]}\dim \morph_{\mathbb{G} \rtimes r \Lambda_{0}r^{-1}}
      \bigl( r \cdot U, s \cdot U\bigr) \\
      & \dmslskip + \frac{1}{\abs*{\Lambda_{0}}^{2}} \sum_{\substack{r,s \in
          \Lambda, \\r^{-1}s \notin \Lambda_{0}}} \frac{1}{[\Lambda \colon
        \Lambda(r,s)]} \dim \morph_{\mathbb{G} \rtimes \Lambda(r, s)}\Bigl((r
      \cdot U) \vert_{\mathbb{G} \rtimes \Lambda(r,s)}, (s \cdot U)
      \vert_{\mathbb{G} \rtimes
        \Lambda(r,s)}\Bigr) \\
      &= \frac{1}{\abs*{\Lambda_{0}}^{2}} \sum_{\substack{r, s \in \Lambda, \\
          r^{-1}s \in \Lambda_{0}}} \dim \morph_{\mathbb{G} \rtimes r
        \Lambda_{0}r^{-1}}
      \bigl( r \cdot U, r \cdot U\bigr) \\
      & \dmslskip + \frac{1}{\abs*{\Lambda_{0}}^{2}} \sum_{\substack{r,s \in
          \Lambda, \\r^{-1}s \notin \Lambda_{0}}} \frac{1}{[\Lambda \colon
        \Lambda(r,s)]} \dim \morph_{\mathbb{G} \rtimes \Lambda(r, s)}\Bigl((r
      \cdot U) \vert_{\mathbb{G} \rtimes \Lambda(r,s)}, (s \cdot U)
      \vert_{\mathbb{G} \rtimes
        \Lambda(r,s)}\Bigr) \\
      &= \frac{1}{\abs*{\Lambda_{0}}^{2}} \sum_{\substack{r, s \in \Lambda, \\
          r^{-1}s \in \Lambda_{0}}} \frac{1}{[\Lambda \colon \Lambda_{0}]}
      \dim_{\mathbb{G} \rtimes
        \Lambda_{0}}(U, U) \\
      & \dmslskip + \frac{1}{\abs*{\Lambda_{0}}^{2}} \sum_{\substack{r,s \in
          \Lambda, \\r^{-1}s \notin \Lambda_{0}}} \frac{1}{[\Lambda \colon
        \Lambda(r,s)]} \dim \morph_{\mathbb{G} \rtimes \Lambda(r, s)}\Bigl((r
      \cdot U) \vert_{\mathbb{G} \rtimes \Lambda(r,s)}, (s \cdot U)
      \vert_{\mathbb{G} \rtimes
        \Lambda(r,s)}\Bigr) \\
      &= \frac{\abs*{\Lambda} \cdot \abs*{\Lambda_{0}}}{\abs*{\Lambda_{0}}^{2}
        \cdot [\Lambda \colon
        \Lambda_{0}]} \dim \morph_{\mathbb{G} \rtimes \Lambda_{0}}(U, U) \\
      &\dmslskip + \frac{1}{\abs*{\Lambda_{0}}^{2}} \sum_{\substack{r,s \in
          \Lambda, \\r^{-1}s \notin \Lambda_{0}}} \frac{1}{[\Lambda \colon
        \Lambda(r,s)]} \dim \morph_{\mathbb{G} \rtimes \Lambda(r, s)}\Bigl((r
      \cdot U) \vert_{\mathbb{G} \rtimes \Lambda(r,s)}, (s \cdot U)
      \vert_{\mathbb{G} \rtimes \Lambda(r,s)}\Bigr).
    \end{split}
  \end{displaymath}
  Since
  \( \abs*{\Lambda} \cdot \abs*{\Lambda_{0}} = \abs*{\Lambda_{0}}^{2} \cdot
  [\Lambda \colon \Lambda_{0}] \) and
  \begin{equation}
    \label{eq:b9b1f498b938935c}
    \begin{split}
      \dim\morph_{\mathbb{G} \rtimes r\Lambda_{0}r^{-1}} \bigl(r \cdot U, s
      \cdot U\bigr) %
      &= \dim\morph_{\mathbb{G} \rtimes r\Lambda_{0}r^{-1}} \bigl(r \cdot U, r
      \cdot U\bigr) \\
      &= \dim\morph_{\mathbb{G} \rtimes \Lambda_{0}}(U, U) %
      = \dim\selfmorph_{\mathbb{G} \rtimes \Lambda_{0}}(U)
    \end{split}
  \end{equation}
  whenever \( r^{-1}s \in \Lambda_{0} \) by Lemma~\ref{lemm:96218436efb1ec44}
  and Proposition~\ref{prop:e4e616da0a0995e6}, the above calculation yields
  \begin{equation}
    \label{eq:e005c2718206856e}
    \dim \selfmorph_{\mathbb{G} \rtimes \Lambda}\left(\indrep(U)\right)
    = \dim \selfmorph_{\mathbb{G} \rtimes \Lambda_{0}}(U) +
    \frac{1}{\abs*{\Lambda_{0}}^{2}} %
    \sum_{\substack{r,s \in \Lambda, \\r^{-1}s \notin \Lambda_{0}}}
    {d(r,s)  \over [\Lambda : \Lambda(r,s)]},
  \end{equation}
  where
  \begin{equation}
    \label{eq:49be7159b83b3989}
    d(r,s) :=  \dim \morph_{\mathbb{G} \rtimes \Lambda(r, s)}\Bigl((r \cdot U)
    \vert_{\mathbb{G} \rtimes \Lambda(r,s)}, (s \cdot U) \vert_{\mathbb{G}
      \rtimes \Lambda(r,s)}\Bigr).
  \end{equation}
  The corollary follows from \eqref{eq:b9b1f498b938935c}
  \eqref{eq:e005c2718206856e}, \eqref{eq:49be7159b83b3989} and the fact that a
  representation is irreducible if and only if the dimension of the space of its
  self-intertwiners is \( 1 \).
\end{proof}

\begin{rema}
  \label{rema:9e4edc5eecb3b25c}
  Corollary~\ref{coro:a250dbc9bc328b12} is the quantum analogue for Mackey's
  criterion for irreducibility.
\end{rema}

\section{The \texorpdfstring{\( C^{\ast} \)}{C\*}-tensor category
  \texorpdfstring{\( \mathcal{CSR}_{\Lambda_{0}} \)}{CSR\_Lambda0}}
\label{sec:deea2fa481e1836a}

We begin by recalling the notations in Proposition~\ref{prop:56963e422c66f26b}:
for any unitary representation
\( U_{\mathbb{G}} \in \mathcal{B}(\mathscr{H}) \otimes \pol(\mathbb{G}) \) of
\( \mathbb{G} \) on some finite dimensional Hilbert space \( \mathscr{H} \), and
any \( r \in \Lambda \), let \( r \cdot U_{\mathbb{G}} \) be the unitary
representation
\( (\id_{\mathscr{H}} \otimes \alpha^{\ast}_{r^{-1}})(U_{\mathbb{G}}) \) of
\( \mathbb{G} \) on the same space \( \mathscr{H} \). It is easy to see that
this defines a left group action of \( \Lambda \) on the proper class of all
unitary representations of \( \mathbb{G} \), and by passing to quotients, this
representation induces an action of \( \Lambda \) on \( \irr(\mathbb{G}) \).
From now on, whenever we talk about \( \Lambda \) acting on a unitary
representation \( U_{\mathbb{G}} \) of \( \mathbb{G} \), or on some class
\( x \in \irr(\mathbb{G}) \), we always mean these actions.

\begin{defi}
  \label{defi:5beb1f580d0ddd31}
  A subgroup \( \Lambda_{0} \) of \( \Lambda \) is called a general isotropy
  subgroup if there is some \( n \in \mathbb{N} \), such that \( \Lambda_{0} \)
  is an isotropy subgroup (subgroup of stabilizer for some point) for the
  \( n \)-fold product \( {[\irr(\mathbb{G})]}^{n} \) as a \( \Lambda \)-set; in
  other words, if there exists an \( n \)-tuple \( (x_{1}, \ldots, x_{n}) \)
  with all \( x_{i} \in \irr(\mathbb{G}) \), such that
  \begin{displaymath}
    \Lambda_{0} = \set*{r \in \Lambda \given
      \forall i = 1, \ldots, n, \quad r \cdot x_{i} = x_{i}}
    = \cap_{i=1}^{n} \Lambda_{x_{i}}.
  \end{displaymath}
  The finite (recall that \( \Lambda \) is finite) family of all general
  isotropy subgroups of \( \Lambda \) is denoted by \( \giso(\Lambda) \).
\end{defi}

The following proposition is an immediate consequence of properties of
\( \Lambda \)-sets and Definition~\ref{defi:5beb1f580d0ddd31}.
\begin{prop}
  \label{prop:81d9440166ac6304}
  The family \( \giso(\Lambda) \) is stable under intersection and conjugation
  by elements of \( \Lambda \). \qed
\end{prop}

\begin{defi}
  \label{defi:f58794ba929ea9e8}
  Let \( \Lambda_{0} \) be a general isotropy subgroup of \( \Lambda \). A
  covariant system of representations (or CSR for short) subordinate to
  \( \Lambda_{0} \) is a triple \( (\mathscr{H}, u, w) \), where
  \begin{itemize}
  \item \label{defi:6601d20a6cbd026a} \( \mathscr{H} \) is a finite dimensional
    Hilbert space;
  \item \label{defi:b2c3ce6ab749faa1} \( u \) is a unitary representation of
    \( \mathbb{G} \) on \( \mathscr{H} \);
  \item \label{defi:475cb02743141d73} \( w \) is a unitary representation of
    \( \Lambda_{0} \) on \( \mathscr{H} \),
  \end{itemize}
  such that \( u \) and \( w \) are covariant. In this paper, CSRs are often
  denoted by bold faced uppercase letters like
  \( \mathbf{A}, \mathbf{B}, \mathbf{C}, \ldots \)(mostly \( \mathbf{S} \)) with
  possible subscripts.
\end{defi}

By Proposition~\ref{prop:3693cca5392c7dd3}, the covariant systems of
representations subordinate to a general isotropy subgroup \( \Lambda_{0} \)
correspond bijectively to the class of unitary representations of
\( \mathbb{G} \rtimes \Lambda_{0} \), via
\begin{displaymath}
  (\mathscr{H}, u, w) \mapsto u_{12}w_{13}
\end{displaymath}
in one direction, and
\begin{displaymath}
  U_{\mathscr{H}} \mapsto (\mathscr{H}, U_{\mathscr{H}, \mathbb{G}},
  U_{\mathscr{H}, \Lambda_{0}})
\end{displaymath}
in the other, where \( \mathscr{H} \) is the underlying space of the
representation \( U_{\mathscr{H}} \) of \( \mathbb{G} \rtimes \Lambda_{0} \),
and \( U_{\mathscr{H},\mathbb{G}} \), \( U_{\mathscr{H}, \Lambda_{0}} \) are the
restrictions of \( U_{\mathscr{H}} \) to \( \mathbb{G} \) and \( \Lambda_{0} \)
respectively. Using this bijection, we can transport the rigid
\( C^{\ast} \)-tensor category structure on
\( \repcat(\mathbb{G} \rtimes \Lambda_{0}) \)--the category of all finite
dimensional unitary representations of \( \mathbb{G} \rtimes \Lambda_{0} \), to
the class of covariant systems of representations subordinate to
\( \Lambda_{0} \), thereby getting a rigid \( C^{\ast} \)-tensor category
\( \mathcal{CSR}_{\Lambda_{0}} \) whose objects are CSRs subordinate to
\( \Lambda_{0} \).

To make this transport of categorical structures less tautological, we make a
convenient characterization of the morphisms in
\( \mathcal{CSR}_{\Lambda_{0}} \).

\begin{prop}
  \label{prop:50b58e1fd8af4aa0}
  Fix a general isotropy subgroup \( \Lambda_{0} \) of \( \Lambda \). For
  \( i = 1, 2 \), let \( \mathbf{S}_{i} = (\mathscr{H}_{i}, u_{i}, w_{i}) \) be
  a CSR subordinate to \( \Lambda_{0} \),
  \( U_{i} = {(u_{i})}_{12}{(w_{i})}_{13} \) the corresponding unitary
  representation of \( \mathbb{G} \rtimes \Lambda_{0} \),
  \( S \in \mathcal{B}(\mathscr{H}_{1}, \mathscr{H}_{2}) \). Then
  \( S \in \morph_{\mathbb{G} \rtimes \Lambda_{0}}\bigl(U_{1}, U_{2}\bigr) \) if
  and only if
  \begin{equation}
    \label{eq:0f111a651ccdef1c}
    S \in \morph_{\mathbb{G}}(u_{1}, u_{2}) \cap \morph_{\Lambda_{0}}(w_{1},
    w_{2}).
  \end{equation}
\end{prop}
\begin{proof}
  The condition is easily seen to be sufficient. Indeed, if condition
  \eqref{eq:0f111a651ccdef1c} holds, then
  \begin{equation}
    \label{eq:4531bf62f73243c8}
    (S \otimes 1) w_{1} = w_{2} (S \otimes 1), \quad \\
    (S \otimes 1) u_{1} = u_{2} (S \otimes 1).
  \end{equation}
  Thus
  \begin{equation}
    \label{eq:4cf21826d634f124}
    \begin{split}
      (S \otimes 1 \otimes 1) U_{1} &= (S \otimes 1 \otimes 1) {(u_{1})}_{12}
      {(w_{1})}_{13} %
      = {(u_{2})}_{12} (S \otimes 1 \otimes 1) {(w_{1})}_{13} \\
      &= {(u_{2})}_{12} {(w_{2})}_{13} (S \otimes 1 \otimes 1) %
      = U_{2} (S \otimes 1 \otimes 1).
    \end{split}
  \end{equation}
  This means exactly \( S \in \morph_{\mathbb{G}}(U_{1}, U_{2}) \).

  To show the necessity of this condition, let
  \( \epsilon_{\mathbb{G}} \colon \pol(\mathbb{G}) \rightarrow \mathbb{C} \) be
  the counit of the Hopf-\( \ast \)-algebra,
  \( \epsilon_{\Lambda_{0}} \colon C( \Lambda_{0} ) \rightarrow \mathbb{C} \)
  the counit for the Hopf \( \ast \)-algebra \( C(\Lambda_{0}) \). Since
  \( U_{i} \in \mathcal{B}(\mathscr{H}_{i}) \otimes \pol(\mathbb{G}) \otimes
  C(\Lambda_{0}) \) for \( i = 1,2 \) and
  \( S \in \morph_{\mathbb{G} \rtimes \Lambda_{0}}(U_{1}, U_{2}) \), we have
  \begin{equation}
    \label{eq:8ae967a0ff2d6b28}
    (S \otimes 1 \otimes 1)U_{1} = U_{2}(S \otimes 1 \otimes 1).
  \end{equation}
  Applying \( \id \otimes \id \otimes \epsilon_{\Lambda_{0}} \) on both sides
  of~\eqref{eq:8ae967a0ff2d6b28} yields
  \begin{equation}
    \label{eq:f4654de458889d3a}
    (S \otimes 1) u_{1} = u_{2}(S \otimes 1),
  \end{equation}
  which means \( S \in \morph_{\mathbb{G}}(u_{1}, u_{2}) \). Applying
  \( \id \otimes \epsilon_{\mathbb{G}} \otimes \id \) on both sides
  of~\eqref{eq:8ae967a0ff2d6b28} yields
  \begin{equation}
    \label{eq:968fbbfe5ba6c32d}
    (S \otimes 1) w_{1} = w_{1} (S \otimes 1),
  \end{equation}
  which means \( S \in \morph_{\Lambda_{0}}(w_{1}, w_{2}) \).
\end{proof}

We now define a pair of functors,
\begin{displaymath}
  \mathscr{R}_{\Lambda_{0}} \colon \mathcal{CSR}_{\Lambda_{0}}
  \rightarrow \repcat(\mathbb{G} \rtimes
  \Lambda_{0}) \qquad\text{ and }\qquad
  \mathscr{S}_{\Lambda_{0}} \colon
  \repcat(\mathbb{G} \rtimes \Lambda_{0}) \rightarrow
  \mathcal{CSR}_{\Lambda_{0}}
\end{displaymath}
between \( \mathcal{CSR}_{\Lambda_{0}} \) and
\( \repcat(\mathbb{G} \rtimes \Lambda_{0}) \) that reflects the transport of
categorical structures discussed above. On the object level, for any
\( (\mathscr{H}, u, w) \in \mathcal{CSR}_{\Lambda_{0}} \), let
\( \mathscr{R}_{\Lambda_{0}}(u, w) \) be the representation \( u_{12}w_{13} \)
of \( \mathbb{G} \rtimes \Lambda_{0} \) on \( \mathscr{H} \); for any unitary
representation \( U \in \repcat(\mathbb{G} \rtimes \Lambda_{0}) \) on
\( \mathscr{H}_{U} \), let \( \mathscr{S}_{\Lambda_{0}}(U) \) be the CSR
\( (\mathscr{H}_{U}, U_{\mathbb{G}}, U_{\Lambda_{0}}) \) where
\( U_{\mathbb{G}} \) (resp.\ \( U_{\Lambda_{0}} \)) is the restriction of
\( U \) onto \( \mathbb{G} \) (resp.\ \( \Lambda_{0} \)). On the morphism level,
both \( \mathscr{R}_{\Lambda_{0}} \) and \( \mathscr{S}_{\Lambda_{0}} \) act as
identity. By Proposition~\ref{prop:50b58e1fd8af4aa0} and
Proposition~\ref{prop:3693cca5392c7dd3}, \( \mathscr{R}_{\Lambda_{0}} \) and
\( \mathscr{S}_{\Lambda_{0}} \) are indeed well-defined functors inverses to
each other, and they are fiber functors (exact unitary tensor
functors~\cite[\S\S2.1, 2.2]{MR3204665}) simply because the rigid
\( C^{\ast} \)-tensor category structure on \( \mathcal{CSR}_{\Lambda_{0}} \) is
transported from that of \( \repcat(\mathbb{G} \rtimes \Lambda_{0}) \) via
\( \mathscr{S}_{\Lambda_{0}} \).

\begin{prop}
  \label{prop:ce7ae0d12013743b}
  For \( i = 1, 2 \), let
  \( \mathbf{S}_{i} = (\mathscr{H}_{i}, u_{i}, w_{i}) \in
  \mathcal{CSR}_{\Lambda_{0}} \),
  \( U_{i} = \mathscr{R}_{\Lambda_{0}}(\mathbf{S}_{i}) \in \repcat(\mathbb{G}
  \rtimes \Lambda_{0}) \), then
  \begin{equation}
    \label{eq:da2e2f508cac451d}
    \mathscr{S}_{\Lambda_{0}}(U_{1} \times U_{2}) %
    = (\mathscr{H}_{1} \otimes \mathscr{H}_{2},
    u_{1} \times u_{2}, w_{1} \times w_{2}) %
    = \mathbf{S}_{1} \otimes \mathbf{S}_{2} .
  \end{equation}
\end{prop}
\begin{proof}
  By definition of the tensor product of representations,
  \( U_{1} \times U_{2} \) is the representation of
  \( \mathbb{G} \rtimes \Lambda_{0} \) defined by
  \begin{equation}
    \label{eq:aa960443aa5b113c}
    U_{1} \times U_{2} = {(U_{1})}_{134} {(U_{2})}_{234}
    \in \mathcal{B}(\mathscr{H}_{1}) \otimes
    \mathcal{B}(\mathscr{H}_{2})
    \otimes \pol(\mathbb{G}) \otimes C(\Lambda_{0}),
  \end{equation}
  where we identified
  \( \mathcal{B}(\mathscr{H}_{1}) \otimes \mathcal{B}(\mathscr{H}_{2}) \) with
  \( \mathcal{B}(\mathscr{H}_{1} \otimes \mathscr{H}_{2}) \) canonically.

  The restriction of \( U_{1} \times U_{2} \) onto \( \mathbb{G} \) is
  \begin{equation}
    \label{eq:a6890a05a4d81dac}
    \begin{split}
      & \leadmathskip (\id \otimes \id \otimes \id \otimes
      \epsilon_{\Lambda_{0}})(U_{1} \times
      U_{2}) \\
      &= (\id \otimes \id \otimes \id \otimes
      \epsilon_{\Lambda_{0}})({(U_{1})}_{134}) %
      (\id \otimes \id \otimes \id \otimes
      \epsilon_{\Lambda_{0}})({(U_{2})}_{234}) \\
      &= {(u_{1})}_{13}{(u_{2})}_{23} %
      = u_{1} \times u_{2} \in \mathcal{B}(\mathscr{H}_{1}) \otimes
      \mathcal{B}(\mathscr{H}_{2}) \otimes \pol(\mathbb{G}).
    \end{split}
  \end{equation}
  Similarly, the restriction of \( U_{1} \times U_{2} \) onto \( \Lambda_{0} \)
  is
  \begin{equation}
    \label{eq:19783806a77c993e}
    \begin{split}
      & \leadmathskip (\id \otimes \id \otimes \epsilon_{\mathbb{G}} \otimes
      \id)
      (U_{1} \times U_{2}) \\
      &= (\id \otimes \id \otimes \epsilon_{\mathbb{G}} \otimes
      \id)({(U_{1})}_{134}) %
      (\id \otimes \id \otimes \epsilon_{\mathbb{G}} \otimes \id)
      ({(U_{2})}_{234}) \\
      &= {(w_{1})}_{13}{(w_{(2)})}_{23} = w_{1} \times w_{2} %
      \in \mathcal{B}(\mathscr{H}_{1}) \otimes \mathcal{B}(\mathscr{H}_{2})
      \otimes C(\Lambda_{0}).
    \end{split}
  \end{equation}
  Now~\eqref{eq:da2e2f508cac451d} follows from \eqref{eq:a6890a05a4d81dac},
  \eqref{eq:19783806a77c993e} and the definition of the tensor product in
  \( \mathcal{CSR}_{\Lambda_{0}} \).
\end{proof}

\begin{prop}
  \label{prop:d6600a425936cd86}
  For \( i = 1, 2 \), let
  \( \mathbf{S}_{i} = (\mathscr{H}, u_{i}, w_{i}) \in
  \mathcal{CSR}_{\Lambda_{0}} \),
  \( U_{i} := \mathscr{R}_{\Lambda_{0}}(\mathbf{S}_{i}) \in \repcat(\mathbb{G}
  \rtimes \Lambda_{0}) \), then
  \begin{equation}
    \label{eq:7bca88cf1365f0f2}
    \mathscr{S}_{\Lambda_{0}}(U_{1} \oplus U_{2})
    = (\mathscr{H}_{1} \oplus \mathscr{H}_{2},
    u_{1} \oplus u_{2}, w_{1} \oplus w_{2})
    = \mathbf{S}_{1} \oplus \mathbf{S}_{2}.
  \end{equation}
\end{prop}
\begin{proof}
  The proof use the same restriction technique as in the proofs of
  Proposition~\ref{prop:50b58e1fd8af4aa0} and
  Proposition~\ref{prop:ce7ae0d12013743b}, which is even simpler in this case.
\end{proof}

Until now, we've shown that the morphisms, tensor products, and direct sums all
behave as expected in \( \mathcal{CSR}_{\Lambda_{0}} \). The description of the
dual of a CSR when \( \mathbb{G} \) is of non-Kac type requires a bit further
work on the so-called modular operator, as we presently discuss.

Recall that the contragredient representation \( U^{c} \) of a \emph{unitary}
representation \( U \) of \( \mathbb{G} \rtimes \Lambda_{0} \) on some finite
dimensional Hilbert space \( \mathscr{H} \) is defined as
\( U^{c} = (j \otimes \id_{\pol(\mathbb{G}) \otimes C(\Lambda_{0})})(U^{\ast})
\), where
\( j \colon \mathcal{B}(\mathscr{H}) \rightarrow
\mathcal{B}(\overline{\mathscr{H}}) \) is defined as
\( T \mapsto \overline{T^{\ast}} \), with \( \overline{\mathscr{H}} \) being the
conjugate Hilbert space of \( \mathscr{H} \), and \( \overline{T^{\ast}} \)
meaning \( T^{\ast} \) viewed as a linear mapping from
\( \overline{\mathscr{H}} \) to \( \overline{\mathscr{H}} \).  Note that
\( j : \mathcal{B}(\mathscr{H}) \rightarrow \mathcal{B}(\overline{\mathscr{H}})
\), \( T \mapsto \overline{T^{\ast}} \) is linear, antimultiplicative and
positive (in particular, it preserves adjoints). If \( \mathbb{G} \) is of
non-Kac type, so is \( \mathbb{G} \rtimes \Lambda_{0} \) by
Proposition~\ref{prop:22073f2b39059e9e}, in which case \( U^{c} \) might not be
unitary, which is exactly why the ``modular'' operator \( \rho_{U} \) is
necessary to express the dual object of \( \mathscr{S}_{\Lambda_{0}}(U) \) in
\( \mathcal{CSR}_{\Lambda_{0}} \) as presented in
Proposition~\ref{prop:1ecfd88218e081c9}.

\begin{prop}
  \label{prop:1ecfd88218e081c9}
  Let \( \mathbf{S} = (\mathscr{H}, u, w) \in \mathcal{CSR}_{\Lambda_{0}} \),
  \( U = \mathscr{R}_{\Lambda_{0}}(\mathbf{S}) \in \repcat(\mathbb{G} \rtimes
  \Lambda_{0}) \), \( U^{c} \) the contragredient representation of \( U \) on
  the conjugate space \( \overline{\mathscr{H}} \) of \( \mathscr{H} \). If
  \( \rho_{U} \) is the unique invertible positive operator in
  \( \morph_{\mathbb{G} \rtimes \Lambda_{0}}(U, U^{cc}) \) (which we call
  modular operator) such that
  \( \tr( \cdot \, \rho_{U}) = \tr( \cdot \, \rho_{U}^{-1}) \) on
  \( \selfmorph_{\mathbb{G} \rtimes \Lambda_{0}}(U) \), so that
  \begin{equation}
    \label{eq:1292a1f6dcc336b4}
    \overline{U}
    = {\bigl\{{[j(\rho_{U})]}^{1/2} \otimes 1_{\pol(\mathbb{G})} \otimes
      1_{C(\Lambda_{0})}\bigr\}} U^{c} {\bigl\{{[j(\rho_{U})]}^{-1/2} \otimes
      1_{\pol(\mathbb{G})} \otimes 1_{C(\Lambda_{0})}\bigr\}}
  \end{equation}
  is the conjugate representation of \( U \), then the dual of \( \mathbf{S} \)
  is given by \( \overline{\mathbf{S}} = (\overline{\mathscr{H}}, u', w') \),
  where
  \begin{equation}
    \label{eq:687776145feb7bb0}
    \begin{split}
      u' %
      &= ({j(\rho_{U})}^{1/2} \otimes 1) u^{c} ({j(\rho_{U})}^{-1/2} \otimes 1),
      \\
      w' %
      &= ({j(\rho_{U})}^{1/2} \otimes 1) w^{c} ({j(\rho_{U})}^{-1/2} \otimes 1).
    \end{split}
  \end{equation}
  Note that \( w^{c} = \overline{w} \) as \( \Lambda \) is a finite (compact)
  group. In particular, if \( \mathbb{G} \) is of Kac-type, then
  \( u^{c} = \overline{u} \), \( \rho_{U} = 1 \), and
  \( \overline{\mathbf{S}} = (\overline{\mathscr{H}}, \overline{u},
  \overline{w}) \).
\end{prop}
\begin{proof}
  By definition, \( \overline{S} = \mathscr{S}_{\Lambda_{0}}(\overline{U}) \),
  thus
  \begin{displaymath}
    \begin{split}
      u' &= (\id_{\mathcal{B}(\mathscr{H})} \otimes \id_{\pol(\mathbb{G})}
      \otimes \epsilon_{\Lambda_{0}}) %
      (\overline{U}) \\
      &= (\id \otimes \id \otimes \epsilon_{\Lambda_{0}}) \bigl[
      ({j(\rho_{U})}^{1/2} \otimes 1 \otimes 1) U^{c}
      ({j(\rho_{U})}^{-1/2} \otimes 1 \otimes 1)\bigr] \\
      &= (\id \otimes \id \otimes \epsilon_{\Lambda_{0}}) \Bigl(
      ({j(\rho_{U})}^{1/2} \otimes 1 \otimes 1) \\
      & \dmslmidskip [(j \otimes \id_{\pol(\mathbb{G})} \otimes
      \id_{C(\Lambda_{0})})
      (U^{\ast})]({j(\rho_{U})}^{-1/2} \otimes 1 \otimes 1)\Bigr) \\
      &= ({j(\rho_{U})}^{1/2} \otimes 1) \bigl[(j \otimes \id \otimes
      \epsilon_{\Lambda_{0}})(U^{\ast})\bigr]
      ({j(\rho_{U})}^{-1/2} \otimes 1) \\
      &= ({j(\rho_{U})}^{1/2} \otimes 1) {\bigl[(j \otimes \id \otimes
        \epsilon_{\Lambda_{0}})(U)\bigr]}^{\ast}
      ({j(\rho_{U})}^{-1/2} \otimes 1) \\
      &= ({j(\rho_{U})}^{1/2} \otimes 1) {\bigl[(j \otimes \id)(u)\bigr]}^{\ast}
      ({j(\rho_{U})}^{-1/2} \otimes 1) \\
      &= ({j(\rho_{U})}^{1/2} \otimes 1) \bigl[(j \otimes \id)(u^{\ast})\bigr]
      ({j(\rho_{U})}^{-1/2} \otimes 1) \\
      &= ({j(\rho_{U})}^{1/2} \otimes 1) u^{c} ({j(\rho_{U})}^{-1/2} \otimes 1).
    \end{split}
  \end{displaymath}
  The expression for \( w' \) is proved analogously by applying
  \( \id_{\mathcal{B}(\mathscr{H})} \otimes \epsilon_{\pol(\mathbb{G})} \otimes
  \id_{C(\Lambda_{0})} \) on~\eqref{eq:1292a1f6dcc336b4}.  Finally, if
  \( \mathbb{G} \) is of Kac-type, then \( \rho_{U} = \id_{\mathscr{H}} = 1 \).
\end{proof}

\begin{rema}
  \label{rema:d7b7b5149a93355d}
  The ``modular'' operator \( \rho_{U} \) of the representation \( U \) is
  derived from the representation theory of \( \mathbb{G} \rtimes \Lambda_{0} \)
  instead of the representation theory of \( \mathbb{G} \) and (projective)
  representation theory of \( \Lambda_{0} \). This makes the description of
  \( \overline{\mathbf{S}} \) in Proposition~\ref{prop:1ecfd88218e081c9} quite
  unsatisfactory in the non-unimodular case. This being said, we point out that
  as far as the fusion rules of \( \mathbb{G} \rtimes \Lambda_{0} \) are
  concerned, the duals of a sufficiently large family of CSRs admit a much more
  satisfactory description (see Proposition~\ref{prop:ef1711719362d940}).
\end{rema}

Of course, the description of the dual in \( \mathcal{CSR}_{\Lambda_{0}} \) is
much easier if \( \mathbb{G} \) is of Kac-type, as is clearly seen from the last
part of Proposition~\ref{prop:1ecfd88218e081c9}.

\section{Group actions and projective representations}
\label{sec:0bd5d546e466b2bc}

Fix a \( \Lambda_{0} \in \giso(\Lambda) \).  Via the functors
\( \mathscr{S}_{\Lambda_{0}} \) and \( \mathscr{R}_{\Lambda_{0}} \), we see that
the problem classifying of irreducible representations of
\( \mathbb{G} \rtimes \Lambda_{0} \) are essentially the same as classifying
simple CSRs in \( \mathcal{CSR}_{\Lambda_{0}} \). Thus for the moment, it might
be too much to hope there exists a satisfactory description of \emph{all} simple
CSRs in \( \mathcal{CSR}_{\Lambda_{0}} \). However, as we will see in
\S~\ref{sec:a9de7f0149f3e3a1}, if we restrict our attention to the so-called
\emph{stably pure} simple CSRs in \( \mathcal{CSR}_{\Lambda_{0}} \), then such a
description is indeed achievable via the theory of unitary projective
representations of \( \Lambda_{0} \).  This section studies how such projective
representations arise naturally from the action of \( \Lambda \) on irreducible
representations of \( \mathbb{G} \), as well as establishes some basic
properties of these projective representations. The results here will be used in
\S~\ref{sec:a9de7f0149f3e3a1} to describe the structure of stably pure CSRs in
\( \mathcal{CSR}_{\Lambda_{0}} \) (Proposition~\ref{prop:f8d2ac97bef564cc} and
Proposition~\ref{prop:1911ec535b058627}).

We begin with a simple observation which is a trivial quantum analogue of one of
the most basic ingredients of the Mackey analysis. Let \( U_{\mathbb{G}} \) be a
unitary representation of \( \mathbb{G} \) on some finite dimensional Hilbert
space \( \mathscr{H} \). Since
\( \alpha^{\ast} \colon \Lambda \rightarrow \aut\bigl( C(\mathbb{G}), \Delta
\bigr) \) is an antihomomorphism of groups, we know that
\( {(\id_{\mathcal{B}(\mathscr{H})} \otimes
  \alpha_{r^{-1}}^{\ast})}(U_{\mathbb{G}}) \) is again a unitary representation
of \( \mathbb{G} \) on the same space \( \mathscr{H} \), and we denote this new
representation by \( r \cdot U_{\mathbb{G}} \) as we did in
Proposition~\ref{prop:56963e422c66f26b}. One checks that
\( (r s) \cdot U_{\mathbb{G}} = r \cdot (s \cdot U_{\mathbb{G}}) \). Thus this
defines a left action of the group \( \Lambda \) on the (proper) class of all
unitary representation of \( \mathbb{G} \), which is easily seen to preserve
irreducibility and pass to a well-defined action of \( \Lambda \) on the
\emph{set} \( \irr(\mathbb{G}) \) by letting \( r \cdot [u] = [r \cdot u] \),
where \( r \in \Lambda \), \( u \) is an irreducible unitary representation of
\( \mathbb{G} \) and \( [u] \) is the equivalence class of \( u \) in
\( \irr(\mathbb{G}) \). Take another unitary representation \( W_{\mathbb{G}} \)
of \( \mathbb{G} \) on some other finite dimensional Hilbert space
\( \mathscr{K} \). For any \( r, s \in \Lambda \) and any
\( T \in \mathcal{B}(\mathscr{H}, \mathscr{K}) \), we have
\begin{equation}
  \label{eq:f5ecd1bd13d50e74}
  \begin{split}
    & T \in \morph_{\mathbb{G}}(r \cdot U_{\mathbb{G}}, W_{\mathbb{G}}) \\
    \iff & W_{\mathbb{G}}(T \otimes 1) = (T \otimes 1) (\id \otimes
    \alpha_{r^{-1}}^{\ast})
    (U_{\mathbb{G}}) \\
    \iff & [(\id \otimes \alpha_{s^{-1}}^{\ast})(W_{\mathbb{G}})] (T \otimes 1)
    = (T \otimes 1) (\id \otimes \alpha_{{(sr)}^{-1}}^{\ast})
    (U_{\mathbb{G}}) \\
    \iff & T \in \morph_{\mathbb{G}}(sr \cdot U_{\mathbb{G}}, s \cdot
    W_{\mathbb{G}}).
  \end{split}
\end{equation}

Now take any irreducible unitary representation \( u \) of \( \mathbb{G} \) on
some finite dimensional Hilbert space \( \mathscr{H} \). Let
\( x = [u] \in \irr(\mathbb{G}) \), and
\begin{displaymath}
  \Lambda_{x} = \set*{r \in \Lambda \given r \cdot x = x},
\end{displaymath}
i.e.\ \( \Lambda_{x} \) is the isotropy subgroup of \( \Lambda \) fixing
\( x \). Then for any \( r_{0} \in \Lambda_{x} \), \( u \) and
\( r_{0} \cdot u \) are equivalent by definition, hence there exists a unitary
\( V(r_{0}) \in \mathcal{U}(\mathscr{H}) \) intertwining \( r_{0} \cdot u \) and
\( u \), in other words,
\begin{equation}
  \label{eq:329c09668e71d843}
  \bigl(V(r_{0}) \otimes 1\bigr)( \id \otimes \alpha^{\ast}_{r_{0}^{-1}})(u)
  = u \bigl( V(r_{0})  \otimes 1 \bigr),
\end{equation}
which is clearly equivalent to
\begin{equation}
  \label{eq:ff84fdb6222c5b10}
  \forall r_{0} \in \Lambda_{x}, \quad
  \bigl(V(r_{0}) \otimes 1\bigr) u %
  = [(\id \otimes \alpha_{r_{0}}^{\ast}) (u)] \bigl(V(r_{0}) \otimes 1 \bigr).
\end{equation}
It is remarkable that~\eqref{eq:ff84fdb6222c5b10} takes exactly the same form as
the covariance condition~\eqref{eq:1611871ef2500d9a} when we define covariant
representations in \S~\ref{sec:cac0f6b0354689f8}. Now if we choose a
\begin{equation}
  \label{eq:df18a33e72dd2039}
  V(r_{0}) \in \morph_{\mathbb{G}}(r_{0} \cdot u,
  u) \cap \mathcal{U}(\mathscr{H})
\end{equation}
for each \( r_{0} \in \Lambda_{x} \), then for any \( s_{0} \in \Lambda_{x} \),
by~\eqref{eq:f5ecd1bd13d50e74}, we have
\begin{displaymath}
  V(r_{0}) \in \morph_{\mathbb{G}}(s_{0}r_{0} \cdot u, s_{0} \cdot
  u), \quad
  V(s_{0}) \in \morph_{\mathbb{G}}(s_{0} \cdot u,
  u), \quad
  V(s_{0}r_{0}) \in \morph_{\mathbb{G}}(s_{0}r_{0} \cdot u, u),
\end{displaymath}
thus
\begin{equation}
  \label{eq:12b2a6a90456dd72}
  \forall r_{0}, s_{0} \in \Lambda_{x}, \quad
  V(s_{0}r_{0}) {[V(r_{0})]}^{\ast}{[V(s_{0})]}^{\ast} \in
  \morph_{\mathbb{G}}(u, u) \cap
  \mathcal{U}(\mathscr{H})
  = \mathbb{T} \cdot \id_{\mathscr{H}}.
\end{equation}
This means that \( V \colon \Lambda_{x} \rightarrow \mathcal{U}(\mathscr{H}) \)
is a unitary \textbf{projective representation}
(Definition~\ref{defi:530f76c332f4637b}) of \( \Lambda_{x} \) on
\( \mathscr{H} \), which satisfies the covariant
condition~\eqref{eq:ff84fdb6222c5b10} for each \( r_{0} \in \Lambda_{x} \).

To facilitate our discussion, we digress now to give a brief summary of some
basic terminologies of the theory of group cohomology which we will use (cf.\
\cite{MR1324339}). We regard \( \mathbb{T} \) as a trivial module over any
finite group when considering unitary projective representations of finite
groups. For any finite group \( \Gamma \), an \( n \)-cochain on \( \Gamma \)
with coefficients in \( \mathbb{T} \), or simply an \( n \)-cochain (on
\( \Gamma \)), as we won't consider coefficient module other that the trivial
module \( \mathbb{T} \), is a mapping from the \( n \)-fold product
\( \Gamma^{n} = \Gamma \times \cdots \times \Gamma \) to \( \mathbb{T} \). Let
\( C^{n}(\Gamma, \mathbb{T}) \) be the abelian group of \( n \)-cochains on
\( \Gamma \), \( Z^{2}(\Gamma, \mathbb{T}) \) the subgroup of \( 2 \)-cocycles
on \( \Gamma \), i.e.\ mappings
\( \omega \colon \Gamma \times \Gamma \rightarrow \mathbb{T} \) satisfying the
cocycle condition
\begin{equation}
  \label{eq:b0d6e1765a0eb92b}
  \forall r, s, t \in \Gamma, \quad \omega(r, st) \omega(s, t)
  = \omega(r, s) \omega(rs, t).
\end{equation}
The mapping
\begin{equation}
  \label{eq:774f5d351fbf9b30}
  \begin{split}
    \delta \colon C^{1}( \Gamma, \mathbb{T} )
    & \rightarrow Z^{2}(\Gamma, \mathbb{T}) \\
    \mathsf{b} & \mapsto \left\{(r,s) \in \Gamma \times \Gamma \mapsto
      \frac{\mathsf{b}(r)\mathsf{b}(s)}{\mathsf{b}(rs)} \right\}
  \end{split}
\end{equation}
is easily checked to be a well-defined group morphism. We use
\( B^{2}(\Gamma, \mathbb{T}) \) to denote the image of \( \delta \), and the
\( 2 \)-cocycles in \( B^{2}(\Gamma, \mathbb{T}) \) are called
\( 2 \)-coboundaries of \( \Gamma \). The quotient group
\( Z^{2}(\Gamma, \mathbb{T}) / B^{2}(\Gamma, \mathbb{T}) \) is called the second
cohomology group of \( \Gamma \) with coefficients in the trivial
\( \Gamma \)-module \( \mathbb{T} \), and is denoted by
\( H^{2}(\Gamma, \mathbb{T}) \). Elements in \( H^{2}(\Gamma, \mathbb{T}) \) are
called cohomology class. Note that \( \ker(\delta) \) is exactly the group of
characters on \( \Gamma \), i.e.\ group morphisms from \( \Gamma \) to
\( \mathbb{T} \).

\begin{defi}
  \label{defi:530f76c332f4637b}
  Let \( \Gamma \) be a group, \( \mathscr{H} \) a finite dimensional Hilbert
  space, a projective representation of \( \Gamma \) on \( \mathscr{H} \) is a
  mapping \( V : \Gamma \to \mathcal{U}(\mathscr{H}) \) such that
  \( V(e_{\Gamma}) = \id_{\mathscr{H}} \), and there exists a
  \( 2 \)\nobreakdash-cochain \( \omega \in C^{2}(\Gamma, \mathbb{T}) \), such
  that
  \begin{equation}
    \label{eq:60c95264176fa284}
    \forall r,s \in \Gamma, \qquad
    \omega(r,s) V(r,s) = V(r) V(s).
  \end{equation}
  It is easy to check that such \( \omega \) is uniquely determined by \( V \),
  and it is in fact a \( 2 \)\nobreakdash-cocylce, with the additional property
  (which follows from our assumption \( V(e_{\Gamma}) = \id_{\mathscr{H}} \))
  that
  \begin{equation}
    \label{eq:374eb20adafeefb6}
    \forall \gamma \in \Gamma, \qquad \omega(e_{\Gamma}, \gamma)
    = \omega(\gamma, e_{\Gamma}) = 1 \in \mathbb{T}.
  \end{equation}
  We call \( \omega \) the cocylce (or Schur multiplier after Schur who
  introduced them in his work on projective representations
  \cite{schur1904darstellung}) of the projective representation \( V \).
\end{defi}

We will freely use the character theory and the Peter Weyl theory of projective
representations of finite groups, and we refer the reader to \cite{MR3299063}
for the proofs.

We track here the following easy results for convenience of the reader.

\begin{lemm}
  \label{lemm:cc4e3d9cf9663b2b}
  Let \( \Gamma \) be a finite group,
  \( V \colon \Gamma \rightarrow \mathcal{U}(\mathscr{H}) \) a finite
  dimensional unitary projective representation of \( \Gamma \) with cocycle
  \( \omega \). If \( \omega' \in [\omega] \in H^{2}(\Gamma, \mathbb{T}) \),
  then there exists a mapping
  \( \mathsf{b} \colon \Gamma \rightarrow \mathbb{T} \), such that
  \( \mathsf{b} V \colon \Gamma \rightarrow \mathcal{U}(\mathscr{H}) \),
  \( \gamma \mapsto \mathsf{b}(\gamma)V(\gamma) \) is a unitary projective
  representation with cocycle \( \omega' \).
\end{lemm}
\begin{proof}
  Since \( \omega' \in [\omega] \), there is a mapping
  \( \mathsf{b} \colon \Gamma \rightarrow \mathbb{T} \) such that
  \( \omega' = (\delta \mathsf{b}) \omega \), and obviously, \( \mathsf{b}V \)
  is a unitary projective representation with
  \( (\delta \mathsf{b}) \omega = \omega' \) as its cocycle.
\end{proof}

\begin{lemm}
  \label{lemm:62fa40c770de34e0}
  Let \( \Gamma \) be a finite group,
  \( V \colon \Gamma \rightarrow \mathcal{U}(\mathscr{H}) \) a finite
  dimensional unitary projective representation of~ \( \Gamma \) with cocycle
  \( \omega \), and let \( \mathsf{b} \colon \Gamma \rightarrow \mathbb{T} \) be
  an arbitrary mapping. The following hold:
  \begin{enumerate}
  \item \label{item:084103c034528bd4}
    \( \mathsf{b}V \colon \Gamma \rightarrow \mathcal{U}(\mathscr{H}) \),
    \( \gamma \mapsto \mathsf{b}( \gamma ) V( \gamma ) \) is a projective
    representation with cocycle \( ( \delta \mathsf{b} ) \omega \);
  \item \label{item:427971dc4329cf18} \( \mathsf{b}V \) and \( V \) have the
    same cocycle if and only if \( \mathsf{b} \in \ker( \delta ) \), i.e.\
    \( \mathsf{b} \) is a character of \( \Gamma \);
  \item \label{item:41687ad768888c74} \( \mathsf{b} V \) is irreducible if and
    only if \( V \) is irreducible.
  \end{enumerate}
\end{lemm}
\begin{proof}
  It is clear that \ref{item:084103c034528bd4} and \ref{item:427971dc4329cf18}
  are direct consequences of the relevant definitions. We now prove
  \ref{item:41687ad768888c74}. If we denote the character of \( V \) by
  \( \chi_{V} \), then the character of \( \mathsf{b}V \) is
  \( \mathsf{b} \chi_{V} \). Hence
  \begin{equation}
    \label{eq:9455d1b35405a8bb}
    \begin{split}
      & \leadmathskip \dim \morph_{\Gamma}(\mathsf{b}V, \mathsf{b}V) %
      = \frac{1}{\abs*{\Gamma}} \sum_{\gamma \in \Gamma}
      \overline{\mathsf{b}(\gamma) \chi_{V}(\gamma)} %
      \mathsf{b}(\gamma) \chi_{V}(\gamma) \\
      &= \frac{1}{\abs*{\Gamma}} \sum_{\gamma \in \Gamma}
      \overline{\chi_{V}(\gamma)} \chi_{V}(\gamma) %
      = \dim \morph_{\Gamma}(V, V),
    \end{split}
  \end{equation}
  and \( \mathsf{b}V \) is irreducible if and only if \( V \) is.
\end{proof}

\begin{rema}
  \label{rema:c0671b6740777559}
  If \( \mathsf{b} \) is a character of \( \Gamma \), and
  \( V \colon \Gamma \rightarrow \mathcal{U}(\mathscr{H}) \) an irreducible
  unitary projective representation, then \( \mathsf{b}V \) is also an
  irreducible unitary projective representation with the same cocycle as that of
  \( V \). Note that \( \abs*{\mathsf{b}(\gamma)} = 1 \) for all
  \( \gamma \in \Gamma \), we have
  \begin{equation}
    \label{eq:622c6963980432a7}
    \begin{split}
      & \leadmathskip \dim \morph_{\Gamma}(\mathsf{b}V, V) %
      = \frac{1}{\abs*{\Gamma}} \sum_{\gamma \in \Gamma}
      \overline{\mathsf{b}(\gamma)} \chi_{V}(\gamma)
      \overline{\chi_{V}(\gamma)} \\
      &\leq \frac{1}{\abs*{\Gamma}} \sum_{\gamma \in \Gamma} \chi_{V}(\gamma)
      \overline{\chi_{V}(\gamma)} = \dim \morph_{\Gamma}(V, V) = 1.
    \end{split}
  \end{equation}
  with equality holds if and only if \( \mathsf{b}(\gamma) = 1 \) whenever
  \( \chi_{V}(\gamma) \ne 0 \). If equality doesn't hold
  in~\eqref{eq:622c6963980432a7}, then
  \( \dim \morph_{\Gamma}(\mathsf{b}V, V) \) must be \( 0 \) since it is a
  natural number. Therefore, whenever \( \Gamma \) is not trivial, it is
  possible that \( \mathsf{b}V \) and \( V \) are irreducible unitary projective
  representations with the same cocycle but not equivalent. Thus one must be
  careful not to confuse our definition of projective representation with the
  more naive definition where one simply replaces
  \( \generallinear(\mathscr{H}) \) by \( \prglin(\mathscr{H}) \) as the target
  model group. For us, how we lift from \( \prglin(\mathscr{H}) \) to
  \( \generallinear(\mathscr{H}) \) \emph{does} matter, even if we keep the
  cocycle in the process.
\end{rema}

After this digression, we now resume our discussion. Using terminologies in the
theory of group cohomology, and regarding \( \mathbb{T} \) as the trivial
\( \Lambda_{x} \)-module, we see that the \( 2 \)-cocycle
\( \omega_{x} \in C^{2}(\Lambda_{x}, \mathbb{T}) \) of the unitary projective
representation \( V \) of \( \Lambda_{x} \) is determined up to a
\( 2 \)-boundary in \( B^{2}( \Lambda_{x}, \mathbb{T} ) \), because each unitary
operator \( V(r_{0}) \), \( r_{0} \in \Lambda_{x} \) is uniquely determined up
to a scalar multiple in \( \mathbb{T} \) (Schur's lemma plus the unitarity of
\( V(r_{0}) \)). In other words,
\( [\omega_{x}] \in H^{2}(\Lambda_{x}, \mathbb{T}) \) is a well-defined
cohomology class of \( \Lambda_{x} \) with coefficients in \( \mathbb{T} \).

Conversely, let \( u \) be an \emph{irreducible} unitary representation of
\( \mathbb{G} \) on some finite dimensional Hilbert space \( \mathscr{H} \), and
\( x = [u] \in \irr(\mathbb{G}) \). If \( \Lambda_{0} \) is a subgroup of
\( \Lambda \), \( V \colon \Lambda_{0} \rightarrow \mathcal{U}(\mathscr{H}) \) a
unitary projection representation of \( \Lambda_{0} \) such that \( u \) and
\( V \) satisfy the covariance condition~\eqref{eq:ff84fdb6222c5b10}, then
\begin{displaymath}
  \forall r_{0} \in \Lambda_{0}, \quad V(r_{0})
  \in \morph_{\mathbb{G}}(r_{0} \cdot u,
  u).
\end{displaymath}
In particular, \( \Lambda_{0} \) fixes \( x = [u] \) under the action
\( \Lambda \curvearrowright \irr(\mathbb{G}) \). Repeat the above reasoning
shows that~\eqref{eq:12b2a6a90456dd72} still holds.

We summarize the above discussion in the following proposition, which proves
slightly more.

\begin{prop}
  \label{prop:902dca3889872c81}
  Let \( u \) be an irreducible unitary representation of \( \mathbb{G} \) on
  some finite dimensional Hilbert space \( \mathscr{H} \),
  \( x = [u] \in \irr(\mathbb{G}) \), \( \Lambda_{x} \) the isotropy group
  fixing \( x \) (under the action
  \( \Lambda \curvearrowright \irr(\mathbb{G}) \)). For any
  \( r_{0} \in \Lambda_{x} \), choose a unitary \( V(r_{0}) \) according
  to~\eqref{eq:df18a33e72dd2039}. Then
  \begin{enumerate}
  \item \label{item:925fa8aeabb78e4f}
    \( V \colon \Lambda_{x} \rightarrow \mathcal{U}(\mathscr{H}) \),
    \( r_{0} \mapsto V(r_{0}) \) is a unitary projective representation
    satisfying the covariance condition~\eqref{eq:ff84fdb6222c5b10};
  \item \label{item:40e8eb076d946d6c} let
    \( \omega \in C^{2}(\Lambda_{0}, \mathbb{T}) \) be the \( 2 \)-cocycle of
    \( V \), then the cohomology class
    \( c_{x} \colon= [\omega] \in H^{2}(\Lambda_{x}, \mathbb{T}) \) depends only
    on \( x \), i.e.\ it does not depend on any particular choice of
    \( u \in x \).
  \end{enumerate}
  Conversely, if
  \( V_{0} \colon \Lambda_{0} \rightarrow \mathcal{U}(\mathscr{H}) \) is a
  unitary projective representation of some subgroup \( \Lambda_{0} \) of
  \( \Lambda \) that satisfies the covariance
  condition~\eqref{eq:ff84fdb6222c5b10}, then
  \begin{enumerate}[resume]
  \item \label{item:220033c1e00ecd65} for every \( r_{0} \in \Lambda_{0} \), the
    condition~\eqref{eq:df18a33e72dd2039} holds;
  \item \label{item:5240fb4427fd379a} \( \Lambda_{0} \subseteq \Lambda_{x} \);
  \item \label{item:193b3181a694f8e7} there is a choice of
    \( V \colon \Lambda_{x} \rightarrow \mathcal{U}(\mathscr{H}) \)
    satisfying~\eqref{eq:ff84fdb6222c5b10} such that
    \( V \vert_{\Lambda_{0}} = V_{0} \);
  \item \label{item:07624bf2611cbd70} let
    \( \omega_{0} \in C^{2}(\Lambda_{0}, \mathbb{T}) \) be the \( 2 \)-cocycle
    of \( V_{0} \), then \( [\omega_{0}] \) is the image of \( c_{x} \) under
    the morphism of groups
    \begin{displaymath}
      H^{2}( \Lambda_{0} \hookrightarrow \Lambda_{x} ) \colon
      H^{2}( \Lambda_{x}, \mathbb{T}) \rightarrow H^{2}(
      \Lambda_{0}, \mathbb{T}).
    \end{displaymath}
  \end{enumerate}
\end{prop}
\begin{proof}
  The above discussion already establishes \ref{item:925fa8aeabb78e4f},
  \ref{item:220033c1e00ecd65} and \ref{item:5240fb4427fd379a}. Assertion
  \ref{item:193b3181a694f8e7} follows from \ref{item:925fa8aeabb78e4f} and
  \ref{item:220033c1e00ecd65}, while \ref{item:07624bf2611cbd70} follows from
  \ref{item:193b3181a694f8e7}. Moreover, we've seen that
  \( [\omega] \in H^{2}(\Lambda_{x}, \mathbb{T}) \) does not depend on the
  choice of \( V \). For any \( w \in x \), there exists a unitary intertwiner
  \( U \in \morph_{\mathbb{G}}(u, w) \). It is trivial to check that
  \( V_{w}(r_{0}) = U V(r_{0}) U^{\ast} \) defines a unitary projective
  representation of \( \Lambda_{x} \) such that
  \begin{displaymath}
    V_{w}(r_{0}) \in \morph_{\mathbb{G}}(r_{0} \cdot w, w).
  \end{displaymath}
  Since \( V_{w} \) and \( V \) are unitarily equivalent projective
  representations of \( \Lambda_{x} \), the \( 2 \)-cocycle of \( V_{w} \)
  coincides with \( \omega \)---the \( 2 \)-cocycle of \( V \). This proves that
  \( c_{x} = [\omega] \in H^{2}(\Lambda_{x}, \mathbb{T}) \) indeed depends only
  on \( x \) and not on any particular choice of \( u \in x \). This proves
  \ref{item:40e8eb076d946d6c} and finishes the proof of the proposition.
\end{proof}

\begin{defi}
  \label{defi:7ea8675d7dbf16c8}
  Using the notations in Proposition~\ref{prop:902dca3889872c81}, we call the
  cohomology class \( [\omega] \in H^{2}(\Lambda_{x}, \mathbb{T}) \) the
  cohomology class associated with \( x = [u] \in \irr(\mathbb{G}) \), and we
  denote \( [\omega] \) by \( c_{x} \). If \( \Lambda_{0} \) is a subgroup of
  \( \Lambda_{x} \), the cohomology class
  \( [\omega_{0}] \in H^{2}(\Lambda_{0}, \mathbb{T}) \) is called the
  restriction of the cohomology class \( c_{x} \) on \( \Lambda_{0} \), and is
  denoted by \( c_{x, \Lambda_{0}} \).
\end{defi}

Obviously, \( c_{x, \Lambda_{0}} \) depends on \( \Lambda_{0} \) and \( x \),
and \( c_{x, \Lambda_{0}} = c_{x} \) if \( \Lambda_{0} = \Lambda_{x} \). To
apply the character theory of projective representations, we need to suitably
rescale the projective representations in question so that they share the same
cocycle (and not merely the same cohomology class for their cocycles). In the
case where the representation
\( u \in \mathcal{B}(\mathscr{H}) \otimes \pol(\mathbb{G}) \) of
\( \mathbb{G} \) is irreducible, and
\( V \colon \Lambda_{0} \rightarrow \mathcal{U}(\mathscr{H}) \) is a unitary
projective representation satisfying the covariance
condition~\eqref{eq:ff84fdb6222c5b10}, such a rescaling is implicit in the
choice of \( V(r_{0}) \in \morph_{\mathbb{G}}(r_{0} \cdot u, u) \) for each
\( r_{0} \in \Lambda_{0} \). However, Remark~\ref{rema:c0671b6740777559} tells
us we should take extra care if we want to talk about equivalence class of these
projective representations once we do the rescaling.

We finish this section with an easy result.
\begin{prop}
  \label{prop:80a4899eb993b414}
  Let \( x \in \irr(\mathbb{G}) \), \( u \in x \), \( \Lambda_{0} \) a subgroup
  of \( \Lambda_{x} \), \( c_{0} \in H^{2}(\Lambda_{0}, \mathbb{T}) \) is the
  image of the cohomology class \( c_{x} \in H^{2}( \Lambda_{x}, \mathbb{T}) \)
  associated with \( x \) under
  \( H^{2}(\Lambda_{0} \hookrightarrow \Lambda_{x}, \mathbb{T}) \).  Then for
  any \( 2 \)-cocycle \( \omega_{0} \in c_{0} \), there exists a unitary
  projective representation \( V \) of the isotropy subgroup \( \Lambda_{0} \)
  with cocycle \( \omega_{0} \), such that \( V \) and \( u \) are covariant,
  and such \( V \) is unique up to rescaling by a character of
  \( \Lambda_{0} \).
\end{prop}
\begin{proof}
  This is clear from Proposition~\ref{prop:902dca3889872c81},
  Lemma~\ref{lemm:cc4e3d9cf9663b2b} and Lemma~\ref{lemm:62fa40c770de34e0}.
\end{proof}

\section{Pure, stable, distinguished CSRs and representation parameters}
\label{sec:a9de7f0149f3e3a1}

Recall that for any finite dimensional representation \( u \) of
\( \mathbb{G} \), the support of \( u \), denoted by \( \supp(u) \), is the set
\begin{displaymath}
  \set*{x \in \irr(\mathbb{G}) \given \dim_{\mathbb{G}}
    \morph_{\mathbb{G}}(x, [u]) \neq 0}
\end{displaymath}
where \( [u] \) is the class of unitary representations of \( \mathbb{G} \)
equivalent to \( u \). We call \( u \) \emph{pure} if \( \supp(u) \) is a
singleton.

\begin{defi}
  \label{defi:db0dba45bc40ce46}
  Fix a \( \Lambda_{0} \in \giso(\Lambda) \),
  \( \mathbf{S} = (\mathscr{H}, u, w) \in \mathcal{CSR}_{\Lambda_{0}} \), we
  call \( \mathbf{S} \)
  \begin{itemize}
  \item \textbf{pure}, if \( u \) is pure;
  \item \textbf{stable}, if \( r \cdot [u] (= [r \cdot u]) = [u] \) for all
    \( r \in \Lambda_{0} \);
  \item \textbf{stably pure}, if it is both pure and stable;
  \item \textbf{maximally stable}, if
    \begin{displaymath}
      \Lambda_{0} = \set*{r \in \Lambda \given r \cdot [u] = [u]};
    \end{displaymath}
  \item \textbf{simple}, if \( \mathbf{S} \) is a simple object in
    \( \mathcal{CSR}_{\Lambda_{0}} \);
  \item \textbf{distinguished}, if it is maximally stable, pure and simple.
  \end{itemize}
\end{defi}

As remarked earlier, while it is not reasonable for the moment to hope for a
satisfactory description of all simple CSRs in
\( \mathcal{CSR}_{\Lambda_{0}} \), it is possible to describe simple CSRs that
are stably pure using unitary projective representations of \( \Lambda_{0} \).
Somewhat surprisingly, one can even describe all stably pure CSRs, even the
non-simple ones, in this way. To achieve the latter, we introduce the following
definitions, which are closely related to the materials in
\S~\ref{sec:0bd5d546e466b2bc}.

\begin{defi}
  \label{defi:0bfa2e98e54ead31}
  Let \( \Lambda_{0} \in \giso(\Lambda) \). Suppose \( u \) is a unitary
  representation of \( \mathbb{G} \) on some finite dimensional Hilbert space
  \( \mathscr{H} \), and
  \( V \colon \Lambda_{0} \rightarrow \mathcal{U}(\mathscr{H}) \) is a unitary
  projective representation of \( \Lambda_{0} \). We say \( u \) and \( V \) are
  covariant if they satisfy the covariance
  condition~\eqref{eq:ff84fdb6222c5b10}, or equivalently
  \( V(r_{0}) \in \morph_{\mathbb{G}}(r_{0} \cdot u, u) \) for all
  \( r_{0} \in \Lambda_{0} \).
\end{defi}

\begin{defi}
  \label{defi:9c6a41a5587c9da0}
  Let \( x \in \irr(\mathbb{G}) \), \( \Lambda_{0} \in \giso(\Lambda) \) with
  \( \Lambda_{0} \subseteq \Lambda_{x} \), \( u \in x \),
  \( \omega_{0} \in c_{x, \Lambda_{0}} \) (see
  Definition~\ref{defi:7ea8675d7dbf16c8}), then a unitary projective
  representation \( V \) of \( \Lambda_{0} \) that is covariant with \( u \) is
  said to be a covariant projective \( \Lambda_{0} \)-representation of \( u \)
  (with cocycle \( \omega_{0} \)).
\end{defi}
\begin{rema}
  \label{rema:984d0f723a1e4fe1}
  In the setting of Definition~\ref{defi:9c6a41a5587c9da0}, fix any covariant
  projective \( \Lambda_{0} \)-representation \( V \) of \( u \) with cocycle
  \( \omega_{0} \), the set of covariant projective
  \( \Lambda_{0} \)-representations of \( u \) with multiplier \( \omega_{0} \)
  is in bijective correspondence with the group of characters of
  \( \Lambda_{0} \), via \( \mathsf{b} \mapsto \mathsf{b}V \) (see
  Lemma~\ref{lemm:cc4e3d9cf9663b2b} and Lemma~\ref{lemm:62fa40c770de34e0}).
\end{rema}

\begin{prop}[Structure of stably pure CSR]
  \label{prop:f8d2ac97bef564cc}
  Fix a \( \Lambda_{0} \in \giso(\Lambda) \), let
  \( \mathbf{S} = (\mathscr{H}, u, w) \) be a stably pure CSR in
  \( \mathcal{CSR}_{\Lambda_{0}} \) \( x \in \irr(\mathbb{G}) \) is the support
  point of \( u \), \( u_{0} \in x \) a representation on some finite
  dimensional Hilbert space \( \mathscr{H}_{0} \), \( n \) is the multiplicity
  of \( u_{0} \) in \( u \), \( V_{0} \) a covariant projective
  \( \Lambda_{0} \)-representation of \( u \), then there exists a unique
  unitary projective representation
  \( v_{0} \colon \Lambda_{0} \rightarrow \mathcal{U}(\mathbb{C}^{n}) \) of
  \( \Lambda_{0} \) on \( \mathbb{C}^{n} \), such that the following hold:
  \begin{enumerate}
  \item \label{item:ecec63f36b0b94c1} \( V_{0} \) and \( v_{0} \) have opposing
    cocycles;
  \item \label{item:3bdd1d604068b729}
    \( \mathbf{S}_{0} = (\mathbb{C}^{n} \otimes \mathscr{H}_{0}, \epsilon_{n}
    \times u_{0}, v_{0} \times V_{0}) \) is a CSR in
    \( \mathcal{CSR}_{\Lambda_{0}} \), where \( \epsilon_{n} \) is the trivial
    representation of \( \mathbb{G} \) on \( \mathbb{C}^{n} \);
  \item \label{item:d5285c841ec062ca} \( \mathbf{S}_{0} \) and \( \mathbf{S} \)
    are isomorphic in \( \mathcal{CSR}_{\Lambda_{0}} \).
  \end{enumerate}
\end{prop}
\begin{proof}
  Uniqueness is almost clear once we finish the proof of existence, which we do
  now. Let \( U \) be a unitary intertwiner from \( u \) to
  \( \epsilon_{n} \otimes u_{0} \). Noting that \( u \) is pure and replacing
  \( \mathbf{S} \) with \( U \mathbf{S} U^{\ast} \) if necessary, we may assume
  \( \mathscr{H} = \mathbb{C}^{n} \otimes \mathscr{H}_{0} \) and
  \( u = \epsilon_{n} \times u_{0} = {(u_{0})}_{23} \). For any
  \( r_{0} \in \Lambda_{0} \), we claim that there exists a unique
  \( v_{0}(r_{0}) \in \mathcal{B}(\mathbb{C}^{n}) \) such that
  \( w(r_{0}) = v_{0}(r_{0}) \otimes V_{0}(r_{0}) \). Admitting the claim for
  the moment, the unitarity of \( v_{0}(r_{0}) \) follows from the unitarity of
  \( w(r_{0}) \) and \( V_{0}(r_{0}) \), and \( w \) being a representation and
  \( V_{0} \) being a projective representation force \( v_{0} \) to be a
  unitary projective representation with a cocycle opposing to the cocycle of
  \( V_{0} \). Thus the proposition follows from the claim, which we now
  prove. Since
  \( \mathcal{B}(\mathbb{C}^{n} \otimes \mathscr{H}_{0}) =
  \mathcal{B}(\mathbb{C}^{n}) \otimes \mathcal{B}(\mathscr{H}_{0}) \) by the
  usual identification, there exists an \( m \in \mathbb{N} \),
  \( A_{1}, \ldots, A_{m} \in \mathcal{B}(\mathbb{C}^{n}) \) and
  \( B_{1}, \ldots, B_{m} \in \mathcal{B}(\mathscr{H}_{0}) \), such that
  \begin{equation}
    \label{eq:e398eec963eb6512}
    w(r_{0}) = \sum_{i=1}^{m} A_{i} \otimes B_{i}.
  \end{equation}
  Furthermore, we can and do choose these operators so that
  \( A_{1}, \ldots, A_{m} \) are linearly independent in
  \( \mathcal{B}(\mathbb{C}^{n}) \). Since \( u \) and \( w \) are covariant, we
  have
  \begin{equation}
    \label{eq:7f93909bf92d4386}
    \bigl(w(r_{0}) \otimes 1\bigr) u %
    = \bigl[(\id_{\mathscr{H}} \otimes \alpha^{\ast}_{r_{0}})u\bigr] %
    \bigl(w(r_{0}) \otimes 1\bigr).
  \end{equation}
  Substituting \( u = {(u_{0})}_{23} \) and~\eqref{eq:e398eec963eb6512} in
  \eqref{eq:7f93909bf92d4386} yields
  \begin{equation}
    \label{eq:038416680331ec8a}
    \begin{split}
      & \leadmathskip
      \sum_{i=1}^{m} A_{i} \otimes [(B_{i} \otimes 1) u_{0}] \\
      &= \sum_{i=1}^{m} A_{i} \otimes \Bigl(\bigl[%
      (\id_{\mathscr{H}_{0}} \otimes \alpha^{\ast}_{r_{0}})u_{0}\bigr]%
      \bigl(B_{i} \otimes 1\bigr)\Bigr) %
      \in \mathcal{B}(\mathbb{C}^{n}) \otimes \mathcal{B}(\mathscr{H}_{0})
      \otimes \pol(\mathbb{G}).
    \end{split}
  \end{equation}
  Since \( A_{1}, \ldots, A_{m} \) are linearly independent, there exists linear
  functionals \( l_{1}, \ldots, l_{m} \) on \( \mathcal{B}(\mathbb{C}^{n}) \)
  such that \( l_{i}(A_{j}) = \delta_{i,j} \). Applying
  \( l_{i} \otimes \id_{\mathscr{H}_{0}} \otimes \id_{\pol(\mathbb{G})} \) on
  \eqref{eq:038416680331ec8a} shows that for each \( i = 1, \ldots, m \),
  \begin{equation}
    \label{eq:df4eaa8572d55883}
    (B_{i} \otimes 1) u_{0}
    = \bigl[(\id \otimes \alpha^{\ast}_{r_{0}})u_{0}\bigr] %
    (B_{i} \otimes 1),
  \end{equation}
  or equivalently
  \begin{equation}
    \label{eq:152e3338917fa440}
    B_{i} \in \morph_{\mathbb{G}}(r_{0} \cdot u_{0}, u_{0})
    = \mathbb{C} V_{0}(r_{0}).
  \end{equation}
  Now the claim follows from \eqref{eq:e398eec963eb6512}
  and~\eqref{eq:152e3338917fa440}.
\end{proof}

Conversely, we have
\begin{prop}
  \label{prop:1911ec535b058627}
  Fix a \( \Lambda_{0} \in \giso(\Lambda) \), \( x \in \irr(\mathbb{G}) \) with
  \( \Lambda_{0} \subseteq \Lambda_{x} \). Take a \( u \in x \) acting on some
  finite dimensional Hilbert space \( \mathscr{H} \), and a covariant projective
  \( \Lambda_{0} \)-representation \( V \) of \( u \), then for any unitary
  projective representation
  \( v \colon \Lambda_{0} \rightarrow \mathcal{U}(\mathscr{K}) \) of
  \( \Lambda_{0} \) with cocycle opposing the cocycle of \( V \), the unitary
  representation \( v \times V \) of \( \Lambda_{0} \) is covariant with the
  unitary representation
  \( \id_{\mathscr{K}} \otimes u = \epsilon_{\mathscr{K}} \times u \) of
  \( \mathbb{G} \), where \( \epsilon_{\mathscr{K}} \) is the trivial
  representation of \( \mathbb{G} \) on \( \mathscr{K} \), i.e.\
  \( (\mathscr{K} \otimes \mathscr{H}, \epsilon_{\mathscr{K}} \times u, v \times
  V) \) is a stably pure CSR in \( \mathcal{CSR}_{\Lambda_{0}} \).
\end{prop}
\begin{proof}
  Since \( V \) and \( u \) are covariant, for any \( r_{0} \in \Lambda_{0} \),
  we have
  \begin{equation}
    \label{eq:b526cec2bb6ba021}
    \bigl(V(r_{0}) \otimes 1\bigr) u %
    = \bigl[(\id \otimes \alpha^{\ast}_{r_{0}})u\bigr] %
    \bigl(V(r_{0}) \otimes 1\bigr).
  \end{equation}
  The proposition follows by tensoring \( v(r_{0}) \) on the left
  in~\eqref{eq:b526cec2bb6ba021}.
\end{proof}

By Proposition~\ref{prop:f8d2ac97bef564cc} and
Proposition~\ref{prop:1911ec535b058627}, we now have a satisfactory description
of stably pure CSRs in \( \mathcal{CSR}_{\Lambda_{0}} \)---from any irreducible
representation \( u \) of \( \mathbb{G} \) on \( \mathscr{H} \) such that
\( \Lambda_{0} \cdot [u] = [u] \), one choose a covariant projective
\( \Lambda_{0} \)-representation \( V \) of \( u \) with some cocycle
\( \omega \), then any unitary projective representation \( v \) of
\( \Lambda_{0} \) with cocycle \( \omega^{-1} = \overline{\omega} \) gives rise
to a stably pure CSR in \( \mathcal{CSR}_{\Lambda_{0}} \), namely
\( \mathcal{S}(u, V, v) = (\mathscr{K} \otimes \mathscr{H},
\epsilon_{\mathscr{K}} \times u, v \times V) \); and all stably pure CSRs in
\( \mathcal{CSR}_{\Lambda_{0}} \) arise in this way up to isomorphism.

\begin{rema}
  \label{rema:2afa01207639301d}
  Using the above notations, while it is true that \( V \) is determined by
  \( u \) to a great extent due to the restriction of Schur's lemma, it is still
  not completely determined (see Proposition~\ref{prop:80a4899eb993b414}), and a
  choice of this \( V \) is vitally relevant as is demonstrated by
  Remark~\ref{rema:c0671b6740777559} applied to \( v \). This is why \( V \) can
  \emph{not} be suppressed in our notation \( \mathcal{S}(u, V, v) \).
\end{rema}

\begin{defi}
  \label{defi:a5316bfdaeb22cfa}
  Let \( \Lambda_{0} \in \giso(\Lambda) \).  A triple \( (u, V, v) \) is called
  a \textbf{representation parameter} for \( \mathbb{G} \rtimes \Lambda \)
  associated with \( \Lambda_{0} \), if it the following hold:
  \begin{itemize}
  \item \( u \) is an irreducible unitary representation of \( \mathbb{G} \) on
    some finite dimensional Hilbert space \( \mathscr{H} \);
  \item \( V \) is a covariant projective
    \( \Lambda_{0} \)\nobreakdash-representation of \( u \);
  \item \( v \) is a unitary projective representation of \( \Lambda_{0} \)
    (possibly on Hilbert spaces other than \( \mathscr{H} \)), such that \( v \)
    and \( V \) have opposing cocycles.
  \end{itemize}

  If \( (u, V, v) \) is a representation parameter, the stably pure CSR
  \( \mathcal{S}(u, V, v) \) in \( \mathcal{CSR}_{\Lambda_{0}} \) is called the
  CSR parametrized by the representation parameter \( (u, V, v) \). If
  furthermore the unitary projective representation \( v \) is irreducible, we
  say the representation parameter \( (u, V, v) \) is \emph{irreducible}.
\end{defi}

Thus Proposition~\ref{prop:f8d2ac97bef564cc} immediately implies the following
corollary.
\begin{coro}
  \label{coro:f85e8a8060bb9d12}
  Fix a \( \Lambda_{0} \in \giso(\Lambda) \), then every stably pure CSR
  associated with \( \Lambda_{0} \) is parameterised by some representation
  parameter associated with \( \Lambda_{0} \). \qed
\end{coro}

\begin{defi}
  \label{defi:1b15f5eabe1196e0}
  Fix a \( \Lambda_{0} \in \giso(\Lambda) \). Let \( u \) be an irreducible
  unitary representation of \( \mathbb{G} \) such that
  \( \Lambda_{0} \cdot [u] = [u] \), \( V_{1} \) and \( V_{2} \) are two
  covariant projective \( \Lambda_{0} \)-representations of \( u \), the unique
  mapping \( \mathsf{b} \colon \Lambda_{0} \rightarrow \mathbb{T} \) such that
  \( V_{2} = \mathsf{b} V_{1} \) is called the \( u \)-\textbf{transitional
    mapping} from \( V_{1} \) to \( V_{2} \) (note that we do \emph{not} require
  \( V_{1} \) and \( V_{2} \) to have the same cocycle here).
\end{defi}

\begin{prop}
  \label{prop:7b2a964908c81d83}
  Fix a \( \Lambda_{0} \in \giso(\Lambda) \). For \( i = 1,2 \), let
  \( (u_{i}, V_{i}, v_{i}) \) be a representation parameter associated with
  \( \Lambda_{0} \), \( U_{i} \) denote the unitary representation
  \( \mathscr{R}_{\Lambda_{0}}\Bigl(\mathcal{S}(u_{i}, V_{i},v_{i})\Bigr) \) of
  \( \mathbb{G} \rtimes \Lambda_{0} \), then the following holds:
  \begin{enumerate}
  \item \label{item:cda18708f8635340} if \( [u_{1}] \ne [u_{2}] \) in
    \( \irr(\mathbb{G}) \), then
    \( \dim\morph_{\mathbb{G} \rtimes \Lambda_{0}}(U_{1}, U_{2}) = 0 \);
  \item \label{item:e80ad530289a518d} if \( u_{1} = u_{2} = u \), and
    \( \mathsf{b} \colon \Lambda_{0} \rightarrow \mathbb{T} \) the
    \( u \)-transitional map from \( V_{1} \) to \( V_{2} \), then
    \begin{equation}
      \label{eq:dbf0691d433c1cab}
      \dim\morph_{\mathbb{G} \rtimes \Lambda_{0}}(U_{1}, U_{2}) =
      \dim\morph_{\Lambda_{0}}\left(v_{1}, %
        \mathsf{b}v_{2}\right).
    \end{equation}
  \end{enumerate}
\end{prop}
\begin{proof}
  Let \( h \) be the Haar state of \( \mathbb{G} \),
  by~\eqref{eq:1c67d64dafdd1948}, the Haar state \( h_{\Lambda_{0}} \) of
  \( \mathbb{G} \rtimes \Lambda_{0} \) is the linear functional on
  \( A \otimes C(\Lambda_{0}) \) defined by
  \( a \otimes \delta_{r_{0}} \mapsto \abs*{\Lambda_{0}}^{-1} h(a) \), where
  \( a \in A \), \( r_{0} \in \Lambda_{0} \) (recall that
  \( A = C(\mathbb{G}) \)).

  Suppose \( [u_{1}] \neq [u_{2}] \).  For any \( i = 1, 2 \), by choosing a
  Hilbert space basis for the representation of \( u_{i} \), one can write
  \( u_{i} \) as a square matrix \( \left(u^{(i)}_{jk}\right) \) over
  \( \pol(\mathbb{G}) \subseteq A \), and \( V_{i} \) as a matrix
  \( \left(V^{(i)}_{jk}\right) \) over \( C(\Lambda_{0}) \) of the same size of
  \( \left(u^{(i)}_{jk}\right) \). Then the character \( \chi_{i} \) of
  \( U_{i} \) is given by
  \begin{equation}
    \label{eq:2508a507c6a14a5c}
    \chi_{i} = \sum_{r_{0} \in \Lambda_{0}} \sum_{j = 1}^{n_{i}} \tr(v_{i})
    \left(\sum_{k=1}^{n_{i}} V^{(i)}_{kj}(r_{0})u^{(i)}_{jk} \right) \otimes
    \delta_{r_{0}} \in \pol(\mathbb{G}) \otimes C(\Lambda_{0}).
  \end{equation}
  The orthogonality relation for the nonequivalent irreducible representations
  \( u_{1} \) and \( u_{2} \) implies that
  \begin{equation}
    \label{eq:579f67898ee5b2ec}
    \forall j_{1}, k_{1}, j_{2}, k_{2}, \quad %
    h\left({\bigl(u^{(1)}_{j_{1}k_{1}}\bigr)}^{\ast}
      u^{(2)}_{j_{2}k_{2}}\right) = 0.
  \end{equation}
  Hence, by \eqref{eq:2508a507c6a14a5c} and \eqref{eq:579f67898ee5b2ec},
  \begin{equation}
    \label{eq:cfc9ba1665474da9}
    \begin{split}
      & \leadmathskip \dim \morph_{\mathbb{G} \rtimes \Lambda_{0}}(U_{1},
      U_{2}) %
      = h_{\Lambda_{0}}\bigl(\overline{\chi_{1}} \chi_{2}\bigr) \\
      &= \abs*{\Lambda_{0}}^{-1} \sum_{r_{0} \in \Lambda_{0}} %
      \sum_{j_{1}, k_{1} =1}^{n_{1}} \sum_{j_{2}, k_{2} = 1}^{n_{2}} %
      \overline{\tr\bigl(v_{1}(r_{0})\bigr)} \tr\bigl(v_{2}(r_{0})\bigr) \\
      & \dmslskip \overline{V^{(1)}_{k_{1}j_{1}}(r_{0})}
      V^{(2)}_{k_{2}j_{2}}(r_{0}) %
      h\left({\bigl(u^{(1)}_{j_{1}k_{1}}\bigr)}^{\ast}
        u^{(2)}_{j_{2}k_{2}}\right) \\
      &= 0.
    \end{split}
  \end{equation}
  This proves \ref{item:cda18708f8635340}.

  Under the hypothesis of \ref{item:e80ad530289a518d}, using the same notations
  as in the previous paragraph, we have \( n_{1} = n_{2} = \dim U \).  We may
  assume that \( e^{(1)}_{j} = e^{(2)}_{j} = e_{j} \), hence
  \( u_{jk} := u^{(1)}_{jk} = u^{(2)}_{jk} \) for all possible \( j, k \). Note
  that \( V_{2} = \mathsf{b} V_{1} \), and
  \( \mathcal{S}(u, V_{2}, v_{2}) = \mathcal{S}(u, V_{1}, \mathsf{b}v_{2}) \)
  because \( \mathsf{b}V_{1} \times v_{2} = V_{1} \times \mathsf{b} v_{2} \), we
  may assume that \( V_{2} = V_{1} = V \) and \( \mathsf{b} = 1 \), with
  \( V_{jk} := V^{(1)}_{jk} = V^{(2)}_{jk} \in C(\Lambda_{0}) \) for all
  possible \( j, k \). Let \( \rho \) be the unique invertible positive operator
  in \( \morph_{\mathbb{G}}(u, u^{cc}) \) such that
  \( \tr( \cdot \, \rho) = \tr( \cdot \, \rho^{-1}) \) on
  \( \selfmorph_{\mathbb{G}}(u) \). With these assumptions,
  by~\eqref{eq:cfc9ba1665474da9}, the orthogonality relation takes the form
  \begin{equation}
    \label{eq:8f4ad60167971100}
    h(u_{ij}^{\ast} u_{kl}) = \frac{\delta_{j,l}
      {\bigl(\rho^{-1}\bigr)}_{ki}}{\dim_{q}U}
  \end{equation}
  where \( \dim_{q}U = \tr(\rho) = \tr(\rho^{-1}) \) is the quantum dimension of
  \( U \) (see~\cite[\S 1.4]{MR3204665}). Since \( \rho \) is positive, we might
  choose the basis \( e_{1}, \ldots, e_{n} \) to diagonize \( \rho \), so that
  \( \rho_{ki} = {\bigl(\rho^{-1}\bigr)}_{ki} = 0 \) whenever \( k \neq i
  \). Using this basis, \eqref{eq:8f4ad60167971100} and
  \eqref{eq:cfc9ba1665474da9}, we have
  \begin{equation}
    \label{eq:b5147c8b6d5d415c}
    \begin{split}
      & \leadmathskip
      \dim \morph_{\mathbb{G} \rtimes \Lambda_{0}}(U_{1}, U_{2}) \\
      &= \abs*{\Lambda_{0}}^{-1} \sum_{r_{0} \in \Lambda_{0}} %
      \sum_{j_{1}, k_{1} =1}^{n} \sum_{j_{2}, k_{2} = 1}^{n} %
      \overline{\tr\bigl(v_{1}(r_{0})\bigr)} \tr\bigl(v_{2}(r_{0})\bigr) \\
      & \dmsllongskip \cdot \overline{V_{k_{1}j_{1}}(r_{0})}
      V_{k_{2}j_{2}}(r_{0}) %
      h\left({\bigl(u_{j_{1}k_{1}}\bigr)}^{\ast} u_{j_{2}k_{2}}\right) \\
      &= \abs*{\Lambda_{0}}^{-1} \sum_{r_{0} \in \Lambda_{0}} %
      \sum_{j_{1}, k_{1} =1}^{n} \sum_{j_{2}, k_{2} = 1}^{n}
      \overline{\tr\bigl(v_{1}(r_{0})\bigr)} \tr\bigl(v_{2}(r_{0})\bigr) %
      \overline{V_{k_{1}j_{1}}(r_{0})} V_{k_{2}j_{2}}(r_{0})
      \\
      & \dmsllongskip %
      \cdot \frac{\delta_{j_{1},j_{2}} \delta_{k_{1}, k_{2}}
        {\bigl(\rho^{-1}\bigr)}_{j_{2}j_{1}}}{\dim_{q}U}
      \\
      &= \abs*{\Lambda_{0}}^{-1} \overline{\tr\bigl(v_{1}(r_{0})\bigr)}
      \tr\bigl(v_{2}(r_{0})\bigr) %
      \sum_{r_{0} \in \Lambda_{0}} \sum_{j = 1}^{n} \left\{\sum_{k = 1}^{n}
        \overline{V_{kj}(r_{0})} V_{kj}(r_{0})\right\}
      \frac{{\bigl(\rho^{-1}\bigr)}_{jj}}{\dim_{q}U} \\
      & \leadmathskip (\text{Note that \( V(r_{0}) \) is unitary}) \\
      &= \abs*{\Lambda_{0}}^{-1} \sum_{r_{0} \in \Lambda_{0}}
      \overline{\tr\bigl(v_{1}(r_{0})\bigr)} %
      \tr\bigl(v_{2}(r_{0})\bigr) %
      \frac{\sum_{j=1}^{n} {\bigl(\rho^{-1}\bigr)}_{jj}}{\dim_{q} U} \\
      &= \abs*{\Lambda_{0}}^{-1} \sum_{r_{0} \in \Lambda_{0}}
      \overline{\tr\bigl(v_{1}(r_{0})\bigr)} %
      \tr\bigl(v_{2}(r_{0})\bigr) = \dim\morph_{\Lambda_{0}}(v_{1}, v_{2}).
    \end{split}
  \end{equation}
  This proves \ref{item:e80ad530289a518d}.
\end{proof}

The following corollary is now clear.
\begin{coro}
  \label{coro:d09f1fd81ad53d88}
  Fix a \( \Lambda_{0} \in \giso(\Lambda) \). Let \( (u, V, v) \) be a
  representation parameter associated with \( \Lambda_{0} \), then the
  representation \( \mathscr{R}_{\Lambda_{0}}\Bigl(\mathcal{S}(u, V, v)\Bigr) \)
  of \( \mathbb{G} \rtimes \Lambda_{0} \) is irreducible if and only if the
  representation parameter \( (u, V, v) \) is irreducible. \qed
\end{coro}

\section{Distinguished representation parameters and distinguished
  representations}
\label{sec:9a5ddf920bd49589}

Fix a \( \Lambda_{0} \in \giso(\Lambda) \). For any unitary projective
representation \( V \colon \Lambda_{0} \rightarrow \mathcal{U}(\mathscr{H}) \)
of \( \Lambda_{0} \), and any \( r \in \Lambda \), define \( r \cdot V \) to be
the unitary projective representation of \( r\Lambda_{0}r^{-1} \) on
\( \mathscr{H} \) sending \( s_{0} = rr_{0}r^{-1} \in r\Lambda_{0}r^{-1} \) to
\( (V \circ \adj_{r^{-1}})(s_{0}) = V(r_{0}) \). Then
\( (rs) \cdot V = r \cdot (s \cdot V) \) for all \( r, s \in \Lambda \) with
\( 1_{\Lambda} \cdot V = V \), in other words, this defines an action of the
group \( \Lambda \) on the class of all unitary projective representations of
general isotropy subgroups of \( \Lambda \).

It is easy to see from Proposition~\ref{prop:e4e616da0a0995e6} that whenever
\( \mathbf{S} = (\mathscr{H}, u, w) \in \mathcal{CSR}_{\Lambda_{0}} \), the
triple \( r \cdot \mathbf{S} = (\mathscr{H}, r \cdot u, r \cdot w) \) is a CSR
in \( \mathcal{CSR}_{r\Lambda_{0}r^{-1}} \). If
\( U = \mathscr{R}_{\Lambda_{0}}(\mathbf{S}) \) is the unitary representation of
\( \mathbb{G} \rtimes \Lambda_{0} \), then it is easy to see by restriction that
\( \mathscr{R}_{r\Lambda_{0}r^{-1}}(r \cdot \mathbf{S}) \) is the unitary
representation
\( r \cdot U = (\id \otimes \alpha^{\ast}_{r^{-1}} \otimes
\adj^{\ast}_{r^{-1}})(U) \) of \( \mathbb{G} \rtimes r\Lambda_{0}r^{-1} \), as
described in Proposition~\ref{prop:56963e422c66f26b}. Thus by
Corollary~\ref{coro:a2b4b06e5632471c}, we see that \( \indrep(U) \) and
\( \indrep(r \cdot U) \) are equivalent representations of
\( \mathbb{G} \rtimes \Lambda \).

Similarly, for any representation parameter \( (u, V, v) \) associated with
\( \Lambda_{0} \) and any \( r \in \Lambda \), the triple
\( (r \cdot u, r \cdot V, r \cdot v) \) is a representation parameter associated
with \( r \Lambda_{0} r^{-1} \), which we denoted by \( r \cdot (u, V, v) \).
This clearly defines an \( \Lambda \)-action on the proper class of all
representation parameters associated with any group in some conjugacy class of a
general isotropy subgroup of \( \Lambda \). A simple calculation shows that
(recall \( \mathcal{S}(u, V, v) \) is the CSR parameterized by \( (u, V, v) \))
\begin{equation}
  \label{eq:e2a420685bba7c4f}
  \forall r \in \Lambda, \quad r \cdot \mathcal{S}(u, V, v)
  = \mathcal{S}\bigl(r \cdot (u, V, v)\bigr).
\end{equation}

\begin{defi}
  \label{defi:ab1995a9dda3749d}
  Let \( (u, V, v) \) be a representation parameter associated with some
  \( \Lambda_{0} \in \giso(\Lambda) \), the induced representation
  \( \indrep\Bigl(\mathscr{R}_{\Lambda_{0}}\bigl(\mathcal{S}(u, V,
  v)\bigr)\Bigr) \) of \( \mathbb{G} \rtimes \Lambda \) is called the
  representation of \( \mathbb{G} \rtimes \Lambda \) parameterized by
  \( (u, V, v) \).
\end{defi}

\begin{prop}
  \label{prop:49d184eecd55653e}
  Let \( (u, V, v) \) be a representation parameter associated with some
  \( \Lambda_{0} \in \giso(\Lambda) \). Then for any \( r \in \Lambda \), the
  representation parameters \( (u, V, v) \) and \( r \cdot (u, V, v) \)
  parameterize equivalent representations of \( \mathbb{G} \rtimes \Lambda \).
\end{prop}
\begin{proof}
  Since \( \mathscr{R}_{\Lambda_{0}}\bigl(\mathcal{S}(u, V, v)\bigr) \) and
  \( \mathscr{R}_{r\Lambda_{0}r^{-1}}\bigl(r \cdot \mathcal{S}(u, V, v)\bigr) \)
  induces equivalent representations of \( \mathbb{G} \rtimes \Lambda \), the
  proposition now follows from equation~\eqref{eq:e2a420685bba7c4f} and
  Definition~\ref{defi:ab1995a9dda3749d}.
\end{proof}

\begin{prop}
  \label{prop:7cb6d8671da22396}
  Fix a \( \Lambda_{0} \in \giso(\Lambda) \). Let \( (u, V, v) \) be an
  irreducible representation parameter associated with \( \Lambda_{0} \),
  \( U \) denote the representation
  \( \mathscr{R}_{\Lambda_{0}}\Bigl(\mathcal{S}(u, V, v)\Bigr) \). If
  \( \Lambda_{0} = \Lambda_{[u]} \), then the the induced representation
  \( \indrep(U) \) of \( \mathbb{G} \rtimes \Lambda \) is irreducible.
\end{prop}
\begin{proof}
  By Corollary~\ref{coro:a250dbc9bc328b12}, the proposition amounts to show that
  \begin{equation}
    \label{eq:3b9327fe982589d1}
    \begin{split}
      & \forall r, s \in \Lambda, \\
      & r^{-1} s \notin \Lambda_{0} \implies \dim \morph_{\mathbb{G} \rtimes
        \Lambda(r, s)}\Bigl((r \cdot U) \vert_{\mathbb{G} \rtimes \Lambda(r,s)},
      (s \cdot U) \vert_{\mathbb{G} \rtimes \Lambda(r,s)}\Bigr) = 0,
    \end{split}
  \end{equation}
  where \( \Lambda(r, s) = r\Lambda_{0}r^{-1} \cap s\Lambda_{0}s^{-1} \). Since
  \( \Lambda_{0} = \Lambda_{[u]} \), by the definition of \( \Lambda_{[u]} \),
  we have \( [r \cdot u] \neq [s \cdot u] \) whenever
  \( r^{-1}s \notin \Lambda_{0} \). Now condition \eqref{eq:3b9327fe982589d1}
  holds by Proposition~\ref{prop:7b2a964908c81d83}.
\end{proof}

\begin{defi}
  \label{defi:cbe3a197ae3ac6c1}
  Fix a \( \Lambda_{0} \in \giso(\Lambda) \), an \emph{irreducible}
  representation parameter \( (u, V, v) \) associated with \( \Lambda_{0} \) is
  called \textbf{distinguished} if \( \Lambda_{0} = \Lambda_{[u]} \). When this
  is the case, the irreducible unitary representation \( \indrep(U) \) of
  \( \mathbb{G} \rtimes \Lambda \) is called \textbf{distinguished}, where
  \( U \) is the unitary representation
  \( \mathscr{R}_{\Lambda_{0}}\Bigl(\mathcal{S}(u, V, v)\Bigr) \) of
  \( \mathbb{G} \rtimes \Lambda_{0} \).
\end{defi}

\begin{rema}
  \label{rema:d4178ef14b0e8a24}
  The associated group of a distinguished representation parameter must be an
  isotropy subgroup of \( \Lambda \) for the action
  \( \Lambda \curvearrowright \irr(\mathbb{G}) \). More precisely, a
  representation parameter \( (u, V, v) \) is distinguished if and only if its
  associated group is exactly the isotropy subgroup of
  \( [u] \in \irr(\mathbb{G}) \) under the action
  \( \Lambda \curvearrowright \irr(\mathbb{G}) \). As we will see presently, in
  the formulation of our results on the classification of irreducible
  representations of \( \mathbb{G} \rtimes \Lambda \) and the conjugation on
  \( \irr(\mathbb{G}) \), only distinguished representation parameters are
  needed. This makes one wonder why we pose the family of general isotropy
  subgroup \( \giso(\Lambda) \) instead of only isotropy subgroups. The main
  reason we need general isotropy subgroups of \( \Lambda \) is that in proving
  these results, as well as the formulation and the proof of the fusion rules,
  we need to express the dimensions of various intertwiner spaces. The
  calculation of the dimensions of these intertwiner spaces will rely on
  Proposition~\ref{prop:a8bf0fe189bec421}, which clearly requires us to consider
  the intersections of isotropy subgroups, i.e.\ general isotropy subgroups.
\end{rema}

\begin{defi}
  \label{defi:490c0876fe8d24d1}
  Let \( \Lambda_{0} \) be an isotropy subgroup of \( \Lambda \) for the action
  \( \Lambda \curvearrowright \irr(\mathbb{G}) \). Suppose
  \( (u_{1}, V_{1}, v_{1}) \) and \( (u_{2}, V_{2}, v_{2}) \) are two
  distinguished representation parameters associated with \( \Lambda_{0} \). If
  the CSRs \( \mathcal{S}(u_{1}, V_{1}, v_{1}) \) and
  \( \mathcal{S}(u_{2}, V_{2}, v_{2}) \) are isomorphic in
  \( \mathcal{CSR}_{\Lambda_{0}} \), we say \( (u_{1}, V_{1}, v_{1}) \) and
  \( (u_{2}, V_{2}, v_{2}) \) are equivalent.
\end{defi}

The following proposition serves to characterize equivalence of distinguished
representation parameters in some more concrete ways.
\begin{prop}
  \label{prop:01cb8f0781bcbbf7}
  Let \( \Lambda_{0} \) be an isotropy subgroup of \( \Lambda \) for the action
  \( \Lambda \curvearrowright \irr(\mathbb{G}) \), \( (u_{1}, V_{1}, v_{1}) \)
  and \( (u_{2}, V_{2}, v_{2}) \) two distinguished representation parameters
  associated with \( \Lambda_{0} \). The following are equivalent:
  \begin{enumerate}
  \item \label{item:b269e40f7b8257dd} \( (u_{1}, V_{1}, v_{1}) \) and
    \( (u_{2}, V_{2}, v_{2}) \) are equivalent;
  \item \label{item:c9d3eb41af7f039a} \( (u_{1}, V_{1}, v_{1}) \) and
    \( (u_{2}, V_{2}, v_{2}) \) parameterize equivalent representations of
    \( \mathbb{G} \rtimes \Lambda_{0} \);
  \item \label{item:a6f5dee8ea3ee9bd} there exists a mapping
    \( \mathsf{b} \colon \Lambda_{0} \rightarrow \mathbb{T} \) such that
    \( \mathsf{b}V_{1} \) and \( V_{2} \) share the same cocycle, and both
    \( \morph_{\mathbb{G}}(u_{1}, u_{2}) \cap
    \morph_{\Lambda_{0}}(\mathsf{b}V_{1}, V_{2}) \) and
    \( \morph_{\Lambda_{0}}(v_{1}, \mathsf{b}v_{2}) \) are nonzero;
  \item \label{item:795a54c3334933d3} there exists a mapping
    \( \mathsf{b} \colon \Lambda_{0} \rightarrow \mathbb{T} \) such that
    \( \mathsf{b}V_{1} \) and \( V_{2} \) share the same cocycle, and both
    \( \morph_{\mathbb{G}}(u_{1}, u_{2}) \cap
    \morph_{\Lambda_{0}}(\mathsf{b}V_{1}, V_{2}) \) and
    \( \morph_{\Lambda_{0}}(v_{1}, \mathsf{b}v_{2}) \) contain unitary
    operators.
  \end{enumerate}
\end{prop}
\begin{proof}
  The equivalence of \ref{item:b269e40f7b8257dd} and \ref{item:c9d3eb41af7f039a}
  follows directly from the definitions. It is also clear that
  \ref{item:795a54c3334933d3} implies \ref{item:a6f5dee8ea3ee9bd}. If
  \ref{item:a6f5dee8ea3ee9bd} holds, and
  \begin{equation}
    \label{eq:49f40c49afcaaa49}
    \begin{gathered}
      0 \neq S \in \morph_{\mathbb{G}}(u_{1}, u_{2}) \cap
      \morph_{\Lambda_{0}}(\mathsf{b}V_{1}, V_{2}), \\
      \text{and} \quad 0 \neq T \in \morph_{\Lambda_{0}}(v_{1}, \mathsf{b}v_{2})
      = \morph_{\Lambda_{0}}(\mathsf{b}^{-1}v_{1}, v_{2}),
    \end{gathered}
  \end{equation}
  then both \( S \) and \( T \) are invertible by Schur's lemma as
  \( u_{1}, u_{2}, \mathsf{b}^{-1}v_{1}, v_{2} \) are all irreducible. Since
  \( u_{1}, u_{2}, \mathsf{b}V_{1}, V_{2}, v_{1}, \mathsf{b}v_{2} \) are all
  unitary, we have
  \begin{equation}
    \label{eq:996d23f427faa56b}
    0 \neq \Upsilon_{S} \in  \morph_{\mathbb{G}}(u_{1}, u_{2})
    \cap \morph_{\Lambda_{0}}(\mathsf{b}V_{1},
    V_{2}), \quad
    \text{and} \quad 0 \neq \Upsilon_{T} \in \morph_{\Lambda_{0}}(v_{1},
    \mathsf{b}v_{2}),
  \end{equation}
  where \( S = \Upsilon_{S} \abs*{S} \) is the polar decomposition of \( S \),
  and \( T = \Upsilon_{T} \abs*{T} \) the polar decomposition of \( T \). As
  \( S, T \) are invertible, \( \Upsilon_{S} \) and \( \Upsilon_{T} \) are
  unitary. This proves that \ref{item:a6f5dee8ea3ee9bd} implies
  \ref{item:795a54c3334933d3}.

  Let \( \mathscr{K}_{i} \) be the representation space of \( v_{i} \) for
  \( i = 1, 2 \). By definition,
  \( \mathcal{S}(u_{i}, V_{i}, v_{i}) = (\id_{\mathscr{K}_{i}} \otimes u, v_{i}
  \times V_{i}) \), and
  \( \mathsf{b}^{-1}v_{i} \times \mathsf{b}V_{i} = v_{i} \times V_{i} \) for any
  mapping \( \mathsf{b} \colon \Lambda_{0} \rightarrow \mathbb{T} \). If
  \ref{item:a6f5dee8ea3ee9bd} holds, let \( S \), \( T \) be operators as
  in~\eqref{eq:49f40c49afcaaa49}, then
  \begin{equation}
    \label{eq:20ff567e8b68f09f}
    T \otimes S \in \morph_{\mathbb{G}}(\id_{\mathscr{K}_{1}} \otimes u_{1},
    \id_{\mathscr{K}_{2}} \otimes u_{2})
    \cap \morph_{\Lambda_{0}}(v_{1} \times V_{1},
    v_{2} \times V_{2}).
  \end{equation}
  Now \ref{item:b269e40f7b8257dd} follows from \eqref{eq:20ff567e8b68f09f},
  Proposition~\ref{prop:50b58e1fd8af4aa0} and the fact that both \( S \) and
  \( T \) are invertible. Thus \ref{item:a6f5dee8ea3ee9bd} implies
  \ref{item:b269e40f7b8257dd}.

  We conclude the proof by showing \ref{item:b269e40f7b8257dd} implies
  \ref{item:795a54c3334933d3}. By Schur's lemma, and the irreducibility of
  \( u_{1} \) and \( u_{2} \), it is easy to see that
  \begin{equation}
    \label{eq:d512bb297d144b3d}
    \morph_{\mathbb{G}}(\id_{\mathscr{K}_{1}} \otimes u_{1},
    \id_{\mathscr{K}_{2}} \otimes u_{2}) = \mathcal{B}(\mathscr{K}_{1},
    \mathscr{K}_{2}) \otimes \morph_{\mathbb{G}}(u_{1}, u_{2}).
  \end{equation}
  Suppose \ref{item:b269e40f7b8257dd} holds. Then the intertwiner space given by
  the intersection in \eqref{eq:20ff567e8b68f09f} is nonzero, and
  \begin{equation}
    \label{eq:f27cc1a76f9e813a}
    \morph_{\mathbb{G}}(u_{1}, u_{2}) = \mathbb{C} W_{r}
  \end{equation}
  for some unitary operator \( W_{r} \). By \eqref{eq:d512bb297d144b3d} and
  \ref{item:b269e40f7b8257dd}, there exists a unitary
  \( W_{l} \in \mathcal{B}(\mathscr{K}_{1}, \mathscr{K}_{2}) \) such that
  \begin{equation}
    \label{eq:5b7b37993908d266}
    W_{l} \otimes W_{r} \in \morph_{\Lambda_{0}}
    (v_{1} \times V_{1}, v_{2} \times V_{2})
    = \morph_{\Lambda_{0}}\bigl((\mathsf{b}^{-1}v_{1}) \times (\mathsf{b}V_{1}),
    v_{2} \times V_{2} \bigr).
  \end{equation}
  By~\eqref{eq:f27cc1a76f9e813a}, both \( W_{r}V_{1}W_{r}^{\ast} \) and
  \( V_{2} \) are covariant projective \( \Lambda_{0} \)-representations of
  \( u_{2} \). Thus we can take a \( u_{2} \)-transitional mapping
  \( \mathsf{b} \) from \( W_{r}V_{1}W_{r}^{\ast} \) to \( V_{2} \) (see
  Definition~\ref{defi:1b15f5eabe1196e0}), i.e.\ a mapping
  \( \mathsf{b} \colon \Lambda_{0} \rightarrow \mathbb{T} \) such that
  \begin{equation}
    \label{eq:fedff83db892c771}
    W_{r}(\mathsf{b}V_{1}) W_{r}^{\ast}
    =  \mathsf{b}(W_{r}V_{1}W_{r}^{\ast}) = V_{2} ,
  \end{equation}
  which forces the cocycles of \( \mathsf{b}V_{1} \) and \( V_{2} \) coincide,
  and
  \begin{equation}
    \label{eq:7603ec9cdc0d4f11}
    W_{r} \in \morph_{\Lambda_{0}}(\mathsf{b}V_{1}, V_{2}) \cap
    \morph_{\mathbb{G}}(u_{1}, u_{2}).
  \end{equation}
  Now \eqref{eq:5b7b37993908d266} and \eqref{eq:7603ec9cdc0d4f11} forces
  \begin{equation}
    \label{eq:b983406725fb2608}
    W_{l} \in \morph_{\Lambda_{0}}(\mathsf{b}^{-1}v_{1}, v_{2}) =
    \morph_{\Lambda_{0}}(v_{1}, \mathsf{b}v_{2}).
  \end{equation}
  Thus \ref{item:795a54c3334933d3} holds by \eqref{eq:7603ec9cdc0d4f11} and
  \eqref{eq:b983406725fb2608}.
\end{proof}

\section{Density of matrix coefficients of distinguished representations}
\label{sec:51cb2ca9f8d7976c}

The aim of this section is to show that the linear span of matrix coefficients
of distinguished representations of \( \mathbb{G} \rtimes \Lambda \) is exactly
\( \pol(\mathbb{G}) \otimes C(\Lambda) \), hence is dense in
\( C(\mathbb{G} \rtimes \Lambda) = A \otimes C(\Lambda) \) in particular. As a
consequence, any irreducible unitary representation of
\( \mathbb{G} \rtimes \Lambda \) is equivalent to a distinguished one.

The following lemma essentially establishes the density of the linear span of
matrix coefficients of distinguished representations of
\( \mathbb{G} \rtimes \Lambda \) in
\( C(\mathbb{G} \rtimes \Lambda) = A \otimes C(\Lambda) \).

\begin{lemm}
  \label{lemm:be8a473ead6cb802}
  Let \( u \) be an irreducible unitary representation of \( \mathbb{G} \) on
  some finite dimensional Hilbert space \( \mathscr{H} \),
  \( x = [u] \in \irr(\mathbb{G}) \), \( V \) the covariant projective
  \( \Lambda_{x} \)-representation of \( u \) with cocycle \( \omega \). Let
  \( M(u) \) denote the linear subspace of
  \( \pol(\mathbb{G}) \otimes C(\Lambda) \) spanned by matrix coefficients of
  distinguished representations of \( \mathbb{G} \rtimes \Lambda \)
  parameterized by distinguished representation parameters of the form
  \( (u, V, v) \), where \( v \) runs through all irreducible unitary projective
  representations of \( \Lambda_{x} \) with cocycle
  \( \omega^{-1} = \overline{\omega} \). For any \( r \in \Lambda \), suppose
  \( M_{c}(r \cdot u) \) is the linear subspace of \( \pol(\mathbb{G}) \)
  spanned by matrix coefficients of \( r \cdot u \), then
  \begin{equation}
    \label{eq:8d89a7b07bae9ec5}
    M(u) = \sum_{r \in \Lambda} M_{c}(r \cdot u) \otimes C(\Lambda)
    = \left(\sum_{r \in \Lambda} M_{c}(r \cdot
      u)\right) \otimes C(\Lambda).
  \end{equation}
\end{lemm}
\begin{proof}
  Take any irreducible unitary projective representation \( v \) of
  \( \Lambda_{x} \) on some finite dimensional Hilbert space \( \mathscr{K} \)
  with cocycle \( \overline{\omega} \), then \( (u, V, v) \) is a distinguished
  representation parameter. The distinguished CSR \( \mathcal{S}(u, V, v) \)
  subordinate to \( \Lambda_{x} \) parameterized by \( (u, V, v) \) is given by
  \begin{equation}
    \label{eq:6b6644aa7d2b0a12}
    \mathcal{S}(u, V, v) = (\mathscr{K} \otimes \mathscr{H},
    \id_{\mathscr{K}} \otimes u, v \times V)
  \end{equation}
  by definition. Let
  \( U = \mathscr{R}_{\Lambda_{0}}\bigl(\mathcal{S}(u, V, v)\bigr) \), then the
  distinguished representation \( W = \indrep(U) \) of
  \( \mathbb{G} \rtimes \Lambda \) parameterized by \( (u, V, v) \) is obtained
  as follows by the construction of induced representations presented in
  \S~\ref{sec:3e39eff392b139a7}. First we define a unitary representation
  \begin{equation}
    \label{eq:b1567fda3dfa4eeb}
    \begin{split}
      \widetilde{W} = \sum_{r, s \in \Lambda} e_{rs^{-1}, r} \otimes
      \id_{\mathscr{K}} \otimes [(\id_{\mathscr{H}} \otimes
      \alpha^{\ast}_{rs^{-1}})(u)] \otimes
      \delta_{s} \\
      \in \mathcal{B}(\ell^{2}(\Lambda)) \otimes \mathcal{B} (\mathscr{K})
      \otimes\mathcal{B}(\mathscr{H}) \otimes \pol(\mathbb{G}) \otimes
      C(\Lambda)
    \end{split}
  \end{equation}
  of \( \mathbb{G} \rtimes \Lambda \) on
  \( \ell^{2}(\Lambda) \otimes \mathscr{K} \otimes \mathscr{H} \). The subspace
  \begin{equation}
    \label{eq:e43c0735a5c07c84}
    \mathscr{H}_{(u, V, v)}
    = \set*{\sum_{r \in \Lambda} \delta_{r} \otimes \zeta_{r} \given
      \begin{array}{r}
        \zeta_{r} \in \mathscr{K}
        \otimes \mathscr{H}, \text{ and }\; \zeta_{r_{0}r} =
        \bigl(v(r_{0}) \otimes V(r_{0})\bigr) \zeta_{r} \\
        \text{ for all } r_{0} \in
        \Lambda_{0}, r \in \Lambda
      \end{array}
    }
  \end{equation}
  of \( \ell^{2}(\Lambda) \otimes \mathscr{K} \otimes \mathscr{H} \) is
  invariant under \( \widetilde{W} \) and \( W \) is the subrepresentation
  \( \mathscr{H}_{(u, V, v)} \) of \( \widetilde{W} \). Recall
  (Lemma~\ref{lemm:8636ddbbcbafe76a}) that the projection
  \( \pi \in \mathcal{B}(\ell^{2}(\Lambda) \otimes \mathscr{K} \otimes
  \mathscr{H}) \) with range \( \mathscr{H}_{u, V, v} \) is given by
  \begin{equation}
    \label{eq:4d915d223a16e971}
    \pi = \frac{1}{\abs*{\Lambda_{x}}}
    \sum_{r_{0} \in \Lambda_{x}} \sum_{s \in \Lambda} e_{r_{0}s,
      s} \otimes v(r_{0}) \otimes V(r_{0}).
  \end{equation}
  Since vectors of the form \( \delta_{r} \otimes \xi \otimes \eta \),
  \( r \in \Lambda \), \( \xi \in \mathscr{K} \), \( \eta \in \mathscr{H} \)
  span \( \ell^{2}(\Lambda) \otimes \mathscr{K} \otimes \mathscr{H} \), the
  matrix coefficients of \( W \) is spanned by elements of
  \( \pol(\mathbb{G}) \otimes C(\Lambda) \) of the form
  \begin{equation}
    \label{eq:de1bf61dd9881085}
    \begin{split}
      & \leadmathskip c(v; r, s, \xi_{1}, \xi_{2}, \eta_{1}, \eta_{2}) \\
      & = (\omega_{\pi(\delta_{r} \otimes \xi_{1} \otimes \eta_{1}),
        \pi(\delta_{s} \otimes \xi_{2} \otimes \eta_{2})} \otimes
      \id_{\pol(\mathbb{G})}
      \otimes \id_{C(\Lambda)})(W) \\
      &= (\omega_{\delta_{r} \otimes \xi_{1} \otimes \eta_{1}, \delta_{s}
        \otimes \xi_{2} \otimes \eta_{2}} \otimes \id \otimes \id)\bigl((\pi
      \otimes 1 \otimes 1)
      \widetilde{W} (\pi \otimes 1 \otimes 1)\bigr) \\
      &= (\omega_{\delta_{r} \otimes \xi_{1} \otimes \eta_{1}, \delta_{s}
        \otimes \xi_{2} \otimes \eta_{2}} \otimes \id \otimes
      \id)\bigl(\widetilde{W} (\pi \otimes 1 \otimes 1)\bigr),
    \end{split}
  \end{equation}
  where the last equality follows from Lemma~\ref{lemm:8636ddbbcbafe76a}, and
  \( \omega_{x,y} \) is the linear form \( \pairing*{\, \cdot \, x}{y} \).

  By~\eqref{eq:b1567fda3dfa4eeb} and \eqref{eq:4d915d223a16e971}, we see that
  \begin{equation}
    \label{eq:46f7291b743fe368}
    \begin{split}
      & \leadmathskip \abs*{\Lambda_{0}} \cdot \bigl[\widetilde{W}(\pi \otimes 1
      \otimes 1)\bigr](\delta_{r} \otimes
      \xi_{1} \otimes \eta_{1} \otimes 1 \otimes 1) \\
      &= \sum_{\substack{r', s', t \in \Lambda,\\ r_{0} \in \Lambda_{x}}}
      \left[e_{r's'^{-1}, r'} e_{r_{0}t, t} \otimes v(r_{0}) \otimes \Bigl(
        \bigl((s'r'^{-1}) \cdot u\bigr) \bigl(V(r_{0}) \otimes 1\bigr)\Bigr)
        \otimes \delta_{s'}\right] \\
      & \dmsllongskip
      \cdot (\delta_{r} \otimes \xi_{1} \otimes \eta_{1} \otimes 1 \otimes 1) \\
      & \leadmathskip (\text{Only terms with \( t = r \), and
        \( r' = r_{0}t = r_{0}r
        \) can be nonzero}) \\
      &= \sum_{s' \in \Lambda} \sum_{r_{0} \in \Lambda_{x}}
      \delta_{r_{0}rs'^{-1}} \otimes [v(r_{0}) \xi_{1}] \otimes
      \Bigl[\bigl((s'r^{-1}r_{0}^{-1}) \cdot u\bigr)\bigl(V(r_{0}) \eta_{1}
      \otimes 1\bigr)\Bigr] \otimes \delta_{s'}.
    \end{split}
  \end{equation}
  Note that
  \( r_{0}rs'^{-1} = s \iff s' = s^{-1}r_{0}r \iff s'r^{-1}r_{0}^{-1} = s^{-1}
  \), by \eqref{eq:de1bf61dd9881085} and \eqref{eq:46f7291b743fe368}, we have
  \begin{equation}
    \label{eq:0dcb4ec0bb375967}
    c(v; r, s, \xi_{1}, \xi_{2}, \eta_{1}, \eta_{2}) %
    = \sum_{r_{0} \in \Lambda_{x}}
    \omega_{\xi_{1}, \xi_{2}}\bigl(v(r_{0})\bigr) %
    \bigl[(\omega_{V(r_{0}) \eta_{1}, \eta_{2}}
    \otimes \id)(s^{-1} \cdot u)\bigr] %
    \otimes \delta_{s^{-1}r_{0}r}.
  \end{equation}
  For any \( r_{0} \in \Lambda_{x} \), we have
  \begin{equation}
    \label{eq:e784989c4d2a01ed}
    \omega_{\xi_{1}, \xi_{2}}\bigl(v(r_{0})\bigr)
    \in \mathbb{C} \quad \text{and} \quad
    \bigl[(\omega_{V(r_{0})\eta_{1}, \eta_{2}} \otimes \id)
    (s^{-1} \cdot u)\bigr] \in
    M_{c}(s^{-1} \cdot u).
  \end{equation}
  By \eqref{eq:0dcb4ec0bb375967} and \eqref{eq:e784989c4d2a01ed}, we have
  \begin{equation}
    \label{eq:124769da5db1b134}
    c(v; r, s, \xi_{1}, \xi_{2}, \eta_{1}, \eta_{2})
    \in M_{c}(s^{-1} \cdot u) \otimes C(\Lambda),
  \end{equation}
  which proves that
  \begin{equation}
    \label{eq:6e015a7afe36e34d}
    M(u) \subseteq \sum_{r' \in \Lambda} M_{c}(r' \cdot u) \otimes C(\Lambda)
    = \left(\sum_{r' \in \Lambda}M_{c}(r'
      \cdot u)\right) \otimes C(\Lambda).
  \end{equation}
  It remains to establish the reverse inclusion, which is easily seen to be
  equivalent to show that for any \( r_{1}, r_{2} \in \Lambda \), we have
  \begin{equation}
    \label{eq:1ff8e1270b4b575c}
    M(u) \supseteq M_{c}(r_{1} \cdot u) \otimes \delta_{r_{2}}.
  \end{equation}

  By the general theory of projective representations, there exists irreducible
  unitary projective representations \( v_{1}, \ldots, v_{m} \) on
  \( \mathscr{K}_{1}, \ldots, \mathscr{K}_{m} \) respectively, all with cocycle
  \( \overline{\omega} \), and
  \( \xi^{(i)}_{1}, \xi^{(i)}_{2} \in \mathscr{K}_{i} \), such that
  \begin{equation}
    \label{eq:632fd1580b88627d}
    \sum_{i=1}^{m} \left(\omega_{\xi^{(i)}_{1}, \xi^{(i)}_{2}} \otimes
      \id\right)(v_{i}) = \delta_{e} \in C(\Lambda_{x}).
  \end{equation}
  By \eqref{eq:0dcb4ec0bb375967} and \eqref{eq:632fd1580b88627d}, we see that
  for any \( r, s \in \Lambda \), and any
  \( \eta_{1}, \eta_{2} \in \mathscr{H} \), \( M(u) \) contains
  \begin{equation}
    \label{eq:8400ce42f4a3f560}
    \begin{split}
      & \leadmathskip \sum_{i=1}^{n} c(v_{i};r, s \xi^{(i)}_{1}, \xi^{(i)}_{2},
      \eta_{1}, \eta_{2}) \\
      &= \sum_{r_{0} \in \Lambda_{x}} \delta_{e}(r_{0}) \bigl[(\omega_{V(r_{0})
        \eta_{1}, \eta_{2}} \otimes \id)(s^{-1} \cdot u)\bigr] %
      \otimes \delta_{s^{-1}r_{0}r} \\
      & \leadmathskip (\text{Only terms with \( r_{0} = e \) can be nonzero, and
        \( V(e) = \id_{\mathscr{H}} \)}) \\
      &= \bigl[(\omega_{\eta_{1}, \eta_{2}} \otimes \id)(s^{-1} \cdot u)\bigr]
      \otimes \delta_{s^{-1}r}.
    \end{split}
  \end{equation}
  Taking \( s = r_{1}^{-1} \) and \( r = sr_{2} = r_{1}^{-1}r_{2} \) in
  \eqref{eq:8400ce42f4a3f560} proves \eqref{eq:1ff8e1270b4b575c} and finishes
  the proof of the lemma.
\end{proof}

\begin{prop}
  \label{prop:895e4cc33597e3de}
  The linear span of matrix coefficients of distinguished representations of
  \( \mathbb{G} \rtimes \Lambda \) in \( \pol(\mathbb{G}) \otimes C(\Lambda) \)
  is \( \pol(\mathbb{G}) \otimes C(\Lambda) \) itself. In particular, every
  unitary irreducible representation of \( \mathbb{G} \rtimes \Lambda \) is
  unitarily equivalent to a distinguished one.
\end{prop}
\begin{proof}
  The first assertion follows from Lemma~\ref{lemm:be8a473ead6cb802}, and the
  second assertion follows from the first and the orthogonality relations of
  irreducible representations of \( \mathbb{G} \rtimes \Lambda \).
\end{proof}

\section{Classification of irreducible representations of
  \texorpdfstring{\( \mathbb{G} \rtimes \Lambda \)}{the semidirect product}}
\label{sec:f7e4fdcede696c78}

For each isotropy subgroup \( \Lambda_{0} \) of \( \Lambda \), let
\( \mathfrak{D}_{\Lambda_{0}} \)denotes the collection of equivalence classes of
distinguished representation parameters associated with \( \Lambda_{0} \). By
Proposition~\ref{prop:01cb8f0781bcbbf7}, the mapping
\begin{equation}
  \label{eq:5e530c98e39ea9ef}
  \begin{split}
    \Psi_{\Lambda_{0}} \colon \mathfrak{D}_{\Lambda_{0}}
    & \to \irr(\mathbb{G} \rtimes \Lambda) \\
    [(u, V, v)] & \mapsto \mathscr{R}_{\Lambda_{0}}\bigl(\mathcal{S}(u, V,
    v)\bigr)
  \end{split}
\end{equation}
is well-defined and injective. In particular, \( \mathfrak{D}_{\Lambda_{0}} \)
is a set (instead of a proper class). Let \( \mathfrak{D} \) be the collection
of equivalence classes of distinguished representation parameters associated
with any isotropy subgroup of \( \Lambda \). By definition, \( \mathfrak{D} \)
is the disjoint union of \( \mathfrak{D}_{\Lambda_{0}} \) as \( \Lambda_{0} \)
runs through all isotropy subgroups of \( \Lambda \), hence \( \mathfrak{D} \)
is also a set. For any \( [(u, V, v)] \in \mathfrak{D}_{\Lambda_{0}} \) and any
\( r \in \Lambda \), \( r \cdot [(u, V, v)] = [r \cdot (u, V, v)] \) is a
well-defined class in \( \mathfrak{D}_{r\Lambda_{0}r^{-1}} \). This defines an
action of \( \Lambda \) on \( \mathfrak{D} \). We are now ready to state and
prove the classification of irreducible representations of
\( \mathbb{G} \rtimes \Lambda \).

\begin{theo}[Classification of irreducible representations of
  \( \mathbb{G} \rtimes \Lambda \)]
  \label{theo:91ac3dad0bd1219d}
  The mapping
  \begin{equation}
    \label{eq:1b6a87803ff7a68e}
    \begin{split}
      \Psi \colon \mathfrak{D} & \rightarrow \irr(\mathbb{G} \rtimes \Lambda) \\
      [(u, V, v)] \in \mathfrak{D}_{\Lambda_{0}} & \mapsto
      \Psi_{\Lambda_{0}}\bigl([(u, V, v)]\bigr) =
      \indrep\Bigl(\mathscr{R}_{\Lambda_{0}}\bigl(\mathcal{S}(u, V,
      v)\bigr)\Bigr)
    \end{split}
  \end{equation}
  is surjective, and the fibers of \( \Psi \) are exactly the
  \( \Lambda \)-orbits in \( \mathfrak{D} \).
\end{theo}
\begin{proof}
  By Proposition~\ref{prop:895e4cc33597e3de}, \( \Psi \) is surjective. By
  Corollary~\ref{coro:7b7b679f656174f9} and \eqref{eq:e2a420685bba7c4f}, each
  \( \Lambda \)-orbits in \( \mathfrak{D} \) maps to the same point under
  \( \Psi \). It remains to show that if \( (u_{i}, V_{i}, v_{i}) \) is a
  distinguished representation parameter with associated subgroup
  \( \Lambda_{i} \) for \( i = 1, 2 \), and
  \begin{equation}
    \label{eq:6c1d2e64c223b5af}
    \Psi\bigl([(u_{1}, V_{1}, v_{1})]\bigr)
    = \Psi\bigl([(u_{2}, V_{2}, v_{2})]\bigr),
  \end{equation}
  then there exists an \( r_{0} \in \Lambda \), such that
  \begin{equation}
    \label{eq:09da9aeaa8bffd22}
    r_{0} \cdot [(u_{1}, V_{1}, v_{1})]
    = [(u_{2}, V_{2}, v_{2})] \in \mathfrak{D}_{\Lambda_{2}}.
  \end{equation}
  Let \( \mathbf{S}_{i} = \mathcal{S}(u_{i}, V_{i}, v_{i}) \),
  \( U_{i} = \mathscr{R}_{\Lambda_{i}}(\mathbf{S}_{i}) \) for \( i = 1, 2 \). If
  \( [u_{2}] \notin \Lambda \cdot [u_{1}] \), then by
  Proposition~\ref{prop:7b2a964908c81d83}, we have
  \begin{equation}
    \label{eq:7af4d48b24b3f200}
    \forall r,s \in \Lambda, \;
    \dim \morph_{\mathbb{G} \rtimes \Lambda(r, s)}
    \Bigl((r \cdot U_{1}) \vert_{\mathbb{G} \rtimes
      \Lambda(r,s)}, (s \cdot U_{2}) \vert_{\mathbb{G} \rtimes \Lambda(r,s)}
    \Bigr) = 0,
  \end{equation}
  where \( \Lambda(r,s) = r\Lambda_{1}r^{-1} \cap s\Lambda_{2}s^{-1} \). This is
  because \( (r \cdot U_{1}) \vert_{\mathbb{G} \rtimes \Lambda(r, s)} \) is
  parameterized by the representation parameter
  \( (u_{1}, V_{1} \vert_{\Lambda(r,s)}, v_{1} \vert_{\Lambda(r,s)}) \)
  associated with \( \Lambda(r, s) \), and a similar assertion holds for
  \( (s \cdot U_{2}) \vert_{\mathbb{G} \rtimes \Lambda(r, s)} \). Thus
  \begin{equation}
    \label{eq:f0989dcb9abacfb2}
    \dim \morph_{\mathbb{G} \rtimes \Lambda}\bigl(\indrep(U_{1}),
    \indrep(U_{2})\bigr) = 0
  \end{equation}
  by Proposition~\ref{prop:a8bf0fe189bec421}, which contradicts
  \eqref{eq:6c1d2e64c223b5af}.

  Thus \( [u_{2}] \in \Lambda \cdot [u_{1}] \), by replacing
  \( [(u_{1}, V_{1}, v_{1})] \) with \( r_{0} \cdot [(u_{1}, V_{1}, v_{1})] \)
  for some \( r_{0} \in \Lambda \) if necessary, we may assume without loss of
  generality that \( [u_{1}] = [u_{2}] \in \irr(\mathbb{G}) \), and
  \( \Lambda_{1} = \Lambda_{2} \), which we now denote by \( \Lambda_{0} \). It
  remains to prove that under this assumption, we have
  \begin{equation}
    \label{eq:2381bc56b18e8a32}
    [(u_{1}, V_{1}, v_{1})]
    = [(u_{2}, V_{2}, v_{2})] \in \mathfrak{D}_{\Lambda_{0}}
  \end{equation}
  Since when \( r^{-1}s \notin \Lambda_{0} \) if and only if
  \( r \cdot [u_{1}] \neq s \cdot [u_{2}] \), we have
  \begin{equation}
    \label{eq:daa9843af0f41714}
    \forall r, s \in \Lambda, \quad r^{-1}s \notin \Lambda_{0} \implies
    \dim \morph_{\mathbb{G} \rtimes \Lambda(r, s)}
    \bigl((r \cdot U_{1})\vert_{\mathbb{G} \rtimes
      \Lambda(r,s)}, (r \cdot U_{2})\vert_{\mathbb{G} \rtimes \Lambda(r,s)}
    \bigr) = 0.
  \end{equation}
  Note that when \( r^{-1}s \in \Lambda_{0} \), we have
  \( \Lambda(r, s) = r\Lambda_{0}r^{-1} = s \Lambda_{0} s^{-1} \), and
  \( [\Lambda \colon \Lambda(r, s)] = [\Lambda \colon \Lambda_{0}] \). By
  \eqref{eq:6c1d2e64c223b5af}, \eqref{eq:daa9843af0f41714} and
  Proposition~\ref{prop:a8bf0fe189bec421}, we have
  \begin{equation}
    \label{eq:3dcf561cb808e7ca}
    1 = \frac{1}{\abs*{\Lambda_{0}}^{2}[\Lambda \colon \Lambda_{0}]}
    \sum_{\substack{r, s \in \Lambda, \\
        r^{-1}s \in \Lambda_{0}}}  \dim
    \morph_{\mathbb{G} \rtimes r\Lambda_{0}r^{-1}}\bigl(r \cdot U_{1}, s \cdot
    U_{2}\bigr).
  \end{equation}
  Since \( r \cdot U_{1} \), \( s \cdot U_{2} \) are both irreducible, we have
  \begin{equation}
    \label{eq:d5d9a29e04ec902e}
    r^{-1}s \in \Lambda_{0} \implies
    \dim\morph_{\mathbb{G} \rtimes r \Lambda_{0}r^{-1}}
    \bigl(r \cdot U_{1}, s \cdot U_{2}\bigr) = 0 \;\text{or}\; 1.
  \end{equation}
  Note that there are
  \( \abs*{\Lambda_{0}}^{2} [\Lambda \colon \Lambda_{0}] = \abs*{\Lambda} \cdot
  \abs*{\Lambda_{0}} \) terms on the right side of \eqref{eq:3dcf561cb808e7ca},
  \eqref{eq:d5d9a29e04ec902e} forces
  \begin{equation}
    \label{eq:9fae2a27c3284eff}
    r^{-1}s \in \Lambda_{0} \implies
    \dim\morph_{\mathbb{G} \rtimes r \Lambda_{0}r^{-1}}
    \bigl(r \cdot U_{1}, s \cdot U_{2}\bigr) = 1.
  \end{equation}
  In particular, taking \( r = s = 1_{\Lambda} \) in \eqref{eq:9fae2a27c3284eff}
  shows that \( U_{1} \) and \( U_{2} \) are equivalent, hence
  \eqref{eq:2381bc56b18e8a32} holds by
  Proposition~\ref{prop:01cb8f0781bcbbf7}. This finishes the proof of the
  theorem.
\end{proof}

\section{The conjugate representation of distinguished representations}
\label{sec:fd70c0937e9538e4}

We now study the conjugation of irreducible representations of
\( \mathbb{G} \rtimes \Lambda \) in terms of the classification presented in
Theorem~\ref{theo:91ac3dad0bd1219d}. There is a small complication here in the
non-Kac type case, where the contragredient of a unitary representation need not
be unitary. Resolving this kind of question involves the modular operator, just
as in Proposition~\ref{prop:1ecfd88218e081c9}.

We begin with a simple lemma on linear operators.

\begin{lemm}
  \label{lemm:02b9a7af62be9f6b}
  Let \( \mathscr{H} \) be a Hilbert space,
  \( U, P \in \mathcal{B}(\mathscr{H}) \) such that \( U \) is unitary, \( P \)
  is invertible and positive, if \( PUP^{-1} \) is unitary, then
  \( PUP^{-1} = U \), i.e.\ \( P \) commutes with \( U \).
\end{lemm}
\begin{proof}
  Let \( V = PUP^{-1} \). We have
  \begin{equation}
    \label{eq:9bcc26e060fd68ef}
    PU^{\ast}P^{-1} = PU^{-1}P^{-1} = V^{-1} = V^{\ast}
    = P^{-1}U^{\ast}P.
  \end{equation}
  Thus \( U^{\ast} \) commutes with the positive operator \( P^{2} \). Hence
  \( U^{\ast} \) commutes with \( {(P^{2})}^{1/2} = P \), i.e.\
  \( U^{\ast}P = PU^{\ast} \). Taking adjoints of this proves \( PU = UP \).
\end{proof}

\begin{prop}
  \label{prop:9c465524a5b2723f}
  Let \( u \) be an irreducible unitary representation of \( \mathbb{G} \),
  \( \Lambda_{0} \) a subgroup of the isotropy subgroup \( \Lambda_{[u]} \),
  \( V \) a covariant projective \( \Lambda_{0} \)-representation of \( u \).
  Then any operator \( \rho\in \morph_{\mathbb{G}}(u, u^{cc}) \) commutes with
  \( V \) (i.e.\ \( \rho V(r_{0}) = V(r_{0}) \rho \) for all
  \( r_{0} \in \Lambda_{0} \)).
\end{prop}
\begin{proof}
  Since \( u \) is irreducible, \( \morph_{\mathbb{G}}(u, u^{cc}) \) is a one
  dimensional space spanned by an invertible positive operator (\cite[Lemma
  1.3.12]{MR3204665}).  By definition (see~\cite[Proposition 1.4.4 and
  Definition 1.4.5]{MR3204665}), the conjugation \( \overline{u} \) of \( u \)
  is given by
  \begin{equation}
    \label{eq:779418ead6089548}
    \overline{u} = \bigl({j(\rho_{u})}^{1/2} \otimes
    1\bigr)u^{c}\bigl({j(\rho_{u})}^{-1/2} \otimes 1\bigr),
  \end{equation}
  where \( \rho_{u} \) is the unique \emph{positive} operator in
  \( \morph_{\mathbb{G}}(u, u^{cc}) \) with
  \( \tr(\rho_{u}) = \tr(\rho_{u}^{-1}) \). Since
  \( \morph_{\mathbb{G}}(u, u^{cc}) = \mathbb{C} \rho_{u} \), it suffices to
  show that \( \rho_{u} \) commutes with \( V \).

  Since \( u \), \( V \) are covariant, we have
  \begin{equation}
    \label{eq:630e384e8d91b821}
    \forall r_{0} \in \Lambda_{0}, \;
    \bigl(V(r_{0}) \otimes 1\bigr) (r_{0} \cdot  u) = u \bigl(V(r_{0})
    \otimes 1\bigr).
  \end{equation}
  Taking the adjoint of both sides of \eqref{eq:630e384e8d91b821} then applying
  \( j \otimes \id \), we get
  \begin{equation}
    \label{eq:c37e02aeb0b4ee2d}
    \forall r_{0} \in \Lambda_{0}, \;
    \bigl(V^{c}(r_{0}) \otimes 1\bigr) (r_{0} \cdot  u^{c})
    = u^{c} \bigl(V^{c}(r_{0}) \otimes 1\bigr),
  \end{equation}
  where
  \begin{equation}
    \label{eq:a57a14f5b27f77c0}
    V^{c} = (j \otimes \id)(V^{-1}) = (j \otimes \id)(V^{\ast})
  \end{equation}
  is the contragredient of \( V \), and
  \begin{equation}
    \label{eq:fc3b3bd021d1a480}
    u^{c} = (j \otimes \id)(u^{-1}) = (j \otimes \id)(u^{\ast})
  \end{equation}
  the contragredient of \( u \). We pose
  \begin{equation}
    \label{eq:ba91c9f504635eb1}
    \overline{V} = \bigl({j(\rho_{u})}^{1/2} \otimes 1\bigr)
    V^{c} \bigl({j(\rho_{u})}^{-1/2} \otimes 1\bigr),
  \end{equation}
  then by \eqref{eq:c37e02aeb0b4ee2d} and \eqref{eq:779418ead6089548}, we have
  \begin{equation}
    \label{eq:ee1e03f0aa27c5ad}
    \forall r_{0} \in \Lambda_{0}, \; \bigl(\overline{V}(r_{0}) \otimes 1\bigr)
    (r_{0} \cdot
    \overline{u}) = \overline{u} \bigl(\overline{V}(r_{0}) \otimes 1\bigr).
  \end{equation}
  Thus for any \( r_{0} \in \Lambda_{0} \),
  \( \overline{V}(r_{0}) \in \morph_{\mathbb{G}}(r_{0} \cdot \overline{u},
  \overline{u}) \), which is a one dimensional space spanned by a unitary
  operator since both \( r_{0} \cdot \overline{u} \) and \( \overline{u} \) are
  irreducible unitary representations of \( \mathbb{G} \). Note that
  \( V^{c}(r_{0}) = j\bigl({V(r_{0})}^{\ast}\bigr) \) is unitary, by
  \eqref{eq:ba91c9f504635eb1}, we have
  \begin{equation}
    \label{eq:271dd114db8136d8}
    \det\bigl(\overline{V}(r_{0})\bigr) = \det\bigl({j(\rho_{u})}^{1/2}
    V^{c}(r_{0}) {j(\rho_{u})}^{-1/2}\bigr) =
    \det\bigl(V^{c}(r_{0})\bigr) \in \mathbb{T}.
  \end{equation}
  This forces \( \overline{V}(r_{0}) \) to be unitary since it is a scalar
  multiple of a unitary operator. Applying Lemma~\ref{lemm:02b9a7af62be9f6b}
  to~\eqref{eq:ba91c9f504635eb1} (evaluated on each
  \( r_{0} \in \Lambda_{0} \)), we see that
  \begin{equation}
    \label{eq:0d0c406d77493efc}
    V^{c} = \overline{V}
    = \bigl({j(\rho_{u})}^{1/2} \otimes 1\bigr)
    V^{c} \bigl({j(\rho_{u})}^{-1/2} \otimes 1\bigr).
  \end{equation}
  Applying \( j \otimes \id \) to the inverse of both sides of
  \eqref{eq:0d0c406d77493efc} and note that \( V^{cc} = V \), we see that
  \begin{equation}
    \label{eq:cabe366bb8425892}
    V = V^{cc} = (\rho_{u}^{1/2} \otimes 1) V^{cc} (\rho_{u}^{-1/2} \otimes 1)
    = (\rho_{u}^{1/2} \otimes 1) V (\rho_{u}^{-1/2} \otimes 1),
  \end{equation}
  i.e.\ \( \rho_{u}^{1/2} \) (hence \( \rho_{u} \)) commutes with \( V \).
\end{proof}

\begin{prop}
  \label{prop:ef1711719362d940}
  Let \( (u, V, v) \) be a representation parameter associated with some
  \( \Lambda_{0} \in \giso(\Lambda) \), \( U \) is the unitary representation of
  \( \mathbb{G} \rtimes \Lambda_{0} \) parameterized by \( (u, V, v) \), then
  the following hold:
  \begin{enumerate}
  \item \label{item:8749303f3fc0ece3} \( (\overline{u}, V^{c}, v^{c}) \) is also
    a representation parameter;
  \item \label{item:75a9391ad5acc60a}
    \( \rho_{U} = \id_{\mathscr{H}_{v}} \otimes \rho_{u} \), where
    \( \rho_{U} \) (resp.\ \( \rho_{u} \)) is the modular operator for the
    representation \( U \) (resp.\ \( u \));
  \item \label{item:ee918374fdbf59bb} \( \overline{U} \) is parameterized by
    \( (\overline{u}, V^{c}, v^{c}) \).
  \end{enumerate}
\end{prop}
\begin{proof}
  As we've seen in Proposition~\ref{prop:9c465524a5b2723f} and its proof, we
  have \( \overline{V}(r_{0}) \in \morph_{\mathbb{G}}(r_{0} \cdot u, u) \) for
  all \( r_{0} \in \Lambda_{0} \), thus \( V^{c} = \overline{V} \) is covariant
  with \( \overline{u} \). Since
  \begin{equation}
    \label{eq:21ed7a51560a7d99}
    \begin{split}
      \forall r_{0} \in \Lambda_{0}, \qquad (v^{c} \times V^{c})(r_{0}) &=
      j\bigl({[v(r_{0})]}^{-1} \otimes
      {[V(r_{0})]}^{-1}\bigr) \\
      &= j\bigl(\{[{(v \times V)}](r_{0})\}^{-1}\bigr)
    \end{split}
  \end{equation}
  \( v^{c} \times V^{c} \) is the contragredient of the unitary representation
  \( v \times V \) of \( \Lambda_{0} \), hence is a unitary representation
  itself. Thus \( v^{c} \) and \( V^{c} \) are unitary projective
  representations with opposing cocycles. This proves
  \ref{item:8749303f3fc0ece3}.

  To prove \ref{item:75a9391ad5acc60a}, by the characterizing property of
  \( \rho_{U} \), it suffices to show that the invertible positive operator
  \( \id \otimes \rho_{u} \) satisfies
  \begin{equation}
    \label{eq:50f9f3d0d00e78f4}
    \id \otimes \rho_{u} \in \morph_{\mathbb{G}}(U, U^{cc})
  \end{equation}
  and (by Proposition~\ref{prop:50b58e1fd8af4aa0} and Schur's lemma applied to
  the irreducible representation \( u \))
  \begin{equation}
    \label{eq:4d6438289d07877c}
    \begin{split}
      & \leadmathskip \tr(( \cdot ) (\id \otimes \rho_{u}))
      = \tr(( \cdot )(\id \otimes \rho_{u}^{-1})) \\
      & \in \selfmorph_{\mathbb{G} \rtimes \Lambda_{0}}(U) =
      \selfmorph_{\mathbb{G}}(\id \otimes u) \cap \selfmorph_{\Lambda_{0}}(v
      \times V) \subseteq \mathcal{B}(\mathscr{H}_{v}) \otimes \mathbb{C} \id.
    \end{split}
  \end{equation}
  Since \( \tr(\rho_{u}) = \tr(\rho_{u}^{-1}) \), \eqref{eq:4d6438289d07877c}
  holds. We now prove~\eqref{eq:50f9f3d0d00e78f4}.  As is seen in the proof of
  Proposition~\ref{prop:9c465524a5b2723f}, condition~\eqref{eq:630e384e8d91b821}
  holds, and a similar calculation by applying \( j \otimes \id \) to the
  inverse of both sides of \eqref{eq:630e384e8d91b821} yields (note that
  \( V^{cc} = V \)),
  \begin{equation}
    \label{eq:9d8984165d848a38}
    \forall r_{0} \in \Lambda_{0}, \;
    \bigl(V(r_{0}) \otimes 1\bigr)(r_{0} \cdot u^{cc}) =
    u^{cc} \bigl(V(r_{0}) \otimes 1\bigr).
  \end{equation}
  By definition, we have
  \begin{equation}
    \label{eq:3290c1f8095bc211}
    \begin{split}
      U &= {(\id \otimes u)}_{123}{(v \times V)}_{124}
      = (\id \otimes u \otimes 1) v_{14} V_{24}  \\
      &\in \mathcal{B}(\mathscr{H}_{v}) \otimes \mathcal{B}(\mathscr{H}_{u})
      \otimes \pol(\mathbb{G}) \otimes C(\Lambda_{0}).
    \end{split}
  \end{equation}
  Thus
  \begin{equation}
    \label{eq:e761183ab609bf71}
    U^{c} = (j \otimes j \otimes \id \otimes \id)(U^{-1}) %
    = (\id \otimes u^{c} \otimes 1)v^{c}_{14}V^{c}_{24},
  \end{equation}
  and
  \begin{equation}
    \label{eq:561749f196e30f45}
    U^{cc} = (\id \otimes u^{cc} \otimes 1)v^{cc}_{14}
    V^{cc}_{24} = (\id \otimes u^{cc}
    \otimes 1)v_{14}V_{24}.
  \end{equation}
  By~\eqref{eq:3290c1f8095bc211}, \eqref{eq:561749f196e30f45} and
  Proposition~\ref{prop:9c465524a5b2723f}, we have
  \begin{equation}
    \label{eq:8fe68e92d7edf090}
    \begin{split}
      &\leadmathskip (\id \otimes \rho_{u} \otimes 1 \otimes 1) U %
      = (\id \otimes \rho_{u} \otimes 1 \otimes 1)
      (\id \otimes u \otimes 1) v_{14} V_{24} \\
      &= (\id \otimes u^{cc} \otimes 1)v_{14}[(\id \otimes \rho_{u} \otimes 1
      \otimes 1)V_{24}] \\
      &= (\id \otimes u^{cc} \otimes 1)v_{14}V_{24}
      (\id \otimes \rho_{u} \otimes 1 \otimes 1) \\
      &= U^{cc}(\id \otimes \rho_{u} \otimes 1 \otimes 1).
    \end{split}
  \end{equation}
  This proves~\eqref{eq:50f9f3d0d00e78f4} and finishes the proof of
  \ref{item:75a9391ad5acc60a}.

  By Proposition~\ref{prop:1ecfd88218e081c9} and \ref{item:75a9391ad5acc60a},
  \( \overline{U} \) corresponds to the CSR
  \( (\overline{\mathscr{H}_{u}}, u', w') \) in
  \( \mathcal{CSR}_{\Lambda_{0}} \), where \( \mathscr{H}_{u} \) is the
  underlying finite dimensional Hilbert space of \( u \),
  \begin{equation}
    \label{eq:6d1f358c37d9cf35}
    u' = ({\id \otimes j(\rho)}^{1/2} \otimes 1)(\id \otimes u^{c})%
    (\id \otimes {j(\rho)}^{-1/2} \otimes 1) = \id \otimes \overline{u},
  \end{equation}
  and
  \begin{equation}
    \label{eq:dc7bec653268cb6d}
    \begin{split}
      w' &= (\id \otimes {j(\rho)}^{1/2} \otimes 1) (v^{c}_{13}V^{c}_{23}) %
      (\id \otimes {j(\rho)}^{-1/2} \otimes 1) \\
      &= v^{c}_{13} [(\id \otimes {j(\rho)}^{1/2} \otimes 1) V^{c}_{23} %
      (\id \otimes {j(\rho)}^{1/2} \otimes 1)] \\
      &= v^{c}_{13}V^{c}_{23} = v^{c} \times V^{c}.
    \end{split}
  \end{equation}
  Thus the CSR \( (\overline{\mathscr{H}_{u}}, u', w') \), and consequently
  \( \overline{U} \), is indeed parameterized by
  \( (\overline{u}, v^{c}, V^{c}) \), which proves \ref{item:ee918374fdbf59bb}.
\end{proof}

Proposition~\ref{prop:9c465524a5b2723f} motivates the following definition.

\begin{defi}
  \label{defi:c2787a2357939225}
  Let \( (u, V, v) \) be a representation parameter associated with some
  \( \Lambda_{0} \in \giso(\Lambda) \), the representation parameter
  \( (\overline{u}, V^{c}, v^{c}) \) is called the \textbf{conjugate} of
  \( (u, V, v) \).
\end{defi}

By Proposition~\ref{prop:01cb8f0781bcbbf7} and
Corollary~\ref{coro:d09f1fd81ad53d88}, it is clear that the conjugation of an
irreducible representation parameter is irreducible, and
\( \overline{[(u, V, v)]} = [(\overline{u}, V^{c}, v^{c})] \) gives a
well-defined mapping
\( \overline{( \cdot )} : \mathfrak{D} \rightarrow \mathfrak{D} \). The
following theorem describes how the conjugate representation of irreducible
(unitary) representation of \( \mathbb{G} \rtimes \Lambda \) looks like in terms
of the classification given in Theorem~\ref{theo:91ac3dad0bd1219d}.

\begin{theo}
  \label{theo:e6d75eb838fadd50}
  Let \( [(u, V, v)] \in \mathfrak{D} \),
  \( x = \Psi([(u, V, v)]) \in \irr(\mathbb{G} \rtimes \Lambda) \), then
  \begin{equation}
    \label{eq:96c4babd2d220a8b}
    \overline{x} = \Psi(\overline{[(u, V, v)]}) = \Psi([(\overline{u},
    V^{c}, v^{c})]).
  \end{equation}
\end{theo}
\begin{proof}
  This follows immediately from Proposition~\ref{prop:ef1711719362d940} and the
  character formula \eqref{eq:0391706cde602b93} for representations induced from
  representations of principal subgroups of \( \mathbb{G} \rtimes \Lambda \).
\end{proof}

\section{The incidence numbers}
\label{sec:76ceda69b980b4d5}

We now turn our attention to the fusion rules of
\( \mathbb{G} \rtimes \Lambda \). We define and study incidence numbers in this
section, and use these numbers to express the fusion rules in
\S~\ref{sec:20a173060e7c2f82}.

\begin{defi}
  \label{defi:969734aca7f85d19}
  For \( i = 1, 2, 3 \), let \( \Lambda_{i} \in \giso(\Lambda) \). Suppose
  \( U_{i} \) is a unitary representation of
  \( \mathbb{G} \rtimes \Lambda_{i} \), and \( r_{i} \in \Lambda \), then the
  \textbf{incidence number} of \( (r_{1}, r_{2}, r_{3}) \) relative to
  \( (U_{1}, U_{2}, U_{3}) \), denoted by
  \( m_{U_{1}, U_{2}, U_{3}}(r_{1}, r_{2}, r_{3}) \), is defined by
  \begin{equation}
    \label{eq:ba5678d8c0bd5298}
    \begin{split}
      & \leadmathskip m_{U_{1}, U_{2}, U_{3}}(r_{1}, r_{2}, r_{3}) \\
      &= \dim \morph_{\mathbb{G} \rtimes \Lambda_{0}}\left((r_{1} \cdot
        U_{1})\vert_{\mathbb{G} \rtimes \Lambda_{0}}, (r_{2} \cdot
        U_{2})\vert_{\mathbb{G} \rtimes \Lambda_{0}} \times (r_{3} \cdot
        U_{3})\vert_{\mathbb{G} \rtimes \Lambda_{0}}\right),
    \end{split}
  \end{equation}
  where \( \Lambda_{0} = \cap_{i=1}^{3}r_{i} \Lambda_{i} r_{i}^{-1} \).
\end{defi}

We now aim to express the incidence numbers in terms of characters. Let
\( \Theta, \Xi \) be two subgroups of \( \Lambda \) with
\( \Theta \subseteq \Xi \). Recall that \( C(\mathbb{G}) = A \). Suppose
\( F = \sum_{r \in \Xi} a_{r} \otimes \delta_{r} \), \( a_{r} \in A \) is an
element of \( C(\mathbb{G}) \otimes C(\Xi) = A \otimes C(\Xi) \). We use
\( F \vert_{\mathbb{G} \rtimes \Theta} \) to denote the element
\( \sum_{r \in \Theta} a_{r} \otimes \delta_{r} \) in
\( \mathbb{G} \rtimes \Theta \), and call it the restriction of \( F \) to
\( \mathbb{G} \rtimes \Theta \). A simple calculation shows that this
restriction operation gives a surjective unital morphism of
{\( C^{\ast} \)}\nobreakdash-algebras from
\( C(\mathbb{G} \rtimes \Xi) = A \otimes C(\Xi) \) to
\( C(\mathbb{G} \rtimes \Theta) = A \otimes C(\Theta) \) that also preserves
comultiplication, thus allows us to view \( \mathbb{G} \rtimes \Theta \) as a
closed subgroup of \( \mathbb{G} \rtimes \Xi \) in the sense of
Definition~\ref{defi:2f7e12385791f036}. Recall that we also have the extension
morphism \( E_{\Lambda_{0}} \colon C(\Lambda_{0}) \rightarrow C(\Lambda) \),
\( \delta_{r_{0}} \mapsto \delta_{r_{0}} \) for every subgroup \( \Lambda_{0} \)
of \( \Lambda \), which simply sends each function in \( C(\Lambda_{0}) \) to
its unique extension in \( C(\Lambda) \) that vanishes outside
\( \Lambda_{0} \). Finally, we use \( h^{\Lambda_{0}} \) to denote the Haar
state on \( \mathbb{G} \rtimes \Lambda_{0} \). For \( i = 1, 2, 3 \), let
\( \Lambda_{i} \in \giso(\Lambda) \). Suppose \( U_{i} \) is a unitary
representation of \( \mathbb{G} \rtimes \Lambda_{i} \), \( \chi_{i} \) is the
character of \( U_{i} \). Let
\( \Lambda_{0} = \cap_{i=1}^{3}r_{i}\Lambda_{i}r_{i}^{-1} \). Then we have the
following formula to calculate the incidence numbers in terms of characters.
\begin{equation}
  \label{eq:6c5ca51e74d82015}
  \begin{split}
    \forall r_{1}, r_{2}, r_{3} \in \Lambda, \quad & \leadmathskip m_{U_{1},
      U_{2}, U_{3}}(r_{1}, r_{2}, r_{3}) \\
    &= h^{\Lambda_{0}}\left( \overline{(r_{1} \cdot \chi_{1})\vert_{\mathbb{G}
          \rtimes \Lambda_{0}}} (r_{2} \cdot \chi_{2})\vert_{\mathbb{G} \rtimes
        \Lambda_{0}}(r_{3} \cdot \chi_{3})\vert_{\mathbb{G} \rtimes
        \Lambda_{0}}\right).
  \end{split}
\end{equation}

\begin{prop}
  \label{prop:356dd34562dc20c4}
  Using the above notations, the incidence number
  \( m_{U_{1},U_{2},U_{3}}(s_{1},s_{2},s_{3}) \) depends only on the classes
  \( [U_{1}] \), \( [U_{2}] \), \( [U_{3}] \) of equivalent unitary
  representations and the left cosets \( r_{1}\Lambda_{1} \),
  \( r_{2}\Lambda_{2} \), \( r_{3}\Lambda_{3} \).
\end{prop}
\begin{proof}
  Note that for any \( i = 1, 2, 3 \),
  \( s_{i}\Lambda_{i}s_{i}^{-1} = r_{i}\Lambda_{i}r_{i}^{-1} \) whenever
  \( r_{i}^{-1}s_{i} \in \Lambda_{i} \). The proposition follows
  from~\eqref{eq:6c5ca51e74d82015} and Lemma~\ref{lemm:96218436efb1ec44}
  \ref{item:38dcc9ada3065523}.
\end{proof}

By Proposition~\ref{prop:356dd34562dc20c4}, we see immediately that the
following definition is well-defined.

\begin{defi}
  \label{defi:b0695ae26a3edb1d}
  For \( i = 1, 2, 3 \), let \( \Lambda_{i} \in \giso(\Lambda) \). Suppose
  \( x_{i} \) is a class of equivalent unitary representations of
  \( \mathbb{G} \rtimes \Lambda_{i} \), and
  \( z_{i} \in \Lambda / \Lambda_{i} \) is a left coset of \( \Lambda_{i} \) in
  \( \Lambda \), then the \textbf{incidence number} of
  \( (z_{1}, z_{2}, z_{3}) \) relative to \( (x_{1}, x_{2}, x_{3}) \), denoted
  by \( m_{x_{1}, x_{2}, x_{3}}(z_{1}, z_{2}, z_{3}) \), is defined by
  \begin{equation}
    \label{eq:e9b55f739bf5a902}
    m_{x_{1},x_{2},x_{3}}(z_{1},z_{2},z_{3}) %
    =
    m_{U_{1}, U_{2}, U_{3}}(r_{1}, r_{2}, r_{3}) %
  \end{equation}
  where \( U_{i} \in x_{i} \), \( r_{i} \in z_{i} \) for \( i = 1, 2, 3 \).
\end{defi}

The rest of this section is devoted to the calculation of the incidence number
\eqref{eq:e9b55f739bf5a902} in terms of more basic ingredients when
\( x_{i} = \Phi_{\Lambda_{i}}(\mathfrak{p}_{i}) \) for some
\( \mathfrak{p}_{i} \in \mathfrak{D}_{\Lambda_{i}} \) (see
\S~\ref{sec:f7e4fdcede696c78}), as this will be the case we need in the
calculation of fusion rules for \( \mathbb{G} \rtimes \Lambda \) in
\S~\ref{sec:20a173060e7c2f82}. We begin with a result on the structure of
unitary projective representations of some \( \Lambda_{0} \in \giso(\Lambda) \)
that are covariant with some unitary representation of \( \mathbb{G} \).

\begin{lemm}
  \label{lemm:627a277420a433bc}
  Fix a \( \Lambda_{0} \in \giso(\Lambda) \). Let \( u_{0} \) be an irreducible
  unitary representation of \( \mathbb{G} \), \( [u_{0}] \in \irr(\mathbb{G}) \)
  the class of \( u_{0} \), such that
  \( \Lambda_{0} \subseteq \Lambda_{[u_{0}]} \). Suppose \( u \) is a unitary
  representation of \( \mathbb{G} \),
  \( V \colon \Lambda_{0} \rightarrow \mathcal{U}(\mathscr{H}_{u}) \) is a
  unitary projective representation covariant with \( u \), \( p \) is the
  minimal central projection in \( \selfmorph_{\mathbb{G}}(u) \) corresponding
  to the maximal pure subrepresentation of \( u \) supported by
  \( [u_{0}] \in \irr(\mathbb{G}) \). Let \( q = 1 - p \), then \( V \) is
  diagonalizable along \( p \) in the sense that
  \begin{equation}
    \label{eq:9d4194791588813c}
    \begin{split}
      & (p \otimes 1) V = V (p \otimes 1), \quad%
      (q \otimes 1)V = V(q \otimes 1), \\
      & \text{and }\; (p \otimes 1) V (q \otimes 1) = (q \otimes 1) V (p \otimes
      1) = 0.
    \end{split}
  \end{equation}
\end{lemm}
\begin{proof}
  Since \( V \) and \( u \) are covariant, we have
  \begin{equation}
    \label{eq:3c3719019e34e564}
    \forall r_{0} \in \Lambda_{0}, \;
    \bigl(V(r_{0}) \otimes 1\bigr) (r_{0} \cdot u) = u
    \bigl(V(r_{0}) \otimes 1\bigr).
  \end{equation}
  Note that
  \( p \in \selfmorph_{\mathbb{G}}(u) = \selfmorph_{\mathbb{G}}(r_{0} \cdot u)
  \) (see \eqref{eq:f5ecd1bd13d50e74}), then for every
  \( r_{0} \in \Lambda_{0} \), it follows that
  \begin{equation}
    \label{eq:17b79106eef14e43}
    \begin{split}
      & \leadmathskip \bigl([p V(r_{0}) q] \otimes 1\bigr) [(q \otimes 1)(r_{0}
      \cdot u)] %
      = (p \otimes 1) \bigl(V(r_{0})
      \otimes 1\bigr) (q \otimes 1) (r_{0} \cdot u) \\
      &= (p \otimes 1) \bigl(V(r_{0}) \otimes 1\bigr)(r_{0} \cdot u)(q \otimes
      1) %
      = (p \otimes 1) u \bigl(V(r_{0}) \otimes 1\bigr) (q \otimes 1) \\
      &= [(p \otimes 1)u] (p \otimes 1) \bigl(V(r_{0}) \otimes 1\bigr) (q
      \otimes 1) %
      = [(p \otimes 1)u] \bigl([p V(r_{0}) q] \otimes 1\bigr).
    \end{split}
  \end{equation}
  Let \( u_{p} \) (resp.\ \( u_{q} \)) be the subrepresentation of \( u \)
  corresponding to \( p \) (resp.\ \( q \)), then \( r_{0}^{-1} \cdot u_{p} \)
  is equivalent to \( u_{p} \) for all \( r_{0} \in \Lambda_{0} \) since
  \( \Lambda_{0} \subseteq \Lambda_{[u]} \), and
  \begin{equation}
    \label{eq:dfb7241d3fcc6167}
    \morph_{\mathbb{G}}(r_{0} \cdot u_{q}, u_{p}) %
    = \morph_{\mathbb{G}}(u_{q}, r_{0}^{-1} \cdot u_{p}) %
    = \morph_{\mathbb{G}}(u_{q}, u_{p}) = 0.
  \end{equation}
  By \eqref{eq:17b79106eef14e43}, the operator \( p V(r_{0}) q \), when viewed
  as an operator from \( p(\mathscr{H}_{u}) \) to \( q(\mathscr{H}_{u}) \),
  intertwines \( r_{0} \cdot u_{q} \) and \( u_{p} \). Thus by
  \eqref{eq:dfb7241d3fcc6167},
  \begin{equation}
    \label{eq:443e1050525d2dbf}
    \forall r_{0} \in \Lambda_{0}, \quad p V(r_{0}) q = 0.
  \end{equation}
  Similarly,
  \begin{equation}
    \label{eq:54afba55baafd4c2}
    \forall r_{0} \in \Lambda_{0}, \quad q V(r_{0}) p = 0.
  \end{equation}
  Hence
  \begin{equation}
    \label{eq:95370d0136ae21af}
    pV(r_{0}) = pV(r_{0})(p + q) = pV(r_{0})p = (p + q)V(r_{0})p = V(r_{0})p,
  \end{equation}
  and similarly,
  \begin{equation}
    \label{eq:ec5c39b6909ec8b2}
    qV(r_{0}) = qV(r_{0})(p + q) = qV(r_{0})q = (p + q) V(r_{0}) q = V(r_{0})q.
  \end{equation}
  Now \eqref{eq:9d4194791588813c} follows from equations
  \eqref{eq:443e1050525d2dbf}, \eqref{eq:54afba55baafd4c2},
  \eqref{eq:95370d0136ae21af}, and \eqref{eq:ec5c39b6909ec8b2}.
\end{proof}

We also need to generalize the notion of representation parameter a little, as
the natural candidate of the ``tensor product'' of two representation parameters
need not be a representation parameter, but it still possesses the same
covariant property.

\begin{defi}
  \label{defi:45b581db456c6b41}
  Let \( \Lambda_{0} \in \giso(\Lambda) \), we call a triple \( (u, V, v) \) a
  \textbf{generalized representation parameter} (\textbf{(GRP)} for short)
  associated with \( \Lambda_{0} \), if the following hold:
  \begin{enumerate}
  \item \label{item:06768f2a5ad600fe} \( V \) is a unitary projective
    representation of \( \Lambda_{0} \) on \( \mathscr{H}_{u} \), such that
    \begin{equation}
      \label{eq:f6e1f3308142f678}
      \forall r_{0} \in \Lambda_{0}, \quad V(r_{0}) %
      \in \morph_{\mathbb{G}}(r_{0} \cdot u, u);
    \end{equation}
  \item \label{item:191e6543be37041e} \( v \) is a unitary projective
    representation (on some other finite dimensional Hilbert space
    \( \mathscr{H}_{v} \)) of \( \Lambda_{0} \), such that the cocycles of
    \( v \) and \( V \) are opposite to each other.
  \end{enumerate}
\end{defi}

\begin{prop}
  \label{prop:c465e919e5fa2023}
  If \( (u, V, v) \) is a GRP associated with some
  \( \Lambda_{0} \in \giso(\Lambda) \), then
  \( (\mathscr{H}_{v} \otimes \mathscr{H}_{u}, \id \otimes u, v \times V) \in
  \mathcal{CSR}_{\Lambda_{0}} \).
\end{prop}
\begin{proof}
  The proof of Proposition~\ref{prop:1911ec535b058627} applies almost verbatim
  here.
\end{proof}

\begin{defi}
  \label{defi:d8aeb561dd005934}
  If \( (u, V, v) \) is a GRP associated with
  \( \Lambda_{0} \in \giso(\Lambda) \), then the CSR
  \( \mathbf{S} := (\mathscr{H}_{v} \otimes \mathscr{H}_{u}, \id \otimes u, v
  \times V) \) associated with \( \Lambda_{0} \) and the unitary representation
  \( \mathscr{R}_{\Lambda_{0}}(\mathbf{S}) \) of
  \( \mathbb{G} \rtimes \Lambda_{0} \) are said to be parameterized by
  \( (u, V, v) \).
\end{defi}

We now describe a reduction process for generalized representation parameters,
which leads to our desired calculation of the incidence numbers using more basic
ingredients---the dimension of a certain intertwiner space of two projective
representations of some generalized isotropy subgroup of \( \Lambda \).

\begin{prop}
  \label{prop:5a82a06194fb75ea}
  Fix a \( \Lambda_{0} \in \giso(\Lambda) \). Let \( (u, V, v) \) be a GRP
  associated with \( \Lambda_{0} \), \( x \in \irr(\mathbb{G}) \) such that
  \( \Lambda_{0} \subseteq \Lambda_{x} \), and \( u_{0} \in x \). Suppose
  \( p \) is the minimal central projection of \( \selfmorph_{\mathbb{G}}(u) \)
  corresponding to the maximal pure subrepresentation of \( u \) supported by
  \( x \). The following holds:
  \begin{enumerate}
  \item \label{item:ff73a94b08e8bbb8} \( (u_{p}, V_{p}, v) \) is a GRP, where
    \( u_{p} \) (resp.\ \( V_{p} \)) is the subrepresentation of \( u \) (resp.\
    \( V \)) on \( p(\mathscr{H}_{u}) \);
  \item \label{item:61b4524c2f8e930a} let \( n \in \mathbb{N} \) be the
    multiplicity of \( x \) in \( u \), \( V_{0} \) a covariant projective
    \( \Lambda_{0} \)-representation of \( u_{0} \), then \emph{up to unitary
      equivalence}, there exists a unique unitary projective representation
    \( V_{1} \) of \( \Lambda_{0} \) on \( \mathbb{C}^{n} \), such that
    \( V_{p} \) is unitarily equivalent to \( V_{1} \times V_{0} \);
  \item \label{item:9fe46e8db326eca7} \( (u_{0}, V_{0}, v \times V_{1}) \) is
    representation parameter, and the CSR
    \( (\mathscr{H}_{v} \otimes p(\mathscr{H}_{u}),\id \otimes u_{p}, v \times
    V_{p}) \) parameterized by \( (u_{p}, V_{p}, v) \) is isomorphic to the CSR
    \( \mathcal{S}(u_{0}, V_{0}, v \times V_{1}) \) parameterized by
    \( (u_{0}, V_{0}, v \times V_{1}) \) in the category
    \( \mathcal{CSR}_{\Lambda_{0}} \). In particular, the representation
    parameter \( (u_{0}, V_{0}, v \times V_{1}) \) and the GRP
    \( (u_{p}, V_{p}, v \times V_{p}) \) parameterize equivalent unitary
    representations of \( \mathbb{G} \rtimes \Lambda_{0} \).
  \end{enumerate}
\end{prop}

\begin{proof}
  By Lemma~\ref{lemm:627a277420a433bc}, \( u_{p} \) and \( V_{p} \) are
  covariant. Since \( V_{p} \) is a subrepresentation of \( V \), it has the
  same cocycle as \( V \), hence \( V_{p} \) and \( v \) have opposing
  cocycles. This proves \ref{item:ff73a94b08e8bbb8}.

  The proof of \ref{item:61b4524c2f8e930a} parallels that of
  Proposition~\ref{prop:f8d2ac97bef564cc}. Since \( u_{p} \) is equivalent to a
  direct sum of \( n \) copies of \( u_{0} \), thus there exists a unitary
  operator
  \( U \in \morph_{\mathbb{G}}(\id_{\mathbb{C}^{n}} \otimes u_{0}, u_{p})
  \). Replace \( (u_{p}, V_{p}, v) \) with
  \( (U^{\ast}u_{p}U, U^{\ast}V_{p}U, v) \) if necessary, we may assume
  \( u_{p} = \mathbb{C}^{n} \otimes u_{0} \). Repeat the proof of
  Proposition~\ref{prop:f8d2ac97bef564cc} with the small modification of
  replacing the unitary representation \( w \) there with the unitary projective
  representation \( V_{p} \), we see that there exists a unique unitary
  projective representation
  \( V_{1} \colon \Lambda_{0} \rightarrow \mathcal{U}(\mathbb{C}^{n}) \), such
  that \( V_{p} = V_{1} \times V_{0} \). This proves
  \ref{item:61b4524c2f8e930a}.

  By \ref{item:61b4524c2f8e930a} and its proof, we may suppose
  \( u_{p} = \id_{\mathbb{C}^{n}} \otimes u_{0} \). Note that the CSR
  parameterized by \( (u_{p}, V_{p}, v) \) is exactly
  \( (\id_{\mathscr{H}_{v}} \otimes \id_{\mathbb{C}^{n}} \otimes u_{0}, v \times
  V_{p}) \), which coincides exactly with the CSR parameterized by
  \( (\id_{\mathbb{C}^{n} \otimes \mathscr{H}_{v}} \otimes u_{0}, V_{0}, v
  \times V_{1}) \) since \( v \times V_{p} = v \times V_{1} \times V_{0}
  \). This proves \ref{item:9fe46e8db326eca7}.
\end{proof}

\begin{defi}
  \label{defi:36e39769ee182b6c}
  Using the notation of Proposition~\ref{prop:5a82a06194fb75ea}, the
  representation parameter \( (u_{0}, V_{0}, v \times V_{1}) \) is called a
  \textbf{reduction} of the GRP \( (u, V, v) \) along \( (u_{0}, V_{0}) \).
\end{defi}
\begin{rema}
  \label{rema:a02167e7e8d94e72}
  Since \( V_{1} \) is determined up to unitary equivalence, so is the reduction
  \( (u_{0}, V_{0}, v \times V_{1}) \).
\end{rema}

The following result describes the incidence numbers
\( m_{[U_{1}], [U_{2}], [{U_{3}}]}(z_{1}, z_{2}, z_{3}) \) in terms of the
dimension of the intertwiner space of some projective representations of
\( \Lambda_{0} \).

\begin{prop}
  \label{prop:72841057708ae472}
  Suppose we are given the following data for each \( i = 1, 2, 3 \):
  \begin{itemize}
  \item a \( \Lambda_{i} \in \giso(\Lambda) \), a left coset \( z_{i} \) in
    \( \Lambda / \Lambda_{i} \) and a \( r_{i} \in z_{i} \);
  \item a representation parameter \( (u_{i}, V_{i}, v_{i}) \) associated with
    \( \Lambda_{i} \);
  \item the unitary representation \( U_{i} \) of
    \( \mathbb{G} \rtimes \Lambda_{i} \) parameterized by
    \( (u_{i}, V_{i}, v_{i}) \).
  \end{itemize}
  Let
  \( \Lambda_{0} = \cap_{i=1}^{3}r_{i}\Lambda_{i}r_{i}^{-1} =
  \cap_{i=1}^{3}z_{i}\Lambda_{i}z_{i}^{-1} \). Suppose
  \begin{displaymath}
    \bigl(r_{1} \cdot u_{1}, (r_{1} \cdot V_{1}) \vert_{\Lambda_{0}}, (r_{2}
    \cdot v_{2})\vert_{\Lambda_{0}}
    \times (r_{3} \cdot v_{3}) \vert_{\Lambda_{0}}
    \times V\bigr)
  \end{displaymath}
  is the reduction of the GRP
  \begin{displaymath}
    \bigl((r_{2} \cdot u_{2}) \times (r_{3} \cdot u_{3}), (r_{2} \cdot
    V_{2})\vert_{\Lambda_{0}} \times (r_{3} \cdot V_{3})\vert_{\Lambda_{0}},
    (r_{2} \cdot v_{2})\vert_{\Lambda_{0}} \times (r_{3} \cdot v_{3})
    \vert_{\Lambda_{0}}\bigr)
  \end{displaymath}
  along
  \( \bigl(r_{1} \cdot u_{1}, (r_{1} \cdot V_{1})\vert_{\Lambda_{0}}\bigr)
  \). Then the unitary projective representations
  \( (r_{1} \cdot v_{1})\vert_{\Lambda_{0}} \) and
  \begin{displaymath}
    (r_{2} \cdot v_{2})\vert_{\Lambda_{0}} \times (r_{3} \cdot
    v_{3})\vert_{\Lambda_{0}} \times V
  \end{displaymath}
  of \( \Lambda_{0} \) have the same cocycle, and
  \begin{equation}
    \label{eq:3d42fbf9ee798922}
    \begin{split}
      & \leadmathskip m_{[U_{1}], [U_{2}],[U_{3}]}(z_{1}, z_{2}, z_{3}) \\
      &= \dim \morph_{\Lambda_{0}}\bigl((r_{1} \cdot
      v_{1})\vert_{\Lambda_{0}}, %
      (r_{2} \cdot v_{2})\vert_{\Lambda_{0}} \times (r_{3} \cdot
      v_{3})\vert_{\Lambda_{0}} \times V\bigr).
    \end{split}
  \end{equation}
\end{prop}
\begin{proof}
  It is easy to check that
  \( \bigl((r_{2} \cdot u_{2}) \times (r_{3} \cdot u_{3}), (r_{2} \cdot
  V_{2})\vert_{\Lambda_{0}} \times (r_{3} \cdot V_{3})\vert_{\Lambda_{0}},
  (r_{2} \cdot v_{2})\vert_{\Lambda_{0}} \times (r_{3} \cdot v_{3})
  \vert_{\Lambda_{0}}\bigr) \) is indeed a generalized representation
  parameter. Take the minimal central projection \( p \) of
  \( \selfmorph_{\mathbb{G}}\bigl((r_{2} \cdot u_{2}) \times (r_{3} \cdot
  u_{3})\bigr) \) corresponding to the maximal pure subrepresentation
  \( u_{p} \) of \( (r_{2} \cdot u_{2}) \times (r_{3} \cdot u_{3}) \) that is
  supported by \( [r_{1} \cdot u_{1}] \in \irr(\mathbb{G}) \). Suppose
  \( q = 1 - p \). By Lemma~\ref{lemm:627a277420a433bc}, \( q \) also
  corresponds to a subrepresentation \( u_{q} \) of
  \( (r_{2} \cdot u_{2}) \times (r_{3} \cdot u_{3}) \) on
  \( q(\mathscr{H}_{u_{2}} \otimes \mathscr{H}_{u_{3}}) \). Similarly, let
  \( V_{p} \) (resp.\ \( V_{q} \)) be the subrepresentation of the unitary
  projective representation
  \( (r_{2} \cdot v_{2})\vert_{\Lambda_{0}} \times (r_{3} \cdot
  v_{3})\vert_{\Lambda_{0}} \) on
  \( p(\mathscr{H}_{u_{2}} \otimes \mathscr{H}_{u_{3}}) \) (resp.\
  \( q(\mathscr{H}_{u_{2}} \otimes \mathscr{H}_{3}) \)). Let \( U_{p} \) (resp.\
  \( U_{q} \)) be the representation of \( \mathbb{G} \rtimes \Lambda_{0} \)
  parameterized by the GRP
  \( (u_{p}, V_{p}, (r_{2} \cdot v_{2})\vert_{\Lambda_{0}} \times (r_{3} \cdot
  v_{3})\vert_{\Lambda_{0}}) \) (resp.\
  \( (u_{q}, V_{q}, (r_{2} \cdot v_{2})\vert_{\Lambda_{0}} \times (r_{3} \cdot
  v_{3})\vert_{\Lambda_{0}}) \)). By construction, the unitary representation
  \( U \) of \( \mathbb{G} \rtimes \Lambda_{0} \) parameterized by
  \( \bigl((r_{2} \cdot u_{2}) \times (r_{3} \cdot u_{3}), (r_{2} \cdot
  V_{2})\vert_{\Lambda_{0}} \times (r_{3} \cdot V_{3})\vert_{\Lambda_{0}},
  (r_{2} \cdot v_{2})\vert_{\Lambda_{0}} \times (r_{3} \cdot v_{3})
  \vert_{\Lambda_{0}}\bigr) \) is the direct sum of \( U_{p} \) and \( U_{q}
  \). By definition,
  \begin{equation}
    \label{eq:990c0eae1b949b7e}
    \begin{split}
      m_{[U_{1}], [U_{2}], [U_{3}]}(z_{1}, z_{2}, z_{3}) %
      &= \dim\morph_{\mathbb{G} \rtimes \Lambda_{0}}(U_{1}, U) \\
      &= \dim\morph_{\mathbb{G} \rtimes \Lambda_{0}}(U_{1}, U_{p}) +
      \dim\morph_{\mathbb{G} \rtimes \Lambda_{0}}(U_{1}, U_{q}).
    \end{split}
  \end{equation}
  From our construction, the matrix coefficients of \( u_{p} \) and \( u_{q} \)
  are orthogonal with respect to the Haar state \( h \) of \( \mathbb{G} \).
  Thus the proof of Proposition~\ref{prop:7b2a964908c81d83}
  \ref{item:cda18708f8635340} applies almost verbatim, and shows that
  \begin{equation}
    \label{eq:956c193930da137f}
    \dim\morph_{\mathbb{G} \rtimes \Lambda_{0}}(U_{1}, U_{q}) = 0.
  \end{equation}
  On the other hand, the cocycles of both
  \( (r_{1} \cdot v_{1})\vert_{\Lambda_{0}} \) and
  \( (r_{2} \cdot v_{2})\vert_{\Lambda_{0}} \times (r_{3} \cdot
  v_{3})\vert_{\Lambda_{0}} \times V \) are both opposite to that of
  \( (r_{1} \cdot V_{1})\vert_{\Lambda_{0}} \) by the reduction process
  described above, hence these cocycles coincide. By
  Proposition~\ref{prop:7b2a964908c81d83} \ref{item:e80ad530289a518d} and
  Proposition~\ref{prop:5a82a06194fb75ea}~\ref{item:9fe46e8db326eca7}, we have
  \begin{equation}
    \label{eq:20e79b5c884d7b1a}
    \dim\morph_{\mathbb{G} \rtimes \Lambda_{0}}(U_{1}, U_{p})
    = \dim \morph_{\Lambda_{0}}\bigl((r_{1} \cdot v_{1})\vert_{\Lambda_{0}}, %
    (r_{2} \cdot v_{2})\vert_{\Lambda_{0}}
    \times (r_{3} \cdot v_{3})\vert_{\Lambda_{0}} \times V\bigr).
  \end{equation}
  Now \eqref{eq:3d42fbf9ee798922} follows from~\eqref{eq:956c193930da137f} and
  \eqref{eq:20e79b5c884d7b1a}.
\end{proof}

\section{Fusion rules}
\label{sec:20a173060e7c2f82}

We now calculate the fusion rules of \( \mathbb{G} \rtimes \Lambda \). From the
classification theorem (Theorem~\ref{theo:91ac3dad0bd1219d}), up to unitary
equivalence, all unitary irreducible representations of
\( \mathbb{G} \rtimes \Lambda \) are distinguished. Thus the task falls to the
calculation of
\begin{equation}
  \label{eq:7b05fc7c1a4d6f3a}
  \dim \morph_{\mathbb{G}}\bigl(\indrep(U_{1}),
  \indrep(U_{2}) \times \indrep(U_{3})\bigr),
\end{equation}
where, for \( i = 1, 2, 3 \), \( U_{i} \) is the irreducible unitary
representation of \( \mathbb{G} \rtimes \Lambda_{i} \) parameterized (see
Definition~\ref{defi:ab1995a9dda3749d} and
Definition~\ref{defi:cbe3a197ae3ac6c1}) by some \emph{distinguished}
representation parameter \( (u_{i}, V_{i}, v_{i}) \) associated with
\( \Lambda_{i} \) (recall that \( \Lambda_{i} = \Lambda_{[u_{i}]} \) since
\( (u_{i}, V_{i}, v_{i}) \) is distinguished). Let \( h \) be the Haar state on
\( C(\mathbb{G}) = A \). For any subgroup \( \Lambda_{0} \) of \( \Lambda \), we
use \( h^{\Lambda_{0}} \) to denote the Haar state on
\( C(\mathbb{G} \rtimes \Lambda_{0}) = A \otimes C(\Lambda_{0}) \), and
\( E_{\Lambda_{0}} \colon C(\Lambda_{0}) \rightarrow C(\Lambda) \) denotes the
linear embedding such that
\( \delta_{r_{0}} \in C(\Lambda_{0}) \mapsto \delta_{r_{0}} \in C(\Lambda)
\)(the extension of functions in \( C(\Lambda_{0}) \) to functions in
\( C(\Lambda) \) that vanishes outside \( \Lambda_{0} \)). In particular,
\( h^{\Lambda} \) is the Haar state on
\( C(\mathbb{G} \rtimes \Lambda) = A \otimes C(\Lambda) \). For
\( i = 1, 2, 3 \), let
\( \chi_{i} = (\tr \otimes \id)(U_{i}) \in A \otimes C(\Lambda_{i}) \) be the
character of \( U_{i} \), and \( r \cdot \chi_{i} \) is defined to be the
character of the representation \( r \cdot U_{i} \) of
\( \mathbb{G} \rtimes r\Lambda_{i}r^{-1} \).

Using these notations, by Proposition~\ref{prop:56963e422c66f26b}, we have the
following formula for the character of \( \indrep(U_{i}) \),
\begin{equation}
  \label{eq:07e981fc699ac67f}
  \forall i = 1, 2, 3, \quad
  \chi(\indrep(U_{i}))
  = \abs*{\Lambda_{i}}^{-1} \sum_{r_{i} \in \Lambda} (\id \otimes
  E_{r_{i}\Lambda_{i}r_{i}^{-1}})(r_{i} \cdot \chi_{i}).
\end{equation}
Thus
\begin{equation}
  \label{eq:6814153fdb9f279c}
  \begin{split}
    & \leadmathskip \dim\morph_{\mathbb{G} \rtimes \Lambda}
    \bigl(\indrep(U_{1}),
    \indrep(U_{2}) \times \indrep(U_{3})\bigr) \\
    &= h^{\Lambda}\bigl(\overline{\chi(\indrep(U_{1}))}
    [\chi(\indrep(U_{2}))][\chi(\indrep(U_{3}))]\bigr)
    \\
    &= \sum_{r_{1},r_{2},r_{3}} %
    { h^{\Lambda}\bigl( \chi(r_{1}, r_{2}, r_{3}) \bigr) \over
      \abs*{\Lambda_{1}} \cdot \abs*{\Lambda_{2}} \cdot \abs*{\Lambda_{3}} },
  \end{split}
\end{equation}
where
\begin{equation}
  \label{eq:299ff078893d6877}
  \begin{split}
    & \leadmathskip \chi(r_{1}, r_{2}, r_{3}) \\
    &= \overline{(\id \otimes E_{r_{1}\Lambda_{1}r_{1}^{-1}})(r_{1} \cdot
      \chi_{1})} [(\id \otimes E_{r_{2}\Lambda_{2}r_{2}^{-1}})(r_{2} \cdot
    \chi_{2})] [(\id \otimes E_{r_{3}\Lambda_{3}r_{3}^{-1}})(r_{3} \cdot
    \chi_{3})].
  \end{split}
\end{equation}

If \( \Theta, \Xi \) are subgroups of \( \Lambda \) with
\( \Theta \subseteq \Xi \), and \( \sum_{r \in \Xi} a_{r} \otimes \delta_{r} \)
is an arbitrary element of \( A \otimes C(\Xi) \) with all \( a_{r} \in A \), we
call the element \( \sum_{r \in \Theta} a_{r} \otimes \delta_{r} \) of
\( A \otimes C(\Theta) \) the restriction of
\( \sum_{r \in \Xi}a_{r} \otimes \delta_{r} \) and denote it by
\( {\left(\sum_{r \in \Xi} a_{r} \otimes \delta_{r}\right)}\vert_{\mathbb{G}
  \rtimes \Theta} \). Recall that
\begin{equation}
  \label{eq:0c3499016bd1e713}
  h^{\Lambda_{0}} = [\Lambda \colon \Lambda_{0}]
  \cdot  h^{\Lambda} \circ (\id \otimes E_{\Lambda_{0}})
\end{equation}
for any subgroup \( \Lambda_{0} \) of \( \Lambda \), posing
\begin{equation}
  \label{eq:2957551cbac25074}
  \Lambda(r_{1},r_{2},r_{3}) = \bigcap_{i=1}^{3} r_{i} \Lambda_{i}r_{i}^{-1},
\end{equation}
we have
\begin{equation}
  \label{eq:3ab41cdd093f9099}
  \begin{split}
    & \leadmathskip h^{\Lambda}\bigl( \chi(r_{1}, r_{2}, r_{3}) \bigr) =
    h^{\Lambda} \left( \chi(r_{1}, r_{2}, r_{3})
      \vert_{\mathbb{G} \rtimes \Lambda(r_{1}, r_{2}, r_{3})}\right) \\
    &= { %
      h^{\Lambda(r_{1},r_{2},r_{3})} \left(\overline{(r_{1} \cdot
          \chi_{1})\vert_{\mathbb{G} \rtimes \Lambda(r_{1}, r_{2},
            r_{3})}}(r_{2} \cdot \chi_{2})\vert_{\mathbb{G} \rtimes
          \Lambda(r_{1}, r_{2}, r_{3})}(r_{3} \cdot \chi_{3})\vert_{\mathbb{G}
          \rtimes \Lambda(r_{1}, r_{2}, r_{3})}\right) %
      \over %
      [\Lambda \colon \Lambda(r_{1}, r_{2},r_{3})] %
    } \\
    &= {[\Lambda \colon \Lambda(r_{1}, r_{2},r_{3})]}^{-1}
    m_{U_{1},U_{2},U_{3}}(r_{1}, r_{2}, r_{3}),
  \end{split}
\end{equation}
where \( m_{U_{1},U_{2},U_{3}}(r_{1}, r_{2}, r_{3}) \) is the incidence number
of \( (r_{1}, r_{2}, r_{3}) \) relative to \( (U_{1}, U_{2}, U_{3}) \).

By~\eqref{eq:6814153fdb9f279c} and~\eqref{eq:3ab41cdd093f9099}, we have
\begin{equation}
  \label{eq:58bd98ab763a843b}
  \begin{split}
    & \leadmathskip \dim\morph_{\mathbb{G} \rtimes \Lambda}
    \bigl(\indrep(U_{1}), \indrep(U_{2}) \times
    \indrep(U_{3})\bigr) \\
    &= \sum_{r_{1},r_{2},r_{3} \in \Lambda} \frac{m_{U_{1}, U_{2},
        U_{3}}(r_{1},r_{2},r_{3})}{\abs*{\Lambda_{1}} \cdot \abs*{\Lambda_{2}}
      \cdot \abs*{\Lambda_{3}} \cdot [\Lambda \colon \Lambda(r_{1}, r_{2},
      r_{3})]}.
  \end{split}
\end{equation}
As we've seen in Definition~\ref{defi:b0695ae26a3edb1d} and the discussion
before it, we have
\begin{equation}
  \label{eq:d1853f2a7698f76d}
  \begin{split}
    & (\forall i = 1,2,3, r_{i} \in z_{i} \in \Lambda / \Lambda_{i}) \\
    &\implies %
    m_{[U_{1}], [U_{2}], [U_{3}]}(z_{1}, z_{2}, z_{3}) = m_{U_{1}, U_{2},
      U_{3}}(r_{1}, r_{2}, r_{3}),
  \end{split}
\end{equation}
where
\( \Lambda(z_{1}, z_{2}, z_{3}) := \cap_{i=1}^{3}r_{i}\Lambda_{i}r_{i}^{-1} \)
does not depend on the choices for \( r_{i} \in z_{i} \), \( i = 1, 2, 3 \).
Thus \eqref{eq:58bd98ab763a843b} can be written more succinctly as
\begin{equation}
  \label{eq:8f942d15000dbdd5}
  \begin{split}
    & \leadmathskip \dim \morph_{\mathbb{G} \rtimes
      \Lambda}\bigl(\indrep(U_{1}), \indrep(U_{2}) \times
    \indrep(U_{3})\bigr) \\
    &= \sum_{z_{1} \in \Lambda / \Lambda_{1}} \sum_{z_{2} \in \Lambda /
      \Lambda_{2}} \sum_{z_{3} \in \Lambda / \Lambda_{3}} \frac{m_{[U_{1}],
        [U_{2}], [U_{3}]}(z_{1}, z_{2}, z_{3})}{[\Lambda \colon \Lambda(z_{1},
      z_{2}, z_{3})]}.
  \end{split}
\end{equation}

We formalize the above calculation as the following theorem, which describes the
fusion rules of \( \mathbb{G} \rtimes \Lambda \) in terms of the more basic
ingredients of incidence numbers, which in turn is completely determined by the
representation theory of \( \mathbb{G} \), the action of \( \Lambda \) on
\( \irr(\mathbb{G}) \), and various unitary projective representations of some
naturally appeared subgroups in \( \giso(\Lambda) \).

\begin{theo}
  \label{theo:ff9839aa14503b7c}
  The fusion rules for \( \mathbb{G} \rtimes \Lambda \) is given as the
  following. For \( i = 1, 2, 3 \), let \( W_{i} \) be an irreducible
  representation of \( \mathbb{G} \rtimes \Lambda \). Suppose \( U_{i} \) is the
  distinguished representation parameterized by some distinguished
  representation parameter \( (u_{i}, V_{i}, v_{i}) \) associated with some
  isotropy subgroup \( \Lambda_{i} \) of \( \Lambda \), such that \( W_{i} \) is
  equivalent to \( \indrep(U_{i}) \), then
  \begin{equation}
    \label{eq:ce4eba7dca3e5343}
    \begin{split}
      & \leadmathskip \dim \morph_{\mathbb{G} \rtimes \Lambda}(W_{1}, W_{2}
      \times W_{3}) \\
      &= \sum_{z_{1} \in \Lambda / \Lambda_{1}} \sum_{z_{2} \in \Lambda /
        \Lambda_{2}} \sum_{z_{3} \in \Lambda / \Lambda_{3}} \frac{m_{[U_{1}],
          [U_{2}], [U_{3}]}(z_{1}, z_{2}, z_{3})}{[\Lambda \colon \Lambda(z_{1},
        z_{2}, z_{3})]}.
    \end{split}
  \end{equation}
  Here the incidence numbers
  \begin{equation}
    \label{eq:a9bcde9ea36d8909}
    \begin{split}
      &\leadmathskip m_{[U_{1}], [U_{2}], [U_{3}]}(z_{1}, z_{2}, z_{3}) %
      = m_{U_{1}, U_{2}, U_{3}}(r_{1}, r_{2}, r_{3}) \\
      &= \dim \morph_{\Lambda(z_{1}, z_{2}, z_{3})}\bigl((r_{1} \cdot
      v_{1})\vert_{\Lambda_{0}}, %
      (r_{2} \cdot v_{2})\vert_{\Lambda_{0}} \times (r_{3} \cdot v_{3})\vert_{\Lambda_{0}} \times V\bigr),
    \end{split}
  \end{equation}
  where \( r_{i} \in z_{i} \) for \( i = 1, 2, 3 \), and the unitary projective
  representation \( V \) of \( \Lambda(z_{1}, z_{2}, z_{3}) \) is taken from the
  reduction
  \begin{displaymath}
    \bigl(r_{1} \cdot u_{1}, (r_{1} \cdot V_{1}) \vert_{\Lambda_{0}}, (r_{2}
    \cdot v_{2})\vert_{\Lambda_{0}}
    \times (r_{3} \cdot v_{3}) \vert_{\Lambda_{0}}
    \times V\bigr)
  \end{displaymath}
  of the generalized representation parameter
  \begin{displaymath}
    \bigl((r_{2} \cdot u_{2}) \times (r_{3} \cdot u_{3}), (r_{2} \cdot
    V_{2})\vert_{\Lambda_{0}} \times (r_{3} \cdot V_{3})\vert_{\Lambda_{0}},
    (r_{2} \cdot v_{2})\vert_{\Lambda_{0}} \times (r_{3} \cdot v_{3})
    \vert_{\Lambda_{0}}\bigr)
  \end{displaymath}
  along \( \bigl(r_{1} \cdot u_{1}, (r_{1} \cdot V_{1})\vert_{\Lambda_{0}}\bigr) \).
\end{theo}
\begin{proof}
  The above calculation proves \eqref{eq:ce4eba7dca3e5343}, and
  \eqref{eq:a9bcde9ea36d8909} follows from
  Proposition~\ref{prop:72841057708ae472}.
\end{proof}


\end{document}